\documentclass[graybox]{svmult}

\usepackage{mathptmx}       
\usepackage{helvet}         
\usepackage{courier}        
\usepackage{type1cm}        
\usepackage{makeidx}         
\usepackage{graphicx}        
\usepackage{multicol}        
\usepackage[bottom]{footmisc}
\usepackage{hyperref}
\usepackage{amsmath}
\usepackage{amssymb}
\usepackage{makecell}
\usepackage{amscd}
\usepackage{indentfirst}
\usepackage{nicefrac}
\usepackage{appendix}

\usepackage{algorithm}\usepackage{algpseudocode}

\usepackage[colorinlistoftodos,bordercolor=orange,backgroundcolor=orange!20,linecolor=orange,textsize=scriptsize]{todonotes}

\newcommand{\argmin}{\mathop{\arg\!\min}}
\newcommand{\argmax}{\mathop{\arg\!\max}}
\def \RR {\mathbb R}
\def \R {\mathbb R}
\def\eqdef{\overset{\text{def}}{=}}

\def\Z{{\Bbb Z}}

\def\1{{\bf 1}}

\def\x{\mathbf{x}}
\def\EE{\mathbb E}

\def\R{\mathbb{R}}

\def\R{\mathbb R}
\def\E{\mathbb E}
\def\EE{\mathbb E}
\def\PP{\mathbb P}

\def\e{\varepsilon}
\def\la{\langle}
\def\ra{\rangle}
\def\vp{\varphi}
\def\y{\mathbf{y}}

\def\x{\mathbf{x}}
\def\one{{\mathbf 1}}

\def\ar#1{{\color{black}#1}} 
\def\ab#1{{\color{black}#1}} 
\def\arrev#1{{\color{black}#1}} 

\usepackage{geometry}
\def \PP {\mathbb P}
\def \EE {\mathbb E}

\newcommand{\mI}{{I}}

\newcommand{\mS}{{S}}

\newcommand{\mU}{{U}}
\newcommand{\mV}{{V}}
\newcommand{\mW}{{\overline W}}
\newcommand{\mX}{{X}}
\newcommand{\mY}{{Y}}
\newcommand{\mZ}{{Z}}
\newcommand{\mix}{{\mathbf{M}}}

\newcommand{\cset}{\mathbf{C}}
\newcommand{\ol}{\overline}

\newcommand{\ds}{\displaystyle}
\newcommand{\norm}[1]{\left\| #1 \right\|}
\newcommand{\angles}[1]{\left\langle #1 \right\rangle}
\newcommand{\cbraces}[1]{\left( #1 \right)}
\newcommand{\sbraces}[1]{\left[ #1 \right]}
\newcommand{\braces}[1]{\left\{ #1 \right\}}

\newcommand{\muav}{\mu_{g}}
\newcommand{\Lav}{L_{g}}
\newcommand{\kappaav}{\kappa_g}
\newcommand{\Lmax}{L_{l}}
\newcommand{\mumin}{\mu_{l}}
\newcommand{\kappamax}{\kappa_l}
\newcommand{\eps}{\varepsilon}

\newtheorem{assumption}[theorem]{Assumption}
\newcommand{\numberthis}{\addtocounter{equation}{1}\tag{\theequation}}

\def\e{\varepsilon}

\def \R {\mathbb R}
\def\RR{\mathcal R}

\def\Rbf{\mathbf{R}}
\def\tL{\tilde{L}}
\def\tg{\tilde{g}}
\def\tnabla{\tilde{\nabla}}
\def\tx{\tilde{x}}
\def\ty{\tilde{y}}

\def\Bxi{\boldsymbol{\xi}}

\def\Bxi{\boldsymbol{\xi}}

\def\bld{\boldsymbol}

\newcommand{\cG}{{\cal G}}

\newcommand{\cW}{{\cal W}}

\usepackage{xcolor}

\def\dm#1{{\color{black}#1}} 

\makeindex             

\begin{document}

\title*{Recent theoretical advances in decentralized distributed convex optimization}

\titlerunning{Recent theoretical advances in decentralized distributed convex optimization}



\author{ Eduard Gorbunov$^{1,2,3}$, Alexander Rogozin$^{1,2,3}$, Aleksandr Beznosikov$^{1,2}$, Darina Dvinskikh$^{4,1,5}$, Alexander Gasnikov$^{1,5,6}$}
\authorrunning{ E.~Gorbunov, A.~Rogozin, A.~Beznosikov, D.~Dvinskikh, A.~Gasnikov }
\institute{
$^1$Moscow Institute of Physics and Technology, Russia\\
$^2$Sirius University of Science and Technology, Russia\\
$^3$Russian Presidential Academy of National Economy and Public Administration, Moscow, Russia\\ 
$^4$Weierstrass Institute for Applied Analysis and Stochastics, Germany\\
$^5$Institute for Information Transmission Problems RAS, Moscow, Russia\\
$^6$Caucasus Mathematical Center, Adyghe State University, Russia
}

\maketitle

\abstract{
In the last few years, the theory of decentralized distributed convex optimization has made significant progress. The lower bounds on communications rounds and oracle calls have appeared, as well as methods that reach both of these bounds. 
In this paper, we focus on how these results can be explained based on optimal algorithms for the non-distributed setup.
In particular, we provide our recent results that have not been published yet and that could be found in details only in arXiv preprints.
}



\section{Introduction}

In this work, we focus on the following convex  optimization problem 

  \begin{equation}\label{P}
      \min_{x\in Q \subseteq \mathbb{R}^n}  f(x):= \frac{1}{m}\sum_{i=1}^m f_i(x),
  \end{equation}
 where the functions $\{f_i\}_{i=1}^m$ are convex and $Q$ is a convex set.
 Such kind of problems arise in many machine learning applications \cite{shalev2014understanding} (e.g., empirical risk minimization) and statistical applications \cite{spokoiny2012parametric} (e.g., maximum likelihood estimation). To solve these problems, decentralized distributed methods are widely used (see \cite{nedic2020distributed,dvinskikh2020parallel} and reference therein).
This direction has gained popularity with the release of the book \cite{bertsekas1989parallel}. Many researchers (among which we especially note Angelia Nedich) have productively promoted distributed algorithms in the last 30 years.
Due to the emergence of big data and the rapid growth of problem sizes, decentralized distributed methods have gained increased interest in the last decade.
In this paper, we mainly focus on the last five years of theoretical advances, starting with the remarkable paper \cite{arjevani2015communication}. The authors of \cite{arjevani2015communication} introduce the lower complexity bounds for communication rounds required to achieve $\varepsilon$-accuracy solution $x^N$ of \eqref{P} in the function value, i.e., $f(x^N) - \min_{x\in Q}f(x) \le \varepsilon$. 

Let us formulate the result of \cite{arjevani2015communication} (see also  \cite{shi2015extra,scaman2017optimal,scaman2018optimal,lan2017communication,uribe2020dual}) formally. Assume that we have some connected undirected graph (network) with $m$ nodes. 
For each node $i$ of this graph, we privately assign function $f_i$ and suppose that the node $i$ can calculate $\nabla f_i$ at some point $x$. 
At each communication round the nodes can communicate with their  neighbors, i.e., send and receive a message with no more than $O(n)$ numbers. 
In the $O(R)$ neighborhood of a solution $x^*$ of \eqref{P} (where $R = \|x^0 - x^*\|_2$ is the Euclidean distance between starting point $x^0$ and the solution $x^*$ that corresponds to the minimum of this norm), we suppose that functions $f_i$'s are $M$-Lipschitz continuous (i.e., $\|\nabla f_i(x)\|_2 \le M$) and   $L$-Lipschitz smooth (i.e., $\|\nabla f_i(y) - \nabla f_i(x)\|_2 \le L \|y-x\|_2$).
The optimal bounds on the number of communications and the number of oracle calls per node are summarized in Table \ref{tab:DetPrimeOr}.
Here and below $\tilde{O}(\text{ })$ means the same as $O(\text{ })$ up to a $\log (1/\varepsilon)$ factor, and $\tilde{O}(\sqrt\chi)$ corresponds to the consensus time, that is the number of communication rounds required to reach the consensus in the considered network (more accurate definition of $\tilde{O}(\sqrt\chi)$ is given in Sections \ref{Rogozin}, \ref{gorbunov}). 

In the last few years, algorithms have been developed that reach the lower bounds from Table~\ref{tab:DetPrimeOr}.
In Section \ref{Rogozin}, we consider one of such algorithms \cite{rogozin2019projected,rogozin2021towards} for the case when the functions $f_i$'s are smooth.
This algorithm has the simplest nature among all known alternatives: this is a direct consensus-projection generalization of Nesterov's fast gradient method. 

When communication networks vary from time to time (time-varying communication networks, see Section~\ref{Rogozin}), we replace $\sqrt{\chi}$ by $\chi$ and we suppose that different $f_i$'s may have different  constants of smoothness $L_i$. 

The non-smooth case (when functions $f_i$'s are Lipschitz continuous) is studied in Section~\ref{gorbunov}, where  the results from \cite{dvinskikh2019decentralized,gorbunov2019optimal} are summarized. 
The approach is based on reformulation of the distributed decentralized problem as non-distributed convex optimization problem with affine constraints, which are further brought into the target function as a composite quadratic penalty. To solve this problem, Lan's sliding algorithm \cite{lan2020first} can be used.  
	\begin{table}[H]
\caption {Optimal bounds for communication rounds and deterministic oracle calls of  $\nabla f_i$ per node}
\label{tab:DetPrimeOr}
\begin{center}
\begin{tabular}{|c| c| c| c| c|}
 \hline
 & \makecell{ $f_i$ is  $\mu$-strongly convex\\ and $L$-smooth} 
 &   \makecell{$f_i$ is $L$-smooth} 
 & \makecell{ $f_i$ is $\mu$-strongly convex} &  \\
 \hline
 \makecell{\# communication rounds} 
 & \makecell{$\widetilde O\left(\sqrt{\frac{L}{\mu}\chi} \right)$}   
 & \makecell{ $ \widetilde O\left({\sqrt{\frac{LR^2}{\e} \chi}}\right)$ }
 &  \makecell{  $O\left(\sqrt{\frac{M^2}{\mu\e}\chi} \right)$} 
 &  \makecell{  $O\left( \sqrt{\frac{M^2R^2}{\e^2}\chi} \right) $}   \\
 \hline
\makecell{\# oracle calls of $\nabla f_i$\\ per node $i$} 
& $ \widetilde O\left(\sqrt{\frac{L}{\mu}} \right)$ 
&  $O\left(\sqrt{\frac{LR^2}{\e}}  \right) $ 
&   $O\left(\frac{M^2}{\mu\e} \right) $ 
&   $O\left( \frac{M^2R^2}{\e^2}\right)$\\
 \hline
\end{tabular}
\end{center}
\end{table}

The same construction and estimates hold  (see Table~\ref{T:stoch_prima_oracle}) in the non-smooth stochastic case, when instead of subgradients $\nabla f_i(x)$'s we have an access only to their unbiased estimates $\nabla f_i(x,\xi_i)$'s.  We assume here that $\E\|\nabla f(x,\xi)\|_2^2\le M^2$ on a $O(R)$ neighborhood of $x^*$.


	\begin{table}[H]
\caption {Optimal bounds for communication rounds and stochastic oracle calls of $\nabla f_i(x,\xi_i)$ per node} 
\label{T:stoch_prima_oracle}
\begin{center}
\begin{tabular}{ |c| c| c| c| c|}
 \hline
&{\makecell{ $f_i$ is 
$\mu${-strongly convex}\\ and $L$-smooth} } &  {\makecell{ $f_i$ is  $L$-smooth} }&  {\makecell{ $f_i$ is 
$\mu${-strongly convex}} } &  \\
 \hline
\makecell{\# communication rounds}& \makecell{ $\widetilde O\left(\sqrt{\frac{L}{\mu}\chi} \right)$}   & \makecell{ $ \widetilde O\left({\sqrt{\frac{LR^2}{\e} \chi}}\right)$ }&  \makecell{  $O\left(\sqrt{\frac{M^2}{\mu\e}\chi} \right)$} &  \makecell{  $O\left( \sqrt{\frac{M^2R^2}{\e^2}\chi} \right) $}  \\
 \hline
   \makecell{\# oracle calls of $\nabla f_i(x, \xi_i)$\\
 per node $i$} & {\makecell{ {$\widetilde O\left(\max\left\{\frac{\sigma^2}{m\mu\e},
 \sqrt{\frac{L}{\mu}}\right\}\right)$} }}  & {\makecell{{$ O\left(\max\left\{\frac{\sigma^2R^2}{m\e^2}, \sqrt{\frac{LR^2}{\e}}  \right\}\right)$}}} & { {\makecell{  $O\left(\frac{M^2}{\mu\e} \right)$}}} &{\makecell{  $O\left(\frac{M^2R^2}{\e^2}\right)$}} \\
 \hline
\end{tabular}
\end{center}
\end{table}

The smooth part of Table~\ref{T:stoch_prima_oracle} describes the known lower bounds. There exist methods that are optimal only in one of the two mentioned criteria \cite{dvinskikh2019decentralized}: either in communication rounds or in oracle calls per node. \ar{The technique from \cite{rogozin2021towards} (also described in Section~\ref{Rogozin}) combined with proper batch-size policy \cite{dvinskikh2020accelerated} allows to reach these lower bounds up to a logarithmic factor \cite{rogozin2021accelerated}.}


Section~\ref{gorbunov} also contains analogues of the results mentioned in Tables~\ref{tab:DetPrimeOr} and \ref{T:stoch_prima_oracle} for dual (stochastic) gradient type oracle. That is, instead of an access at each node to $\nabla f_i$ we have an access to the gradient of conjugated function $\nabla f_i^*$ \cite{scaman2017optimal,uribe2020dual}. Such oracle appears in different applications, in particular, in Wasserstein barycenter problem \cite{uribe2018distributed,dvurechenskii2018decentralize,kroshnin2019complexity,dvinskikh2019dual,dvinskikh2021decentralized}.

    In Section~\ref{beznosikov}, we transfer the results mentioned above to gradient-free oracle assuming that we have an access only to  $f_i$ instead of $\nabla f_i$. 
In this case, a trivial solution comes to mind: to restore the gradient from finite differences. Based on optimal gradient-type methods in smooth case, it is possible to build optimal gradient-free methods. But what is about non-smooth case? To the best of our knowledge, until recently, it was an open question. Based on \cite{beznosikov2019derivative},  we provide an answer for this question (Section~\ref{beznosikov}). To say more precisely, we transfer optimal gradient-free algorithms for non-smooth (stochastic two-points) convex optimization problems \cite{shamir2017optimal,bayandina2018mirror} from non-distributed set up to decentralized distributed one.  Here, as in  Section~\ref{gorbunov}, we also mainly use the penalty trick and the Lan's sliding.      

It is worth to add several results to the list of recent advances collected in Tables~\ref{tab:DetPrimeOr},~\ref{T:stoch_prima_oracle}. The first result describes the case when $f_i(x) =  \frac{1}{r}\sum_{j=1}^r f_i^j(x)$ in \eqref{P}, $Q = \R^n$. All $f_i^j$ are $L$-smooth and $\mu$-strongly convex. In this case, the lower bounds were obtained in \cite{hendrikx2020optimal}. Optimal algorithms were proposed in \cite{hendrikx2020optimal,li2020optimal}. These algorithms require $\tilde{O}\left(\sqrt{ \frac{L}{\mu}\chi}\right)$ communication rounds and $\tilde{O}\left(r+\sqrt{r \frac{L}{\mu}}\right)$ oracle calls ($\nabla f_i^j$ calculations) per node. This is valuable result since in real machine learning applications the sum-type representation of $f_i$ is typical. 

Another way to use this representation is statistical similarity of $f_i$. The lower bound for communication rounds in deterministic smooth case was also obtained \cite{arjevani2015communication}. Roughly speaking, if the Hessians of the $f_i$'s are $\beta$-close  in the 2-norm, then the lower bound for communication rounds will be \dm{ $\tilde{\Omega}\left(\sqrt{ \frac{\beta}{\mu}\chi}\right)$.
Here $\tilde{\Omega} (\cdot)$  is the notation for lower bounds on the growth rate hiding logarithms.}
For example, if $f_i(x) =  \frac{1}{r}\sum_{j=1}^r f_i^j(x)$ and all $f_i^j$ are $\mu$-strongly convex we have $\beta \simeq \mu + \frac{\text{const}}{\sqrt{r}}$.
In the decentralized distributed setup there is a gap between this lower bound and the optimal (non-accelerated) bound $\tilde{O}\left( \frac{\beta}{\mu}\chi\right)$ that can be achieved at the moment \cite{sun2020convergence}. But for centralized distributed architectures with additional assumptions on $f_i$'s, partial acceleration is possible \cite{hendrikx2020statistically}. \arrev{In the recent paper \cite{tian2021acceleration} the gap for decentralized optimization was closed by using distributed Catalyst.}

Other group of results  relate to very specific (but rather popular) centralized federated learning architectures \cite{kairouz2019advances}. According to  mentioned above estimates, heterogeneous federated learning can be considered as a partial case (with $\chi = 1$) \cite{karimireddy2019scaffold,woodworth2020minibatch,koloskova2020unified,gorbunov2020local}. In the paper \cite{koloskova2020unified}, this was explained based on the analysis of unified decentralized SGD. Paper \cite{koloskova2020unified} also summarizes a lot of different distributed setups in one general approach. We partially try to use the generality from \cite{koloskova2020unified} in Section~\ref{Rogozin}. To the best of our knowledge, it is an open question to accelerate all the results of \cite{koloskova2020unified}. Section~\ref{Rogozin} contains such an acceleration only in deterministic case. 

\ar{
Along with minimization, distributed saddle-point problems are an interesting venue of research \cite{mateos2015distributed,wai2018multi,rogozin2021decentralized,liudecentralized}. The basis for distributed solution of min-max problems are extragradient and Mirror-Prox methods \cite{nemirovski2004prox}. Unlike minimization, where the optimal dependence of iteration complexity on function condition number is $\sqrt\kappa$, the lower complexity bound for saddle-point algorithms includes $\kappa$, and Nesterov acceleration does not improve classical non-accelerated methods. In decentralized case, the lower bound for number of communications is $O(\kappa\sqrt\chi\log(1/\eps))$ \cite{rogozin2021decentralized}. Therefore, Nesterov acceleration technique is not needed for obtaining optimal methods for min-max problems both in classical and distributed optimization. \arrev{But in particular cases (i.e. different constants of strong convexity and strong concavity) acceleration is possible due to distributed Catalyst \cite{tian2021acceleration} and \cite{lin2020near,gasnikov2021accelerated,yang2020catalyst,vladislav2021accelerated}.}}
\arrev{
 Note, also that lower bound and optimal decentralized algorithm for saddle-point problems with variance reduction was proposed in \cite{beznosikov2021optimal_} (this paper develops the results of \cite{hendrikx2020optimal,li2020optimal}). Lower bound optimal decentralized algorithm for saddle-point problems with similarity was proposed in \cite{beznosikov2021distributed}.}

\section{Decentralized Optimization of Smooth Convex Functions} \label{Rogozin}

Consider problem \eqref{P} and rewrite it in the following form:
\begin{align}\label{eq:problem_equality_constraints}
    \min_{\mX\in\R^{m\times n}}~ F(\mX) = \sum_{i=1}^m f_i(x_i)~ \text{ s.t. }~ x_1 = \ldots = x_m
\end{align}
where $\mX = (x_1 \ldots x_m)^\top$. Decentralized optimization problem  is now reformulated as an optimization problem with linear constraints. The constraint set writes as $\cset = \braces{x_1 = \ldots = x_m}$.

Functions $f_i$ are stored on the nodes across the network, which is represented as an undirected graph $\mathcal{G} = (V, E)$. Every node has an access to the function and its first-order characteristics. Decentralized first-order methods use two types of steps -- computational steps, i.e. performing local computations and communication steps, which is exchanging the information with neighbors. Alternating these two types of steps results in minimizing the objective while maintaining agents' vectors approximately equal.

We begin with an overview of how communication procedures are developed and analyzed. The iterative information exchange is referred to as \textit{consensus} \arrev{or} \textit{gossip} algorithms \arrev{in the literature \cite{boyd2006randomized,tsitsiklis1984problems,xiao2004fast,muthukrishnan1998first}}.

\subsection{Consensus algorithms}\label{sec:consensus_algorithms}

Let each agent in the network initially hold a vector $x_i^0$ and let the communication network be represented by a connected graph $\mathcal{G} = (V, E)$. The agents seek to find the average vector across the network, but their communication is restricted to sending and receiving information from their direct neighbors. In one communication round, every two nodes linked by an edge exchange their vector values. After that, agent $i$ sums the received values with predefined coefficients $m_{ij}$, where $j$ is the number of the corresponding neighbor. In other words, every node runs an update
\begin{align*}
    x_i^{k+1} = \ar{[\mix]}_{ii} x_i^k + \sum_{(i, j)\in E} \ar{[\mix]}_{ij} x_j^k,
\end{align*}
\ar{where $[\mix]_{ij}$ are elements of the \textit{mixing matrix} $\mix$. The} update at one communication step takes the form
\begin{align}\label{eq:consensus_iter}
    \mX^{k+1} = \mix \mX^k.
\end{align}
Under additional assumptions this iterative scheme converges to the average of initial vectors over network, i.e. to 
\begin{align*}
    \ol\mX^0 = \frac{1}{m} \one \one^\top \mX^0 = \frac{1}{m} \sum_{i=1}^m x_i^0.
\end{align*}
\begin{assumption}\label{assum:mixing_matrix}
Mixing matrix $\mix$ satisfies the following properties.
\begin{itemize}
		\item (Decentralized property) If $(i, j)\notin E$ and $i\ne j$, then $[\mix]_{ij} = 0$. Otherwise $[\mix]_{ij} > 0$.
		\item (Symmetry and double stochasticity) $\mix \one = \one$ and $\mix = \mix^\top$.
        \item (Spectrum property) Denote $\lambda_2(\mix)$ the absolute value of second largest (in absolute value) eigenvalue of $\mix$. Then $\lambda_2(\mix) < 1$.
	\end{itemize}
\end{assumption}
The choice of weights for mixing matrix is an interesting problem we do not address here (see \cite{boyd2006randomized} for details). A mixing matrix with Metropolis weights satisfies Assumption \ref{assum:mixing_matrix}:
\begin{align*}
	[\mix]_{ij} = 
	\begin{cases}
		1 / (1 + \max\{d_i, d_j\}) &\text{if }(i, j)\in E, \\
		0 &\text{if } (i, j)\notin E, \\
		1 - \ds\sum_{m: (i, m)\in E} [\mix]_{im} &\text{if } i = j,
	\end{cases}
\end{align*}
where $d_i$ denotes the degree of node $i$.

Several variations of Assumption \ref{assum:mixing_matrix} can be found in literature. In particular, in \cite{liu2011accelerated} the mixing matrix is not needed to be symmetric. Instead, it is assumed to be doubly stochastic and have a real spectrum. Moreover, the spectrum property in Assumption \ref{assum:mixing_matrix} implies that $\one$ is the only (up to a scaling factor) eigenvector corresponding to eigenvalue $1$, i.e. $\ker\cbraces{\mI - \mix} = \text{span}(\one)$.

\begin{lemma}\label{lemma:consensus}
    For iterative consensus procedure \eqref{eq:consensus_iter} it holds
    \begin{align*}
        \norm{\mX^k - \ol\mX^0}_2 \le (\lambda_2(\mix))^k \norm{\mX^0 - \ol\mX^0}_2.
    \end{align*}
\end{lemma}
\begin{proof}
    Let $x\in \R^{n}$ and $\ol x = \frac{1}{m}\one \one^\top x$. First, note that $\mix\ol x = \mix\cdot \frac{1}{m} \one \one^\top x = \frac{1}{m} \one\one^\top x = \ol x$. It can be easily seen that $x - \ol x \in (\text{span}(\one))^\bot$ and $\mix x - x\in (\text{span}(\one))^\top$:
    \begin{align*}
        \angles{x - \ol x, \one} &= \angles{\cbraces{\mI - \frac{1}{m} \one \one^\top} x, \one} = \angles{x, \cbraces{\mI - \frac{1}{m} \one \one^\top} \one} = 0, \\
        \angles{\mix x - \ol x, \one} &= \angles{\cbraces{\mix - \frac{1}{m} \one \one^\top} x, \one} = \angles{x, \cbraces{\mix - \frac{1}{m} \one \one^\top} \one} = 0.
    \end{align*}
    On subspace $(\text{span}(\one))^\top$ the largest eigenvalue of $\mix$ is $\lambda_2(\mix)$. We have
    \begin{align*}
        \norm{\mix x - \ol x}_2 = \norm{\mix(x - \ol x)}_2\le \lambda_2(\mix)\norm{x - \ol x}_2.
    \end{align*}
    Applying the derived fact to every column of \arrev{$X^k$} we get $\ol\mX^k = \ol\mX^0$ and $\norm{\mX^{k+1} - \ol\mX^0}_2 \le \lambda_2(\mix) \norm{\mX^k - \ol\mX^0}_2$ for every $k\ge 0$, which concludes the proof.
\end{proof}

By Lemma \ref{lemma:consensus}, consensus scheme \eqref{eq:consensus_iter} requires $O\cbraces{\frac{1}{1 - \lambda_2(\mix)}\log(\frac{1}{\eps})}$ iterations to achieve accuracy $\eps$, i.e. to find arithmetic mean of vectors over the network with precision $\eps$: $\norm{\mX^k - \ol\mX^0}_2\le \eps$.

\arrev{
\textbf{Remark}. If the graph changes with time, we associate a sequence of mixing matrices $\braces{\mix^k}_{k=0}^\infty$ with it. In the time-varying case, the consensus algorithm convergence rate is ruled by worst-case second largest eigenvalue, i.e. $\max\limits_{k\geq 0} \lambda_2(\mix^k)$. Provided that each $\mix^k$ is symmetric, doubly stochastic and satisfies the decentralized property in Assumption \ref{assum:mixing_matrix}, the number of communication rounds to reach consensus accuracy $\eps$ is $O\cbraces{\frac{1}{1 - \max\limits_{k\geq 0} \lambda_2(\mix^k)} \log\frac 1 \eps}$.
}

\subsubsection{Quadratic optimization point of view}\label{subsubsec:quadratic_point_of_view}

For a given undirected graph $\mathcal{G} = (V, E)$ introduce its Laplacian matrix
\begin{equation*}
    [\mW]_{ij} = \begin{cases} 
    -1, &\text{if } (i,j)\in E,\\
    \deg(i), &\text{if } i=j,\\
    0 &\text{otherwise}.
    \end{cases}
\end{equation*}
Laplacian matrix is positive semi-definite and for $\mX = (x_1\ldots x_m)^\top$ it holds $\mW\mX = 0 \Leftrightarrow x_1 = \ldots = x_m$. A more detailed discussion of Laplacian matrix and its applications is provided in Section \ref{sec:distributed_opt}. The consensus problem can be reformulated as
\begin{align}\label{eq:consensus_quadratic_reformulation}
    \min_{\mX\in\R^{m\times n}}~ g(\mX) := \frac{1}{2}\angles{\mX, \mW\mX}.
\end{align}
Any matrix $\mX^*$ with equal rows is a solution of Problem \eqref{eq:consensus_quadratic_reformulation}, and therefore the set of minimizers of Problem \eqref{eq:consensus_quadratic_reformulation} is a linear subspace of form $\mathcal{X}^* = \braces{\one x^\top:~ x\in\R^n}$.
Denote $\lambda_{\max}(\mW)$ and $\lambda_{\min}^+(\mW)$ the largest and the smallest non-zero eigenvalues of $\mW$, respectively. Then $g(\mX)$ is has Lipschitz gradients with constant $\lambda_{\max}(\mW)$ and is strongly convex on $(\ker\mW)^\bot$ with modulus $\lambda_{\min}^+(\mW)$. Let non-accelerated gradient descent be run over function $g(\mX)$
\begin{align}\label{eq:consensus_iter_laplacian}
    \mX^{k+1} = \mX^k - \frac{1}{\lambda_{\max}(\mW)} \mW\mX^k.
\end{align}
First, note that trajectory of method \eqref{eq:consensus_iter_laplacian} stays in $\mX^0 + (\ker\mW)^\bot$. To verify this, consider $\mZ\in\ker\mW$:
\begin{align*}
    \angles{\mW\mX^k, \mZ} &= \angles{\mX^k, \mW\mZ} = 0 \Rightarrow \arrev{\mW\mX^k} \in (\ker\mW)^\bot, \\
    \mX^{N+1} - \mX^0 &= -\frac{1}{\lambda_{\max}(\arrev{\mW})} \sum_{k=0}^N \mW\mX^k \in (\ker\mW)^\bot.
\end{align*}
Method \eqref{eq:consensus_iter_laplacian} converges to some point in $\mathcal{X}^*$. Since its trajectory lies in $\mX^0 + (\ker\mW)^\bot$, the limit point of $\braces{\mX^k}_{k=0}^\infty$ is the projection of $\mX^0$ onto $\ker\mW$, i.e. to $\ol\mX^0$. The algorithm requires $O\cbraces{\frac{\lambda_{\max}(\mW)}{\lambda_{\min}^+(\mW)} \log(\frac{1}{\eps})}$ iterations to reach accuracy $\eps$. 

In order to establish the connection between gradient descent for problem \eqref{eq:consensus_quadratic_reformulation} and consensus algorithm \eqref{eq:consensus_iter}, introduce $\mix = \mI - \frac{\mW}{\lambda_{\max}(\mW)}$. Matrix $\mix$ satisfies Assumption \ref{assum:mixing_matrix}, and update rule \eqref{eq:consensus_iter_laplacian} then rewrites as
\begin{align*}
    \mX^{k+1} = \mX^k - \frac{\mW}{\lambda_{\max}(\mW)} \mX^k  = \cbraces{\mI - \frac{\mW}{\lambda_{\max}(\mW)}} \mX^k = \mix \mX^k.
\end{align*}
Therefore, gradient descent on $g(\mX)$ with constant step-size $\frac{1}{\lambda_{\max}(\mW)}$ is equivalent to non-accelerated consensus algorithm \eqref{eq:consensus_iter}. Moreover, the iteration complexities coincide, since $\lambda_2(\mix) = 1 - \frac{\lambda_{\min}^+(\mW)}{\lambda_{\max}(\mW)}$ and therefore $O\cbraces{\frac{\lambda_{\max}(\mW)}{\lambda_{\min}^+(\mW)} \log(\frac{1}{\eps})} = O\cbraces{\frac{1}{1 - \lambda_2(\mix)} \log(\frac{1}{\eps})}$.

\arrev{
We note that the same gradient descent analogy holds for time-varying networks. Given a sequence of connected undirected graphs $\braces{\cG^k}_{k=0}^\infty$, consider a sequence of corresponding Laplacians $\braces{\mW^k}_{k=0}^\infty$. The consensus algorithm may be interpreted as a gradient descent on a time-varying quadratic function $\braces{g^k(X)}_{k=0}^\infty$, where $g^k(X)$ is defined as
\begin{align}\label{eq:tw_consensus_quadratic_reformulation}
    g^k(X) = \frac{1}{2} \angles{X, \mW^k X}.
\end{align}
All $g^k(X)$ have a common set of minimizers $\braces{\one x^\top:~ x\in\R^d}$. The worst-case Lipschitz constant over time is $\max\limits_{k\geq 0} \lambda_{\max}(\mW^k)$, and consensus iteration writes similar to \eqref{eq:consensus_iter_laplacian} with $\lambda_{\max}(\mW)$ replaced by $\max\limits_{k\geq 0} \lambda_{\max}(\mW^k)$. The convergence guarantees of gradient descent over a time varying function $\braces{g^k(X)}_{k=0}^\infty$ do not break since the Lyapunov function for non-accelerated gradient dynamics is the squared distance to the solution set, i.e. it does not depend on the minimization objective. Therefore, non-accelerated gradient descent decreases the Lyapunov function at each step and is robust to changes in the objective function. The number of communication rounds to reach consensus accuracy $\eps$ is $O\cbraces{\frac{\max\limits_{k\geq 0} \lambda_{\max}(\mW^k)}{\max\limits_{k\geq 0} \lambda_{\min}^+(\mW^k)} \log\frac 1 \eps}$.
}

In order to obtain a better dependence on $\frac{\lambda_{\max}(\mW)}{\lambda_{\min}^+(\mW)}$, Nesterov acceleration \cite{nesterov2004introduction} may be employed. Consider Nesterov accelerated method for strongly convex objectives
\begin{subequations}\label{eq:consensus_iter_laplacian_accelerated}
    \begin{align}
        \beta &= \frac{\sqrt{\lambda_{\max}(\mW)} - \sqrt{\lambda_{\min}^+(\mW)}}{\sqrt{\lambda_{\max}(\mW)} + \sqrt{\lambda_{\min}^+(\mW)}}, \\
        \mY^{k} &= \mX^k + \beta(\mX^k - \mX^{k-1}), \\
        \mX^{k+1} &= \mY^k - \frac{\mW}{\lambda_{\max}(\mW)} \mY^k.
    \end{align}
\end{subequations}
Analogously to non-accelerated scheme \eqref{eq:consensus_iter_laplacian}, the trajectory of accelerated Nesterov method lies in $\mX^0 + (\ker\mW)^\bot$. This can be easily seen by induction:
\begin{align*}
    \mY^k - \mX^0 &= (\mX^k - \mX^0) + \beta((\mX^k - \mX^0) - (\mX^{k-1} - \mX^0)) \in (\ker\mW)^\bot, \\
    \mX^{k+1} - \mX^0 &= \underbrace{(\mY^k - \mY^0)}_{\in(\ker\mW)^\bot} - \frac{1}{\lambda_{\max}(\mW)} \underbrace{\mW\mY^k}_{\in \text{Im}\mW = (\ker\mW)^\bot} \in (\ker\mW)^\bot.
\end{align*}
Therefore, accelerated scheme converges to the projection of $\mX^0$ onto $\ker\mW$, which is $\ol\mX^0$, i.e. a matrix \arrev{which} rows are arithmetic averages of $\mX^0$.

\arrev{
Note that acceleration is not attainable over time-varying graphs. Imagine we run Nesterov gradient method over a time-varying objective $\braces{g^k(X)}_{k=0}^\infty$ defined in \eqref{eq:tw_consensus_quadratic_reformulation}. The 
potential function for accelerated gradient dynamics \cite{bansal2019potential} includes
the objective function. Since the objective function changes, the 
potential function is also time-dependent and this fact breaks the prove of convergence result.
A formal proof of why acceleration is impossible over time-varying networks that stay connected and undirected at each iteration is provided in \cite{kovalev2021lower}.
}

\subsubsection{Chebyshev acceleration}\label{subsubsec:chebyshev}

As shown in Section \ref{subsubsec:quadratic_point_of_view}, the convergence of consensus algorithm depends on the condition number $\chi = \frac{\lambda_{\max}(\mW)}{\lambda_{\min}^+(\mW)}$. \arrev{The factor $\chi$ also appears in upper complexity bounds for decentralized algorithms \cite{scaman2017optimal} and represents the measure of the communication graph connectivity. There exists a technique called \textit{Chebyshev acceleration} that enhances the dependence on $\chi$ by replacing the communication matrix $\mW$ with a Chebyshev polynomial $P_K(\mW)$. The structure of the polynomial ensures that the condition number of $P_K(\mW)$ is $O(1)$ whence its power is $K = \lfloor\sqrt\chi\rfloor$. In other words, multiplication by $P_K(\mW)$ requires $\lfloor\sqrt\chi\rfloor$ communication rounds, but the condition number is reduced from $\chi$ to $O(1)$.} A multiple $P_K(\mW)\arrev{X}$ is computed via an iterative consensus-based process. Introduce $c_2 = \frac{\chi + 1}{\chi - 1},~ a_0 = 1,~ a_1 = c_2,~ c_3 = \frac{2}{\lambda_{\max}(\mW) + \lambda_{\min}^+(\mW)}$, $X^0 = X$, $X^1 = c_2(\mI - c_3\mW)X$, for $t = 1, \ldots, K - 1$ do
\begin{align*}
    a_{t+1} = 2c_2 a_t - a_{t-1},~ X^{t+1} = 2c_2(\mI - c_3\mW)X^t - X^{t-1}
\end{align*}
and return $X^0 - {X^K}/{a_K}$.

Summing up, a consensus algorithm of form \arrev{$X^{k+1} = (I - P_K(\mW) / \lambda_{\max}(P_K(\mW))) X^k$} requires a total of $O(\sqrt\chi\log{1/\eps})$ communications to achieve accuracy $\eps$. \arrev{This complexity is better than $O(\chi\log{1/\eps})$ that corresponds to standard consensus algorithm $X^{k+1} = (I - \mW / \lambda_{\max}(\mW))X^k$. Chebyshev acceleration was used to obtain first optimal decentralized methods in \cite{scaman2017optimal}.}

\subsubsection{\arrev{Summary}}

\arrev{In this section, we covered consensus algorithms over time-static and time-varying undirected graphs.}

\arrev{Firstly, let us cover the time-static networks. Different types of matrices may correspond to the graph: mixing matrix $\mix$ or Laplacian matrix $\mW$. The key difference is that mixing matrix is doubly stochastic, i.e. $\mix\one = \one$ and $\one^\top\mix = \one^\top$, while the Laplacian has a null-space property $\mW\one = 0$. Consensus step based on mixing matrix is just multiplication by $\mix$; therefore, it does not require additional knowledge of mixing matrix spectrum. Concerning Laplacian $\mW$, we can either build a mixing matrix $I - \mW/\lambda_{\max}(\mW)$ or use a quadratic minimization approach (see Section \ref{subsubsec:quadratic_point_of_view}). In both cases, knowledge of $\lambda_{\max}(\mW)$ is required.}

\arrev{Secondly, acceleration techniques are also applicable to consensus iteration schemes, but only in the \textit{time-static} case. We can use Chebyshev acceleration covered in Section \ref{subsubsec:chebyshev} or employ an accelerated Nesterov method to a quadratic problem (Section \ref{subsubsec:quadratic_point_of_view}). In both cases, we improve the iteration complexity from $O\cbraces{\chi\log\cbraces{\frac{1}{\eps}}}$ to $O\cbraces{\sqrt\chi \log\cbraces{\frac{1}{\eps}}}$. However, acceleration is not attainable in the time-varying case. When the graph changes, we can only run non-accelerated consensus.}

\subsection{Main assumptions on objective functions}

In this section, we introduce basic assumptions on the functions locally held by computational entities in the network.
\begin{assumption}\label{assum:convex_smooth}
	For every $i = 1, \ldots, m$, function $f_i$ is differentiable, convex and $L_i$-smooth ($L_i > 0$).
\end{assumption}
\begin{assumption}\label{assum:str_convex}
    For every $i = 1,\ldots, m$, function $f_i$ is $\mu_i$-strongly convex ($\mu_i > 0$).
\end{assumption}
Under Assumptions \ref{assum:convex_smooth} and \ref{assum:str_convex} for any $x_i, y_i\in\R^{m}$ for $i = 1, \ldots, \ar{m}$ it holds
\begin{align*}
    \frac{\mu_i}{2}\norm{y_i - x_i}_2^2 \le f_i(y_i) - f_i(x_i) - \angles{\nabla f(x_i), y_i - x_i} \le \frac{L_i}{2}\norm{y_i - x_i}_2^2.
\end{align*}
Summing the above inequality on $i = 1, \ldots, \ar{m}$ we obtain
\begin{align*}
    \frac{\min_i \mu_i}{2}\norm{\mY - \mX}_2^2 \le \sum_{i=1}^m\frac{\mu_i}{2}\norm{y_i - x_i}_2^2 \le F(\mY) - F(\mX) - \angles{\nabla F(\mX), \mY - \mX} \le \sum_{i=1}^m\frac{L_i}{2}\norm{y_i - x_i}_2^2 \le \frac{\max_i L_i}{2}\norm{\mY - \mX}_2^2.
\end{align*}
On the other hand, given that $\mX, \mY\in\cset$, i.e. $x_1 = \ldots = x_m,~ y_1 = \ldots = y_m$, we have
\begin{align*}
    \frac{1}{2m}\sum_{i=1}^m \mu_i \norm{\mY - \mX}_2^2 \le F(\mY) - F(\mX) - \angles{\nabla F(\mX), \mY - \mX} \le \frac{1}{2m}\sum_{i=1}^m L_i\norm{\mY - \mX}_2^2.
\end{align*}
Therefore, $F(\mX)$ has different strong convexity and smoothness constants on $\R^{m\times n}$ and $\cset$. Following the definitions in \cite{scaman2017optimal}, we introduce
\begin{itemize}
	\item (local constants) $F(X)$ is $\mumin$-strongly convex and $\Lmax$-smooth on $\R^{m\times d}$, where $\ds \mumin = \min_i\mu_i,~ \Lmax = \max_i L_i$.
	\item (global constants) $F(X)$ is $\muav$-strongly convex and $\Lav$-smooth on $\cset$, where $\muav = \frac{1}{m}\sum_{i=1}^m\mu_i,~ \Lav = \frac{1}{m}\sum_{i=1}^m L_i$.
\end{itemize}
Note that local smoothness and convexity constants may be significantly worse then global, i.e. $\Lmax \gg \Lav,~ \mumin \ll \muav$ (see \cite{scaman2017optimal} for details). We denote
\begin{align}\label{eq:kappa_max_av}
    \kappamax = \frac{\Lmax}{\mumin},~ \kappaav = \frac{\Lav}{\muav}.
\end{align}
the local and global condition numbers, respectively. 
\arrev{
A trick proposed in \cite{scaman2017optimal} allows to improve the local condition number by slightly changing functions $f_i$. Namely, introduce $\widehat f_i(x) = f_i(x) - \frac{\mu_i - \muav}{2}\norm{x}_2^2$ instead of $f_i$. Then the local condition number writes as
\begin{align*}
    \widehat\kappamax = \frac{\max_k(L_k - \mu_k)}{\muav} + 1.
\end{align*}
}

\subsection{Distributed gradient descent}

Distributed gradient methods alternate taking optimization updates and information exchange steps. One (synchronized) communication round can be represented as a multiplication by a mixing matrix compatible with the graph topology. One of the first distributed gradient dynamics studied in the literature \cite{nedic2009distributed,yuan2016convergence} uses a time-static mixing matrix and writes as
\begin{align*}
    x_i^{k+1} = \sum_{j=1}^m [\mix]_{ij} x_j^k - \alpha^k \nabla f_i(x_i^k).
\end{align*}
Using the notion of $\mX = (x_1\ldots x_m)^\top$ the above update rule takes the form
\begin{align}\label{eq:dgd}
    \mX^{k+1} = \mix \mX^k - \alpha \nabla F(\mX^k),
\end{align}
which is a combination of two step types: gradient step with constant step-size $\alpha$ and communication round with mixing matrix $\mix^k$. In \cite{yuan2016convergence} the authors showed that function residual $f(\ol\mX^k) - f(x^*)$ in iterative scheme \eqref{eq:dgd} decreases at $O(1/k)$ rate until reaching $O(\alpha)$-neighborhood of solution. 

Method \eqref{eq:dgd} does not find an exact solution in the general case. We follow the arguments in \cite{shi2015extra} to illustrate this fact. First, note that $\tilde\mX$ is a solution of \eqref{eq:problem_equality_constraints} if and only if two following conditions hold.
\begin{enumerate}
    \item (Consensus) $\tilde\mX = \mix\tilde\mX$.
    \item (Optimality) $\one^\top \nabla F(\tilde\mX) = 0$.
\end{enumerate}
Let $\mX^\infty$ be a limit point of \eqref{eq:dgd}. Then
\begin{align*}
    \mX^{\infty} = \mix \mX^\infty - \alpha \nabla F(\mX^\infty).
\end{align*}
Consensus condition yields $\mX^\infty = \mix\mX^\infty$, i.e. $\mX^\infty$ has identical rows $[\mX]_i^\infty = x^\infty$.  Therefore, $\nabla F(\mX^\infty) = 0$, which means $\nabla f_1(x^\infty) = \ldots = \nabla f_m(x^\infty) = 0$. Consequently, $x^\infty$ is a common minimizer of every $f_i$, which is not a realistic case.

Method \eqref{eq:dgd} is a basic distributed first-order method. Its different variations include feasible point algorithms \cite{lee2013distributed} and sub-gradient methods \cite{nedic2009distributed} (actually, the latter work initially proposed scheme \eqref{eq:dgd}). Extensions to stochastic objectives and stochastic mixing matrices have been addressed in \cite{koloskova2020unified,aghajan2020distributed,lin2021quasi}.

\subsection{EXTRA}

Distributed gradient descent \eqref{eq:dgd} is unable to converge to the exact minimum of \eqref{eq:problem_equality_constraints}, which is the major drawback of the method. An exact decentralized first-order algorithm EXTRA was proposed in \cite{shi2015extra}. The approach of \cite{shi2015extra} is based on using two different mixing matrices. Namely, consider two consequent updates of type \eqref{eq:dgd}.
\begin{align}
    \mX^{k+2} &= \mix \mX^{k+1} - \alpha\nabla F(\mX^{k+1}), \label{eq:extra_update_1} \\
    \mX^{k+1} &= \tilde\mix \mX^k - \alpha\nabla F(\mX^k) \label{eq:extra_update_2}
\end{align}
where $\tilde\mix$ is mixing matrix, i.e. $\tilde\mix = (\mix + \mI) / 2$ as proposed in \cite{shi2015extra}. Subtracting \eqref{eq:extra_update_2} from \eqref{eq:extra_update_1} yields
\begin{align}\label{eq:extra_update_3}
    \mX^{k+2} - \mX^{k+1} = \mix\mX^{k+1} - \tilde\mix\mX^k - \alpha\sbraces{\nabla F(\mX^{k+1}) - \nabla F(\mX^k)}
\end{align}
thus leading to an algorithm
\begin{algorithm}
    \caption{EXTRA}
    \label{alg:extra}
    \begin{algorithmic}
        \Require{Step-size $\alpha > 0$.}
        \State{$\mX^1 = \mix\mX^0 - \alpha\nabla F(\mX^0)$}
        \For{$k = 0, 1,\ldots$}
            \State{$\mX^{k+2} = (\mI + \mix)\mX^{k+1} - \tilde\mix\mX^k - \alpha\sbraces{\nabla F(\mX^{k+1}) - \nabla F(\mX^k)}$}
        \EndFor
    \end{algorithmic}
\end{algorithm}

Let $\mX^\infty$ be a limit point of iterate sequence $\braces{\mX^k}_{k=0}^\infty$ generated by \eqref{eq:extra_update_1}, \eqref{eq:extra_update_2}. Then
\begin{align*}
    \mX^\infty - \mX^\infty &= \mix\mX^\infty - \tilde\mix\mX^\infty - \alpha\sbraces{\nabla F(\mX^\infty) - \nabla F(\mX^\infty)}, \\
    (\mix - \tilde\mix)\mX^\infty &= \frac{1}{2}(\mix\mX^\infty - \mX^\infty) = 0.
\end{align*}
The last equality means that $\mX^\infty$ is consensual, i.e. its rows are equal. On the other hand, rearranging the terms in \eqref{eq:extra_update_3} and taking into account that $\mX^1 = \mix\mX^0 - \alpha\nabla F(\mX^0)$ gives
\begin{align*}
    \mX^{k+2} = \tilde\mix\mX^{k+1} - \alpha\nabla F(\mX^{k+1}) + \sum_{t=0}^{k+1} (\mix - \tilde\mix) \mX^{t}.
\end{align*}
Multiplying by $\one^\top$ from the left yields
\begin{align*}
    \one^\top\mX^{k+2} &= \one^\top \tilde\mix\mX^{k+1} - \arrev{\alpha} \one^\top \nabla F(\mX^{k+1}) + \sum_{t=0}^{k-1} \one^\top(\mix - \tilde\mix)\mX^t \\
    &= \one^\top \tilde\mix\mX^{k+1} - \arrev{\alpha} \one^\top \nabla F(\mX^{k+1})
\end{align*}
and taking the limit over $k\to\infty$ we obtain
\begin{align*}
    \one^\top \nabla F(\mX^\infty) = 0
\end{align*}
which is the optimality condition for point $\mX^\infty$. Therefore, a limit point of $\braces{\mX^k}_{k=0}^\infty$ generated by Algorithm \ref{alg:extra} is both consensual and optimal, i.e. is a solution of \eqref{eq:problem_equality_constraints}.

In the original paper \cite{shi2015extra} Algorithm \ref{alg:extra} was proved to converge at a $O(1/k)$ rate for $L$-smooth objectives and achieve a geometric rate $O(C^{-k})$ (where $C < 1$ is some constant) for strongly convex smooth objectives. In \cite{li2020revisiting} explicit dependencies on graph topology were established. Namely, EXTRA requires 
\begin{align*}
    O\cbraces{\cbraces{\frac{\Lmax}{\mumin} + \chi}\log\frac{\ar{(LR^2 + {\tilde M}^2/L)}\chi}{\eps}}\qquad &\text{iterations for strongly convex smooth objectives,} \\
    O\cbraces{\cbraces{\frac{\Lmax}{\eps} + \chi}\log(\ar{(LR^2 + {\tilde M}^2/L)}\chi)}\qquad &\text{iterations for \arrev{(non-strongly)} convex smooth objectives,}
\end{align*}
where
\begin{align}\label{eq:chi}
    \chi = \frac{1}{1 - \lambda_2(\mix)},
\end{align}
$\lambda_2(\mix)$ denotes the second largest eigenvalue of mixing matrix $\mix$, \ar{$\norm{\mX^0 - \mX^*}_2^2\le mR^2,~ \norm{\mX^*}_2^2\le mR^2,~ \norm{\nabla f(\mX^*)}_2^2\le m{\tilde M}^2$}. The term $\chi$ characterizes graph connectivity. A similar term, also referred to as graph condition number, is used in Section \ref{sec:distributed_opt} for graph Laplacian matrix. Graph condition numbers based on mixing matrix and Laplacian have the same meaning, as discussed in Section \ref{sec:consensus_algorithms}.

\subsection{Accelerated decentralized algorithms}

Performance of decentralized gradient methods typically depends on function (local or global) condition number $\kappa$ and graph condition number $\chi$ defined in \eqref{eq:chi}. For non-accelerated dynamics \cite{yuan2016convergence,shi2015extra} complexity bounds include $\kappa$ and $\chi$. Improving dependencies to $\sqrt\kappa$ and $\sqrt\chi$ is an important direction of research in distributed optimization. This can be done by applying direct Nesterov acceleration \cite{nesterov2004introduction} or by employing meta-acceleration techniques such as Catalyst \cite{lin2015universal}. The two major approaches studied in the literature are primal and dual algorithms.

Dual methods are based on a reformulation of problem \eqref{P} using a Laplacian matrix induced by the communication network. This reformulation is discussed in Section \ref{sec:distributed_opt} in more details. The basic idea behind dual approach is to run first-order methods on a dual problem to \eqref{eq:problem_equality_constraints}. Every gradient step on the dual is equivalent to one communication round and one local gradient step taken by every node in the network. \ar{In \cite{scaman2017optimal}, algorithms using Chebyshev acceleration that achieve $O\cbraces{\sqrt{\kappamax\chi}\log(1/\eps)}$ communication complexity are proposed}. On the other hand, lower complexity bound for deterministic methods over strongly convex smooth objectives is $\Omega\cbraces{\sqrt{\kappa_g\chi}\log(1/\eps)}$, as shown in \cite{scaman2017optimal}.

In dual approach, one may run non-distributed accelerated schemes on dual problem and obtain accelerated complexity bounds, i.e. $\sqrt{\kappa\chi}$. For primal-only methods this is not the case, and primal algorithms have to alternate optimization and consensus steps in a proper way and employ specific techniques such as gradient tracking. A direct distributed scheme for Nesterov accelerated method was proposed in \cite{qu2019accelerated}.

\begin{algorithm}
    \caption{Accelerated Distributed Nesterov Method}
    \label{alg:acc_dngd}
    \begin{algorithmic}
        \Require{Starting points $\mX^0 = \mY^0 = \mV^0$, $\mS^0 = \nabla F(\mX^0)$, step-size $\eta > 0$, momentum term $\alpha = \sqrt{\mumin\eta}$}
        \For{$k = 0, 1, \ldots$}
            \State{$\mX^{k+1} = \mix\mY^k - \eta\mS^k$}
            \State{$\mV^{k+1} = (1 - \alpha)\mix\mV^k + \alpha\mix\mY^k - \frac{\eta}{\alpha}\mS^k$}
            \State{$\mY^{k+1} = \frac{\mX^{k+1} + \alpha\mV^{k+1}}{1 + \alpha}$}
            \State{$\mS^{k+1} = \mix\mS^k + \nabla F(\mY^{k+1}) - \nabla F(\mY^k)$}
        \EndFor
    \end{algorithmic}
\end{algorithm}
In Algorithm \ref{alg:acc_dngd} quantity $\mS^{k+1}$ stands for a gradient estimator. The information about the gradients held by different agents is diffused through the network via consensus steps, i.e. $\mix\mS^k$ multiplication. Every node stores one row of $\mS^k$ which approximates the average gradient over the nodes in network: 
\begin{align*}
    s_i^k\approx\frac{1}{m}\sum_{k=1}^m \nabla f_i(y_i^k).
\end{align*}
This technique is referred to as \textit{gradient tracking} and is employed in several primal decentralized methods \cite{qu2019accelerated,nedic2017achieving,ye2020multi,koloskova2021improved,alghunaim2020decentralized,qu2017harnessing,pu2021distributed}.

Algorithm \ref{alg:acc_dngd} requires $O(\chi^{3/2}\kappa_l^{5/7}\log(1/\eps))$ computation and communication steps to achieve accuracy $\eps$, which does not match optimal bounds. EXTRA acceleration via Catalyst envelope \cite{li2020revisiting} requires $O(\sqrt{\kappamax\chi}\log\chi\log(1/\eps))$ iterations for smooth strongly convex objectives. Recently a new method Mudag which unifies gradient tracking, Nesterov acceleration and multi-step consensus procedures was proposed in \cite{ye2020multi}. It \arrev{has}
\begin{align*}
    O\cbraces{\sqrt\kappaav\arrev{\log\cbraces{\frac{1}{\eps}}}}\qquad &\text{computation complexity and} \\
    O\cbraces{\sqrt{\kappaav\chi}\log\cbraces{\frac{\Lmax}{\Lav}\kappaav}\log\cbraces{\frac{1}{\eps}}}\qquad &\text{communication complexity.}
\end{align*}
Mudag reaches optimal computation complexity and optimal communication complexity up to $\log\cbraces{\frac{\Lmax}{\Lav}\kappaav}$ term. A valuable feature of the method is that it has dependencies on global condition number $\kappamax$ instead of local $\kappaav$. In the general case, global condition number may be significantly better. A proximal version of Mudag method for composite optimization is studied in \cite{ye2020decentralized}. The method in \cite{ye2020decentralized} requires an optimal $O(\sqrt{\kappaav\chi}\log(1/\eps))$ number of computations and matches the lower communication complexity bound up to a logarithmic factor. Global condition number is also utilized in paper \cite{rogozin2021towards} where an inexact oracle framework \cite{devolder2013first,devolder2014first} for decentralized optimization is studied. The latter work is discussed in more details in Section \ref{sec:inexact_oracle}. Finally, in \cite{kovalev2020optimal} authors proposed a primal-only method OPAPC which reaches both optimal computation and communication complexities (up to replacing $\kappaav$ with $\kappamax$).

\arrev{
Chebyshev acceleration is widely used to obtain optimal decentralized algorithms. For example, in \cite{kovalev2020optimal} the authors propose method APAPC, which has $O\cbraces{\cbraces{\sqrt{\frac L \mu \chi}  + \chi} \log\frac 1 \eps}$ communication and computational complexities. After that, the authors replace Laplacian $\mW$ with a Chebyshev polynomial $P_K(\mW)$, which results in $\chi(P_K(\mW)) = O(1)$, but every communication round costs $O(\sqrt\chi)$ communication rounds. Therefore, APAPC is modified to a new method OPAPC, which has $O\cbraces{\sqrt\frac L \mu \log\frac 1\eps}$ oracle per node complexity and $O\cbraces{\sqrt{\frac L\mu \chi} \log\frac 1\eps}$ communication complexity. In this particular case, Chebyshev acceleration not only allows to achieve optimal complexity bounds, but also separates oracle and communication complexities of the algorithm.
}

\arrev{Paper \cite{song2021optimal} proposed OGT, a method based on \textit{loopless} Chebyshev acceleration scheme. On the contrary to classical Chebyshev acceleration (used i.e. in OPAPC \cite{kovalev2020optimal}), the loopless technique does not require multiple communication steps at each iteration. OGT requires $O\cbraces{\sqrt{\frac{L}{\mu}}\log\frac{1}{\eps}}$ oracle calls at each node and $O\cbraces{\sqrt{\frac{L}{\mu} \chi} \log\frac{1}{\eps}}$ communication steps, which meets the lower bounds.}

\arrev{Moreover, a recent work \cite{song2021provably} showed that Nesterov acceleration can also be applied for distributed optimization over directed graphs. Their algorithm APD has communication complexity $\sim 1 / \sqrt\eps$ for non-strongly convex objectives and APD-SC has $\sim \sqrt{L/\mu} \log(1/\eps)$ complexity for strongly convex tasks. As stated in \cite{song2021optimal}, in the case of undirected graphs the explicit dependence on network characteristics is attained: the complexity of APD-SC writes as $O\cbraces{\sqrt{\frac L \mu}\chi^{3/2} \log\frac{1}{\eps}}$.}

\subsection{Time-varying networks}

In the time-varying case, the communication network changes from time to time. In practice this changes are typically caused by malfunctions such as loss of connection between \arrev{the} agents. The network is represented as a sequence of undirected communication graphs $\braces{\mathcal{G}^k = (V, E^k)}_{k=0}^\infty$, and every graph $\mathcal{G}^k$ is associated with a mixing matrix $\mix^k$. The algorithms capable of working over time-varying graphs must be robust to sudden network changes. A linearly convergent method DIGing was proposed in \cite{nedic2017achieving}.

\begin{algorithm}
    \caption{DIGing}
    \label{alg:diging}
    \begin{algorithmic}
        \Require{Step-size $\alpha > 0$, starting iterate $\mX^0$, $\mY^0 = \nabla F(\mX^0)$}
        \For{$k = 0, 1, \ldots$}
            \State{$\mX^{k+1} = \mix^k \mX^k - \alpha \mY^k$}
            \State{$\mY^{k+1} = \mix^k \mY^k + \nabla F(\mX^{k+1}) - \nabla F(\mX^k)$}
        \EndFor
    \end{algorithmic}
\end{algorithm}
DIGing incorporates a gradient-tracking scheme and achieves linear convergence under realistic assumptions such as $B$-connectivity (i.e. a union of any $B$ consequent graphs is connected). In \cite{sun2019convergence} authors propose an algorithm which utilizes specific convex surrogates of local functions and local functions similarity in order to enhance convergence speed. \ar{In \cite{li2021accelerated} a gradient-tracking technique combined with Nesterov acceleration was employed to construct an accelerated method AccGT over time-varying B-connected networks.}

Another class of time-varying networks are graphs that stay connected at each iteration. For this type of problems, denote Laplacian at $k$-th iteration $\mW^k$ and define condition number $\chi_{tw} = \frac{\max_{k}\lambda_{\max}(\mW^k)}{\min_k\lambda_{\min}^+(\mW^k)}$. The lower bounds for this class of problems are $O(\sqrt\kappamax\log(1/\eps))$ for the number of (local) computations and $O(\chi\sqrt\kappamax\log(1/\eps))$ for the number of communications \cite{kovalev2021lower}. AccGT \cite{li2021accelerated} and ADOM+ \cite{kovalev2021lower} are optimal algorithms using primal oracle, and ADOM \cite{kovalev2021adom} is an optimal dual method.

\subsection{Inexact oracle point of view}\label{sec:inexact_oracle}

In \cite{rogozin2021towards} the authors study an algorithm which alternates making gradient updates and running multi-step communication procedures. Introduce 
\begin{align*}
    \ol\mX = \frac{1}{m} \one\one^\top \mX = \Pi_{\cset}(\mX) = (\ol x \ldots \ol x)^\top, \text{ where } \ol x = \frac{1}{m}\sum_{i=1}^m x_i \text{ and } \one = (1\ldots 1)^\top.
\end{align*}
Also define an average gradient over nodes $\ol{\nabla F}(\mX) = 1/m\sum_{i=1}^m \nabla f_i(x_i)$. Consider a projection gradient method with trajectory lying in $\cset$
\begin{align*}
    \ol\mX_{k+1} = \ol\mX_k - \beta\ol{\nabla F}(\ol\mX_k).
\end{align*}
In a centralized scenario, the computational network is endowed with a master agent, which communicates with all agents in the network. The master node is able to collect vectors $x_i$ from every node in the network and compute a precise average $\ol x$. In decentralized case the master agent is not available, and therefore nodes are only able to compute an approximate average using consensus procedures. \arrev{The network is allowed to change with time and the sequence of corresponding mixing matrices is restricted to the following assumption}.

\begin{assumption}\label{assum:mixing_matrix_sequence}
	Mixing matrix sequence $\braces{\mix^k}_{k=0}^\infty$ satisfies the following properties.
	\begin{itemize}
		\item (Decentralized property) $(i, j)\notin E_k \;\Rightarrow \;[\mix^k]_{ij} = 0$.
		\item (Double stochasticity) $\mix^k \one = \one,~ \one^\top\mix^k = \one^\top$.
		\item (Contraction property) There exist $\tau\in\Z_{++}$ and $\lambda\in(0, 1)$ such that for every $k\ge \tau - 1$ it holds
		\begin{align*}
		\norm{\mix_{\tau}^k X - \ol X}_2 \le (1 - \lambda)\norm{X - \ol X}_2, 
		\end{align*}
		where $\mix_\tau^k = \mix^k \ldots \mix^{k-\tau+1}$.
	\end{itemize}
\end{assumption}



\begin{algorithm}[H]
	\caption{Consensus}
	\label{alg:consensus}
	\begin{algorithmic}
		\Require{Initial $\mX^0\in\cset$, number of iterations $T$.}
		\For{$t = 1, \ldots, T$}
			\State{$\mX^{t+1} = \mix^t \mX^t$}
		\EndFor
	\end{algorithmic}
\end{algorithm}

\begin{algorithm}[H]
	\caption{Decentralized AGD with consensus subroutine}
	\label{alg:decentralized_agd}
	\begin{algorithmic}[1]
		\Require{Initial guess $\mX^0\in \cset$, constants $L, \mu > 0$, $\mU^0 = \mX^0$, $\alpha^0 = \mix^0 = 0$}
		\For{$k = 0, 1, 2,\ldots$}
			\State{Find $\alpha^{k+1}$ as the greater root of $(A^k + \alpha^{k+1})(1 + A^k \mu) = L(\alpha^{k+1})^2$}
			\State{$A^{k+1} = A^k + \alpha^{k+1}$}
			\State\label{state:agd_step}{$\ds \mY^{k+1} = \frac{\alpha^{k+1} \mU^k + A^k \mX^k}{A^{k+1}}$}
            \State{$\mV^{k+1} = \dfrac{\mu\mY^{k+1} + (1 + A^k\mu)\mU^k}{1 + A^k\mu + \mu} - \dfrac{\alpha^{k+1}}{1 + A^k\mu + \mu} \nabla F(\mY^{k+1})$}
            \State\label{state:consensus_update}{$\mU^{k+1} = \text{Consensus}(\mV^{k+1}, T^k)$}
			\State{$\ds \mX^{k+1} = \frac{\alpha^{k+1} \mU^{k+1} + A^k \mX^k}{A^{k+1}}$}
		\EndFor
	\end{algorithmic}
\end{algorithm}

\subsubsection{Inexact oracle construction}

Trajectory of Algorithm \ref{alg:decentralized_agd} lies in the neighborhood of constraint set $\cset$. It is analyzed in \cite{rogozin2021towards} \arrev{(based on the technique developed in \cite{rogozin2019projected})} using the notation of inexact oracle \cite{devolder2014first,devolder2013first}. Algorithms of this type have been analyzed in time-static case \cite{jakovetic2014fast} using inexact oracle notation, as well. Let $h(x)$ be a convex function defined on a convex set $Q\subseteq\R^m$. For $\delta > 0,~ L > \mu > 0$, a pair $(h_{\delta,L,\mu}(x), s_{\delta,L,\mu}(x))$ is called a $(\delta, L, \mu)$-model of $h(x)$ at point $x\in Q$ if for all $y\in Q$ it holds
\begin{align}\label{eq:inexact_oracle_def_devolder}
	\frac{\mu}{2}\norm{y - x}_2^2 \le h(y) - \cbraces{h_{\delta,L,\mu}(x) + \angles{s_{\delta,L,\mu}(x), y - x}} \le \frac{L}{2} \norm{y - x}_2^2 + \delta.
\end{align}
The inexactness originates from computation of gradient at a point in neighborhood of $\cset$. The next lemma identifies the size of neighborhood and describes the inexact oracle construction.
\begin{lemma}\label{lem:inexact_oracle}
	Define
    \begin{align*}
        \delta &= \frac{1}{2n}\cbraces{\frac{\Lmax^2}{\Lav} + \frac{2\Lmax^2}{\muav} + \Lmax - \mumin} \delta', \numberthis\label{eq:delta_inexact_oracle} \\
        f_{\delta, L, \mu}(\ol x, \mX) &= \frac{1}{n} \sbraces{F(\mX) + \angles{\nabla F(\mX), \ol\mX - \mX} + \frac{1}{2}\cbraces{\mumin - \frac{2\Lmax^2}{\muav}}\norm{\ol\mX - \mX}_2^2}, \\
        g_{\delta, L, \mu}(\ol x, \mX) &= \frac{1}{n}\sum_{i=1}^n \nabla f_i(x_i).
    \end{align*}
	Then $(f_{\delta,L,\mu}(\ol x, \mX), g_{\delta,L,\mu}(\ol x, \mX))$ is a $(\delta,2\Lav,\muav/2)$-model of $f$ at point $\ol x$, i.e.

	\begin{align*}
		\frac{\muav}{4}\norm{\ol y - \ol x}_2^2 \le f(\ol y) - f_{\delta, L, \mu}(\ol x, \mX) - \angles{g_{\delta, L, \mu}(\ol x, \mX), \ol y - \ol x} \le \Lav\norm{\ol y - \ol x}_2^2 + \delta.
	\end{align*}
\end{lemma}
The inexact oracle defined in Lemma \ref{lem:inexact_oracle} represents a $(\delta, 2L_g, \mu_g / 2)$-model of $F$. Note that it uses \textit{global} strong convexity constants instead of \textit{local} ones. Global constants may be significantly better for method performance, as pointed out in \cite{scaman2017optimal}. The lemma relates the projection accuracy $\delta'$ to inexact oracle parameter $\delta$.

\subsubsection{Convergence result for Algorithm \ref{alg:decentralized_agd}}

First, Lemma \ref{lem:inexact_oracle} states that $(\delta, 2L_g, \mu_g/2)$-model of $F$ is obtained if the gradient is computed in $\delta'$-neighborhood of $\cset$. In order to achieve this $\delta'$-neighborhood, one needs to make a sufficient number of consensus (Algorithm \ref{alg:consensus}) iterations.

\begin{lemma}\label{lem:consensus_iters_strongly_convex}
	Let consensus accuracy be maintained at level $\delta'$, i.e. $\norm{\mU^j - \ol\mU^j}_2^2\le \delta' \text{ for } j = 1, \ldots, k$ and let Assumption \ref{assum:mixing_matrix_sequence} hold. Define
	\begin{align*}
	\sqrt D := \cbraces{\frac{2\Lmax}{\sqrt{L\mu}} + 1}\sqrt{\delta'} + \frac{\Lmax}{\mu} \sqrt{n} \cbraces{\norm{\ol u^0 - x^*}_2^2 + \frac{8{\delta'}}{\sqrt{L\mu}}}^{1/2} + \frac{2\norm{\nabla F(\mX^*)}_2}{\sqrt{L\mu}}.
	\end{align*}
	
	Then it is sufficient to make $T_k = T = \frac{\tau}{2\lambda}\log\frac{D}{\delta'}$ consensus iterations \arrev{(where $\tau$ and $\lambda$ are defined in Assumption \ref{assum:mixing_matrix_sequence})} in order to obtain consensus with $\delta'$-accuracy on step $k+1$, i.e. $\norm{\mU^{k+1} - \ol\mU^{k+1}}_2^2\le \delta'$.
\end{lemma}
A basis for the proof of Lemma \ref{lem:consensus_iters_strongly_convex} is a contraction property of mixing matrix sequence $\braces{\mix^k}_{k=0}^\infty$ (see Assumption \ref{assum:mixing_matrix_sequence}).

Second, provided that projection accuracy on every step of Algorithm \ref{alg:decentralized_agd} is sustained at level $\delta'$, the algorithm turns into an accelerated scheme with inexactness. Its convergence \arrev{rate} is given by the following 
\begin{lemma}\label{lem:inexact_agd_convergence}
	Provided that consensus accuracy is $\delta'$, i.e. $\norm{\mU^{j} - \ol\mU^j}_2^2\le \delta' \text{ for } j = 1, \ldots, k$, we have
	\begin{align*}
	f(\ol x^k) - f(x^*) &\le \frac{\norm{\ol u^0 - x^*}_2^2}{2A^k} + \frac{2\sum_{j=1}^{k} A^j\delta}{A^k} \\
	\norm{\ol u^k - x^*}_2^2 &\le \frac{\norm{\ol u^0 - x^*}_2^2}{1 + A^k\mu} + \frac{4\sum_{j=1}^{k} A^j\delta}{1 + A^k\mu}
	\end{align*}
	where $\delta$ is given in \eqref{eq:delta_inexact_oracle}.
\end{lemma}

Finally, putting Lemmas \ref{lem:consensus_iters_strongly_convex} and \ref{lem:inexact_agd_convergence} together yields a convergence result for Algorithm \ref{alg:decentralized_agd}.
\begin{theorem}\label{th:total_iterations_strongly_convex}
	\arrev{Recall the definitions of $\tau$ and $\lambda$ from Assumption 2,} choose some $\eps > 0$ and set
	\begin{align*}
	T_k = T = \frac{\tau}{2\lambda}\log\frac{D}{\delta'},~ \delta' = \frac{n\eps}{32} \frac{\muav^{3/2}}{\Lav^{1/2} \Lmax^2}.
	\end{align*}
	Also define
	\begin{subequations}\label{eq:log_coefs_strongly_conv}
		\begin{align*}
			D_1 &= \frac{\Lmax}{\Lav^{1/2}\muav} \sbraces{ 8\sqrt{2}\Lmax\norm{\ol u^0 - x^*}_2\cbraces{\frac{\Lav}{\muav}}^{3/4} + \frac{4\sqrt{2}\norm{\nabla F(\mX^*)}_2}{\sqrt{n}}\cbraces{\frac{\Lav}{\muav}}^{1/4} }, \\
			D_2 &= \frac{\Lmax}{\Lav^{1/2}\muav} \sbraces{ 3\sqrt{\muav} + 4\sqrt{2n}\cbraces{\frac{\Lav}{\muav}}^{1/4} }.
		\end{align*}
	\end{subequations}
	Then Algorithm \ref{alg:decentralized_agd} requires
	\begin{align}\label{eq:agd_computational_complexity}
		N = 2\sqrt{\frac{\Lav}{\muav}} \log\cbraces{\frac{\norm{\ol u^0 - x^*}_2^2}{2\Lav\eps}}
	\end{align}
	gradient computations at each node and
	\begin{align}\label{eq:agd_communication_complexity}
		N_{tot} = N\cdot T =  2\sqrt{\frac{\Lav}{\muav}} \frac{\tau}{\lambda} \cdot \log\cbraces{\frac{2\Lav\norm{\ol u^0 - x^*}_2^2}{\eps}} \log\cbraces{\frac{D_1}{\sqrt\eps} + D_2}
	\end{align}
	communication steps to yield $\mX^N$ such that
	\begin{align*}
		&f(\ol x^N) - f(x^*)\le \eps,~ \norm{\mX^N - \arrev{\ol X^N}}_2^2\le \delta'.
	\end{align*}
\end{theorem}

In the time-static case, contraction term $\tau/\lambda$ turns into $\chi(\mix)$, and an accelerated consensus procedure of type \eqref{eq:consensus_iter_laplacian_accelerated} may be employed. This results in a better dependence on graph connectivity and leads to a complexity bound $O\cbraces{\sqrt\frac{L_g}{\mu_g}\sqrt{\chi(\mix)} \log^2(\frac{1}{\eps})}$ which is optimal up to a logarithmic term. Similar results are attained in works which use penalty-based methods \cite{li2020decentralized,rogozin2020penalty,gorbunov2019optimal} (see Appendix B in \cite{gorbunov2019optimal}) for details.

\arrev{
\textbf{Remark}. The analysis of Algorithm \ref{alg:decentralized_agd} presented in \cite{rogozin2021accelerated} results in constants $\Lav, \muav$ in the complexity bound. These constants are better than local constants $\Lmax, \mumin$, but still can be improved. The inexact oracle concept allows to reduce decentralized optimization problem to minimization of $f(x)$ over $\R^d$ with inexact oracle. Therefore, the complexity will depend on constants $L_f, \mu_f$ which characterize $f$ itself, not its flattened variant $F$. An accurate analysis on this issue is presented in Section \ref{subsec:decentralized_saddle}.
}

\subsubsection{\arrev{Stochastic decentralized optimization}}

\arrev{
The technique used in Algorithm \ref{alg:decentralized_agd} can be extended to stochastic objectives. Following the definitions in \cite{rogozin2021accelerated}, let
$f_i(x):=\EE_{\xi} f_i(x, \xi_i)$,
where $\xi_i$'s are random variables. Variables $\xi_i$ represent the source of stochasticity in $f_i(x, \xi_i)$ which may be caused by random sampling or stochastic noise. For each $i = 1, \ldots, n$ we assume that $\nabla f_i(x, \xi_i)$ is $L_i(\xi)$ continuous and there exists a constant $L_i\geq 0$ such that $\sqrt{\EE_{\xi_i} L_i(\xi_i)^2} \leq L_i < +\infty$. Under these assumptions $f_i$ is $L_i$-smooth. We also bound the variance of $\nabla f_i(x, \xi_i)$:
\begin{equation*}
\EE_{\xi_i}[\norm{\nabla f_i(x, \xi_i) - \nabla f_i(x)}_2^2] \leqslant \sigma_i^2.
\end{equation*}
Let us define $\sigma_g^2 = \frac{1}{m}\sum\limits_{i=1}^m\sigma_i^2$. The algorithm in \cite{rogozin2021accelerated} combines a consensus subroutine technique similar to Algorithm \ref{alg:decentralized_agd} and also uses a specific batch-size policy. In order to analyze the method, inexact oracle framework similar to that of Section \ref{sec:inexact_oracle} is used. The inexactness of gradient has two sources: inexact projection onto the constraint set via consensus subroutine and stochastic noise. On the one hand, tuning the batch size allows to reduce the variance of the batched gradient at the cost of additional stochastic oracle calls. Therefore, proper batch size guarantees a balance between the stochastic gradient noise and the number of gradient calculations. On the other hand, the accuracy of the consensus is tuned by the choice of number of consensus iterations. Choosing a proper batch size and number of consensus iterations allows to obtain optimal complexities both in the number of computations and communications up to a logarithmic factor. Namely, the method in \cite{rogozin2021accelerated} requires $\widetilde{O}\left(\max\left\{\frac{\sigma_g^2}{n\mu_g\eps},\sqrt{\frac{L_g}{\mu_g}}\log\frac{1}{\eps}\right\}\right)$ oracle calls per node. In the time-varying case, it requires $\widetilde{O}\cbraces{\frac{\tau}{\lambda}\sqrt{\frac{\Lav}{\muav}}}$ communication rounds (where $\tau$ and $\lambda$ are defined in Assumption \ref{assum:mixing_matrix_sequence}), and in time-static case its communication complexity takes the form $\widetilde{O}\cbraces{\sqrt{\frac{\Lav}{\muav}\chi}}$ and is achieved by using Chebyshev acceleration.
}

\subsection{Decentralized Saddle-Point Problems}\label{subsec:decentralized_saddle}

Along with minimization problems, sum-type min-max problems of type
\begin{align*}
    \min_{x\in\mathcal{X}} \max_{y\in\mathcal{Y}}~ \arrev{f(x, y) :=} \frac{1}{m} \sum_{i=1}^m f_i(x, y)
\end{align*}
where $\mathcal{X}$ and $\mathcal{Y}$ are convex compacts, can be solved in a decentralized manner, as well. The same way as in Assumptions \ref{assum:convex_smooth}, \ref{assum:str_convex} for minimization tasks, we introduce assumptions for min-max problems.
\begin{assumption}\label{assum:convex_concave_smooth}
    For every $i = 1, \ldots, m$, function $f_i$ is differentiable, convex in $x$, concave in $y$ and $L_i$-smooth.
\end{assumption}
\begin{assumption}\label{assum:str_convex_concave}
    Function $f$ is $\mu$-strongly-convex in $x$, $\arrev{\mu_f}$-strongly-concave in $y$ ($\arrev{\mu_f} > 0$) and $L_f$ -smooth.
\end{assumption}


\ab{
Saddle-point problems have many practical applications: classical and well-studied in economy and in game theory 
\cite{GT-book,facchinei2007finite}, and modern in imaging denoising \cite{chambolle2011first}, in   adversarial training \cite{Arjovsky_et_al2017, Bengio2014}, and in statistical learning \cite{Abadeh_et_al_2015}. But distributed saddle-point problems is not as widely studied in the literature as the minimization problems. Let  us highlighted the main works devoted to decentralized min-max problems.
}
\ab{
Most of the works are devoted to decentralized algorithms on fixed graph topology. In paper \cite{beznosikov2021distributed}, the authors present lower bounds for deterministic decentralized saddle point problems under Assumptions \ref{assum:convex_concave_smooth} and \ref{assum:str_convex_concave}. These estimates are as follows
\begin{align}
\label{SPP_com}
    \Omega\cbraces{\frac{L_f}{\mu_f}\log\cbraces{1/\eps}}\qquad &\text{computation complexity and} \nonumber\\
    \Omega\cbraces{\sqrt{\chi}\frac{L_f}{\mu_f}\log\cbraces{1/\eps}}\qquad &\text{communication complexity.}
\end{align}
Additionally, the paper provides an optimal algorithm (up to logarithmic factors), which achieves the lower bounds. Among the disadvantages of the Algorithm presented in \cite{beznosikov2021distributed}, one can single out multiple gossip steps, this approach is unstable and not a popular in practice.
A similar Algorithm with multiple gossip steps is proposed in \cite{liu2019decentralized}, but they consider convergence in the non-convex case (under the minty condition \cite{minty62,juditsky2011solving}). Also, this work shows the effectiveness of decentralized training of GANs. \cite{liu2019decentralizedprox} is also devoted to minty non-convex saddle-point problems. It is also interesting to note the work \cite{beznosikov2021distributed} on saddle-point problems in terms of data-similarity. The work gives lower bounds for communication complexity, as well as optimal algorithms for such setting of the problem. In particular, the lower and upper bounds state that
\begin{align*}
    \sqrt{\chi}\left(1 + \frac{\delta}{\mu_f}\right)\log\cbraces{1/\eps}\qquad &\text{communication rounds}
\end{align*}
are enough to achieve $\varepsilon$-precision. Interesting to note, that for uniformly distributed data, with high probability $L_i \approx L_f$ and $\delta \sim {\tilde O}(\max_i L_i/\sqrt{n})$, where $n$ -- the number of local samples on each node. This means that the data-simularity bounds on communication rounds is significantly better than the general one \eqref{SPP_com}.
}

\ab{
It is also important to mention the works devoted to saddle-point problems on time-varying networks. In particular, paper \cite{beznosikov2021optimal} is devoted to lower bounds and optimal algorithms for connected topology. Work \cite{beznosikov2021decentralized} is devoted to the broader case of time-varying networks, for example, methods can do local steps (iterations without communication).
}

\ar{
An interesting variant of saddle-point problems are extensions to local and global variables, i.e. problems of the form
\begin{align}\label{eq:saddle_local_and_global}
    \min_{p, \{x_i\}_{i=1}^m} ~ \max_{r, \{y_i\}_{i=1}^m} ~ {\frac{1}{m}} \sum_{i=1}^m f_i(x_i, p, y_i, r).
\end{align}
Applications of problems of this type minimization tasks with separable and semi-definite constraints \cite{mateos2015distributed}, decentralized reinforcement learning \cite{wai2018multi} and distributed computation of Wasserstein barycenters \cite{rogozin2021decentralized,dvinskikh2020improved}. A subgradient method for problems of type \eqref{eq:saddle_local_and_global} was proposed in \cite{mateos2015distributed}. The method has a $O(1/\sqrt N)$ convergence rate. A recent work \cite{rogozin2021decentralized} proposed a method based on Mirror-Prox, capable of working in general proximal setup, reaching a $O(1/N)$ convergence rate and an accelerated rate on $\chi$. The method achieves optimal oracle and communication complexities in Euclidean convex-concave case over time-static graphs.
}


\section{Convex Problems with Affine Constraints}\label{gorbunov}

In this section\footnote{The narrative in this section follows \cite{gorbunov2019optimal}.}, we consider convex optimization problem with affine constraints
\begin{equation}
\label{PP}
\min_{Ax=0, x\in Q}f(x),    
\end{equation}
where $A \succeq 0$, $\text{Ker} A \neq \{0\}$ and $Q$ is a closed convex subset of $\R^n$. Up to a sign, the dual problem is defined as follows:
\begin{eqnarray}
\min_{y}\psi(y),&& \text{where}\label{DP}\\
\varphi(y) &=& \max_{x\in Q}\left\{\langle y,x\rangle - f(x)\right\},\label{eq:dual_phi_function}\\
\psi(y) &=& \vp(A^\top y) = \max_{x\in Q}\left\{\langle y,Ax\rangle - f(x)\right\}=  \langle A^\top y,x(A^\top y)\rangle - f(x(A^\top y)),\label{eq:dual_function}
\end{eqnarray}
where $x(y) \eqdef \argmax_{x\in Q}\left\{\la y, x\ra - f(x)\right\}$. Since $\text{Ker}A \neq \{0\}$ the solution of the dual problem \eqref{DP} is not unique. We use $y^*$ to denote the solution of \eqref{DP} with the smallest $\ell_2$-norm $R_y \eqdef \|y^*\|_2$. 

\subsection{Primal Approach}\label{sec:primal}
In this section, we focus on primal approaches to solve \eqref{PP} and, in particular, the main goal of this section is to present first-order methods that are optimal both in terms of $\nabla f(x)$ and $A^\top A x$ calculations. One can apply the following trick \cite{dvinskikh2019decentralized, gasnikov2017modern,gorbunov2019optimal} to solve problem \eqref{PP}: instead of \eqref{PP} one can solve penalized problem
\begin{equation}
\label{penalty}
\min_{x\in Q} \left\{F(x) = f(x) + \frac{R_y^2}{\e}\| Ax\|_2^2\right\},  
\end{equation}
where $\e > 0$ is the desired accuracy of the solution in terms of $f(x)$ that we want to achieve (see the details in \cite{gorbunov2019optimal}).

Next, we assume that $f$ is $\mu$-strongly convex, but possibly non-smooth function with bounded (sub) gradients: $\|\nabla f(x)\|_2 \le M$ for all $x\in Q$. In this setting, one can apply {\tt Sliding} algorithm from \cite{Lan2019lectures,lan2016gradient} to get optimal rates of convergence. The method is presented as Algorithm~\ref{alg:sliding} and it is aimed to solve the following problem:
\begin{equation}
    \min\limits_{x\in Q} \left\{\Psi(x) = h(x) + f(x)\right\}, \label{eq:composite_problem}
\end{equation}
where $h(x)$ is convex and $L$-smooth, $f(x)$ is convex, but can be non-smooth, and $x^*$ is an arbitrary solution of the problem. In this case, it is additionally assumed that $f(x)$ has uniformly bounded subgradients: there exists non-negative constant $M$ such that\footnote{For the sake of simplicity, we slightly abuse the notation and denote gradients and subgradients similarly.} $\|\nabla f(x)\|_2 \le M$ for all $x\in Q$ and all subgradients at this point $\nabla f(x)\in\partial f(x)$.
\begin{algorithm} [H]
	\caption{Sliding Algorithm \cite{Lan2019lectures,lan2016gradient}}
	\label{alg:sliding}
	\begin{algorithmic}
\State
\noindent {\bf Input:} Initial point $x_0 \in Q$ and iteration limit $N$.
\State Let $\beta_k \in \RR_{++}, \gamma_k \in \RR_+$, and $T_k \in {\mathbb N}$, $k = 1, 2, \ldots$,
be given and
set $\overline x_0 = x_0$. 
\For {$k=1, 2, \ldots, N$ }
    \State 1. Set $\underline x_k = (1 - \gamma_k) \overline x_{k-1} + \gamma_k x_{k-1}$,
    and let $h_k(\cdot) \equiv l_h(\underline x_{k}, \cdot)$, where $l_h(x,y) = h(x) + \la\nabla h(x), y-x \ra$.
    \State 2. Set
    \begin{equation*}
        (x_k, \tilde x_k) = \text{\tt PS}(h_k, x_{k-1}, \beta_k, T_k).
    \end{equation*}
    \State 3. Set $\overline x_k = (1-\gamma_k) \overline x_{k-1} + \gamma_k \tilde x_k$. 
\EndFor
\State 
\noindent {\bf Output:} $\overline x_N$.

\Statex
\Statex The  $\text{\tt{PS} }$(prox-sliding) procedure.

\State {\bf procedure:} {$(x^+, \tilde x^+) = \text{\tt{PS}}$($g$, $x$, $\beta$, $T$)}
\State Let the parameters $p_t \in \R_{++}$ and $\theta_t \in [0,1]$,
$t = 1, \ldots$, be given. Set $u_{0} = \tilde u_0 = x$.
\State {\bf for} $t = 1, 2, \ldots, T$ {\bf do}
    \begin{eqnarray*}
        u_{t} &=& \argmin_{u \in Q} \Big\{g(u) + l_f(u_{t-1},u)+\frac{\beta}{2}\|u-x\|_2^2 + \frac{\beta p_t}{2}\|u-u_{t-1}\|_2^2\Big\},\\
        \tilde u_t &=&  (1-\theta_t) \tilde u_{t-1} + \theta_t u_t,
    \end{eqnarray*}
    where $l_f(x,y) = f(x) + \la \nabla f(x), y-x \ra$.
\State {\bf end for}
\State Set $x^+ = u_T$ and  $\tilde x^+ = \tilde u_T$.
\State {\bf end procedure:}
\end{algorithmic}
\end{algorithm}
The key property of Algorithm~\ref{alg:sliding} is its ability to separate oracle complexities for smooth and non-smooth parts of the objective. That is, to find such $\hat x$ that $\Psi(\hat x) - \Psi(x^*) \le \varepsilon$ {\tt Sliding} requires
\begin{equation*}
    O\left(\sqrt{\frac{LR^2}{\varepsilon}}\right)\text{ calculations of } \nabla h(x)
\end{equation*}
and
\begin{equation*}
    O\left(\frac{M^2R^2}{\varepsilon^2} + \sqrt{\frac{LR^2}{\varepsilon}}\right)\text{ calculations of } \nabla f(x),
\end{equation*}
where $R = \|x^0 - x^*\|_2$.

Now, we go back to the problem \eqref{penalty} and consider the case when $\mu = 0$. In these settings, to find $\hat x$ such that
\begin{equation}
    F(\hat x) - F(x^*) \le \varepsilon\label{eq:F(x^N)_guarantee}
\end{equation}
one can run Algorithm~\ref{alg:sliding} considering $f(x)$ as the non-smooth term and $\nicefrac{R_y^2}{\varepsilon}\|Ax\|_2^2$ as the smooth one. In this case, {\tt Sliding} requires
\begin{equation}
    O\left(\sqrt{\frac{\lambda_{\max}(A^\top A)R_y^2 R^2}{\e^2}}\right) \text{ calculations of $A^\top Ax$,} \label{eq:sliding_detrem_A^TAx_calculations}
\end{equation}
\begin{equation}
    O\left(\frac{M^2R^2}{\e^2}\right) \text{ calculations of $\nabla f(x)$.} \label{eq:sliding_detrem_grad_calculations}
\end{equation}

Next, we consider the situation when $Q$ is a compact set, $\nabla f(x)$ is not available, and unbiased stochastic gradient $\nabla f(x,\xi)$ is used instead:
\begin{eqnarray}
    \left\|\EE_\xi\left[\nabla f(x,\xi)\right] - \nabla f(x)\right\|_2 &\le& \delta, \label{eq:primal_bias_in_stoch_grad}\\
    \EE_\xi\left[\exp\left(\frac{\left\|\nabla f(x,\xi) - \EE_\xi\left[\nabla f(x,\xi)\right]\right\|_2^2}{\sigma^2}\right)\right] &\le& \exp(1) \label{eq:primal_light_tails_stoch_grad},
\end{eqnarray}
where $\delta \ge 0$ and $\sigma \ge 0$. When $\delta = 0$, i.e., stochastic gradients are unbiased, one can show \cite{Lan2019lectures,lan2016gradient} that Stochastic {\tt Sliding} ({\tt S-Sliding}) method can achieve \eqref{eq:F(x^N)_guarantee} with probability at least $1 - \beta$, $\beta\in(0,1)$ requiring the same number of calculations of $A^\top Ax$ as in \eqref{eq:sliding_detrem_A^TAx_calculations} up to logarithmic factors
and
\begin{equation}
    \widetilde{O}\left(\frac{(M^2+\sigma^2)R^2}{\e^2}\right) \text{ calculations of $\nabla f(x,\xi)$.} \label{eq:sliding_stoch_grad_calculations}
\end{equation}

When $\mu > 0$ one can apply restarts technique for {\tt S-Sliding} and get the method ({\tt RS-Sliding}) \cite{dvinskikh2019decentralized,uribe2017optimal} that guarantees \eqref{eq:F(x^N)_guarantee} with probability at least $1-\beta$, $\beta\in(0,1)$ using
\begin{equation}
    \widetilde{O}\left(\sqrt{\frac{\lambda_{\max}(A^\top A)R_y^2}{\mu\e}}\right) \text{ calculations of $A^\top Ax$,} \label{eq:sliding_stoch_A^TAx_calculations_str_cvx}
\end{equation}
\begin{equation}
    \widetilde{O}\left(\frac{M^2+\sigma^2}{\mu\e}\right) \text{ calculations of $\nabla f(x,\xi)$.} \label{eq:sliding_stoch_grad_calculations_str_cvx}
\end{equation}

We notice that bounds presented above for the non-smooth case are proved when $Q$ is bounded. For the case of unbounded $Q$ the convergence results with such rates were established only in expectation. Moreover, it would be interesting to study {\tt S-Sliding} and {\tt RS-Sliding} in the case when $\delta > 0$, i.e., stochastic gradient is biased.

\subsection{Dual Approach}\label{sec:dual}
In this section, we assume that one can construct a dual problem for \eqref{PP}. If $f$ is $\mu$-strongly convex in $\ell_2$-norm, then $\psi$ and $\varphi$ have $L_{\psi}$--Lipschitz continuous and $L_\varphi$--Lipschitz continuous in $\ell_2$-norm gradients respectively \cite{kakade2009duality,Rockafellar2015}, where $L_{\psi}=\nicefrac{\lambda_{\max}(A^\top A)}{\mu}$ and $L_\varphi = \nicefrac{1}{\mu}$. In our proofs, we often use Demyanov--Danskin theorem \cite{Rockafellar2015} which states that
\begin{equation}
    \nabla \psi(y) = Ax(A^\top y),\quad \nabla\varphi(y) = x(y).\label{eq:gradient_dual_function}
\end{equation}
Moreover, we do not assume that $A$ is symmetric or positive semidefinite.

Below we propose a primal-dual method for the case when $f$ is additionally Lipschitz continuous on some ball and two methods for the problems when the primal function is also $L$-smooth and Lipschitz continuous on some ball. In the subsections below, we assume that $Q = \R^n$. The formal proofs of the presented results are given in \cite{gorbunov2019optimal}.

\subsubsection{Convex Dual Function}\label{sec:conv_dual}
In this section, we assume that the dual function $\varphi(y)$ could be rewritten as an expectation, i.e., $\varphi(y) = \EE_\xi\left[\varphi(y,\xi)\right]$, where stochastic realizations $\varphi(y,\xi)$ are differentiable in $y$ functions almost surely in $\xi$. Then, we can also represent $\psi(y)$ as an expectation: $\psi(y) = \EE_\xi\left[\psi(y,\xi)\right]$. Consider the stochastic function $f(x,\xi)$ which is defined implicitly as follows:
\begin{equation}
    \varphi(y,\xi) = \max\limits_{x\in \R^n}\left\{\la y, x \ra - f(x,\xi)\right\}.\label{eq:dual_stoch_func}
\end{equation}
Similarly to the deterministic case, we introduce $x(y,\xi) \eqdef \argmax_{x\in \R^n}\left\{\la y, x \ra - f(x,\xi)\right\}$ which satisfies $\nabla\varphi(y,\xi) = x(y,\xi)$ due to Demyanov--Danskin theorem, where the gradient is taken w.r.t.\ $y$. As a simple corollary, we get $\nabla \psi(y,\xi) = Ax(A^\top y)$. Finally, introduced notations and obtained relations imply that $x(y) = \EE_\xi[x(y,\xi)]$ and $\nabla\psi(y) = \EE_\xi[\nabla \psi(y,\xi)]$.

Consider the situation when $x(y,\xi)$ is known only through the noisy observations $\tx(y,\xi) = x(y,\xi) + \delta(y,\xi)$ and assume that the noise is bounded in expectation, i.e., there exists non-negative deterministic constant $\delta_y \ge 0$, such that
\begin{equation}
    \left\|\EE_\xi[\delta(y,\xi)]\right\|_2 \le \delta_y,\quad \forall y\in \R^n. \label{eq:noise_level_x}
\end{equation}
Assume additionally that $\tx(y,\xi)$ satisfies so-called ``light-tails'' inequality:
\begin{equation}
    \EE_\xi\left[\exp\left(\frac{\left\|\tx(y,\xi) - \EE_\xi\left[\tx(y,\xi)\right]\right\|_2^2}{\sigma_x^2}\right)\right] \le \exp(1), \quad \forall y\in\R^n,\label{eq:light_tails_x}
\end{equation}
where $\sigma_x$ is some positive constant. It implies that we have an access to the biased stochastic gradient $\tnabla\psi(y,\xi) \eqdef A\tx(y,\xi)$ which satisfies following relations:
\begin{eqnarray}
    \left\|\EE_\xi\left[\tnabla \psi(y,\xi)\right] - \nabla \psi(y)\right\|_2 &\le& \delta, \quad \forall y\in \R^n,\label{eq:bias_stoch_grad}\\
    \EE_\xi\left[\exp\left(\frac{\left\|\tnabla \psi(y,\xi) - \EE_\xi\left[\tnabla \psi(y,\xi)\right]\right\|_2^2}{\sigma_\psi^2}\right)\right] &\le& \exp(1), \quad \forall y\in \R^d,\label{eq:super_exp_moment_stoch_grad}
\end{eqnarray}
where $\delta \eqdef \sqrt{\lambda_{\max}(A^\top A)}\delta_y$ and $\sigma_\psi \eqdef \sqrt{\lambda_{\max}(A^\top A)}\sigma_x$.
We will use $\tnabla\Psi(y,\Bxi^{k})$ to denote batched stochastic gradient:
\begin{equation}
    \tnabla\Psi(y,\Bxi^{k}) = \frac{1}{r_k}\sum\limits_{l=1}^{r_k}\tnabla \psi(y,\xi^{l}), \quad \tx(y,\Bxi^k) = \frac{1}{r_k}\sum\limits_{l=1}^{r_k}\tx(y,\xi^l)\label{eq:batched_biased_stoch_gradient}
\end{equation}
The size of the batch $r_k$ could always be restored from the context, so, we do not specify it here. Note that the batch version satisfies (see the details in \cite{gorbunov2019optimal})
\begin{eqnarray}
    \left\|\EE\left[\tnabla \Psi(x,\Bxi^k)\right] - \nabla \psi(x)\right\|_2 &\le& \delta, \quad \forall x\in\R^n,\label{eq:bias_batched_stoch_grad}\\
    \EE\left[\exp\left(\frac{\left\|\tnabla \Psi(x,\Bxi^k) - \EE\left[\tnabla \Psi(x,\Bxi^k)\right]\right\|_2^2}{O(\nicefrac{\sigma_\psi^2}{r_k^2})}\right)\right] &\le& \exp(1), \quad \forall x\in\R^n,\label{eq:super_exp_moment_batched_stoch_grad}
\end{eqnarray}

In these settings, we consider a method called {\tt SPDSTM} (Stochastic Primal-Dual Similar Triangles Method, see Algorithm~\ref{Alg:PDSTM}). Note that Algorithm~4 from \cite{dvinskikh2019dual} is a special case of {\tt SPDSTM} when $\delta = 0$, i.e., stochastic gradient is unbiased, up to a factor $2$ in the choice of $\tL$.

\begin{algorithm}[h]
\caption{{\tt SPDSTM}}
\label{Alg:PDSTM}   
 \begin{algorithmic}[1]
\Require $\tilde{y}^0=z^0=y^0=0$, number of iterations $N$, $\alpha_0 = A_0=0$
\For{$k=0,\dots, N$}
\State Set $\tL = 2L_\psi$
\State Set $A_{k+1} = A_k + \alpha_{k+1}$, where $2\tL\alpha_{k+1}^2 = A_k + \alpha_{k+1}$
\State $\tilde{y}^{k+1} = \nicefrac{(A_ky^k+\alpha_{k+1}z^k)}{A_{k+1}}$
\State $z^{k+1} = z^k - \alpha_{k+1} \tnabla\Psi(\tilde{y}^{k+1},\Bxi^{k})$
\State $y^{k+1}=\nicefrac{(A_ky^k+\alpha_{k+1}z^{k+1})}{A_{k+1}}$
\EndFor
\Ensure    $y^N$, $\tilde{x}^N = \frac{1}{A_N}\sum_{k=0}^N \alpha_k \tx(A^\top\tilde{y}^k,\Bxi^k)$. 
\end{algorithmic}
 \end{algorithm}

Below we present the main convergence result of this section.
\begin{theorem}[Theorem 5.1 from \cite{gorbunov2019optimal}]\label{thm:spdtstm_smooth_cvx_dual_biased}
    Assume that $f$ is $\mu$-strongly convex and $\|\nabla f(x^*)\|_2 = M_f$. Let $\varepsilon > 0$ be a desired accuracy. Next, assume that $f$ is $L_f$-Lipschitz continuous on the ball $B_{R_f}(0)$ with 
    $$R_f = \tilde{\Omega}\left(\max\left\{\frac{R_y}{A_N\sqrt{\lambda_{\max}(A^\top A)}}, \frac{\sqrt{\lambda_{\max}(A^\top A)}R_y}{\mu}, R_x\right\}\right),$$ where $R_y$ is such that $\|y^*\|_2 \le R_y$, $y^*$ is the solution of the dual problem \eqref{DP}, and $R_x = \|x(A^\top y^*)\|_2$. Assume that at iteration k of Algorithm~\ref{Alg:PDSTM} batch size is chosen according to the formula $r_k \ge \max\left\{1, \frac{ \sigma^2_\psi \widetilde{\alpha}_k \ln(\nicefrac{N}{\beta})}{\hat C\e}\right\}$, where $\widetilde{\alpha}_{k} = \frac{k+1}{2\tL}$, $0 < \varepsilon \le \frac{H\tL R_0^2}{N^2}$, $0 \le \delta \le \frac{G\tL R_0}{(N+1)^2}$ and $N\ge 1$ for some numeric constant $H > 0$, $G > 0$ and $\hat C > 0$. Then with probability  $\geq 1-4\beta$, where $\beta \in \left(0,\nicefrac{1}{8}\right)$, after $N = \widetilde{O}\left(\sqrt{\frac{M_f}{\mu\e}\chi(A^\top A)} \right)$ iterations where $\chi(A^\top A) = \frac{\lambda_{\max}(A^\top A)}{\lambda_{\min}^+(A^\top A)}$, the outputs $\tx^N$ and $y^N$ of Algorithm \ref{Alg:PDSTM} satisfy the following condition
\begin{equation}
    f(\tilde{x}^N) -f(x^*) \le f(\tilde{x}^N) + \psi(y^N) \le \e, \quad \|A\tilde{x}^N\|_2 \le \frac{\e}{R_{y}} \label{eq:SPDSTM_non_str_cvx_guarantee}
\end{equation}
with probability at least $1-4\beta$. What is more, to guarantee \eqref{eq:SPDSTM_non_str_cvx_guarantee} with probability at least $1-4\beta$ Algorithm~\ref{Alg:PDSTM} requires
\begin{equation}
    \widetilde{O}\left(\max\left\{\frac{\sigma_x^2M_f^2}{\varepsilon^2}\chi(A^\top A)\ln\left(\frac{1}{\beta}\sqrt{\frac{M_f}{\mu\e}\chi(A^\top A)}\right), \sqrt{\frac{M_f}{\mu\e}\chi(A^\top A)}\right\}\right)\label{eq:SPDSTM_oracle_calls}
\end{equation}
calls of the biased stochastic oracle $\tnabla\psi(y,\xi)$, i.e.\ $\tx(y,\xi)$.
\end{theorem}

\subsubsection{Strongly Convex Dual Functions and Restarts Technique}\label{sec:restarts}
In this section, we assume that primal functional $f$ is additionally $L$-smooth. It implies that the dual function $\psi$ in \eqref{DP} is additionally $\mu_{\psi}$-strongly convex in $y^0 + (\text{Ker} A^\top)^{\perp}$ where $\mu_{\psi} = \nicefrac{\lambda_{\min}^{+}(A^\top A)}{L}$ \cite{kakade2009duality,Rockafellar2015} and $\lambda_{\min}^{+}(A^\top A)$ is the minimal positive eigenvalue of $A^\top A$.

From weak duality $-f(x^*)\le \psi(y^*)$ and \eqref{eq:dual_function} we get the key relation of this section (see also \cite{allen2018make,anikin2017dual,nesterov2012make})
\begin{equation}\label{eq:key_ineq_for_restarts}
f(x(A^\top y))- f(x^*) \le\langle\nabla \psi(y), y\rangle = \langle Ax(A^\top y), y\rangle.    
\end{equation}
This inequality implies the following theorem.
\begin{theorem}[Theorem 5.2 from \cite{gorbunov2019optimal}]\label{thm:grad_norm_testarts_motivation}
    Consider function $f$ and its dual function $\psi$ defined in \eqref{eq:dual_function} such that problems \eqref{PP} and \eqref{DP} have solutions. Assume that $y^N$ is such that $\|\nabla \psi(y^N)\|_2 \le \nicefrac{\e}{R_y}$ and $y^N \le 2R_y$, where $\e > 0$ is some positive number and $R_y = \|y^*\|_2$ where $y^*$ is any minimizer of $\psi$. Then for $x^N = x(A^\top y^N)$ following relations hold:
    \begin{equation}
        f(x^N) - f(x^*) \le 2\e,\quad \|Ax^N\|_2 \le \frac{\e}{R_y},\label{eq:consequence_of_small_grad_norm}
    \end{equation}
    where $x^*$ is any minimizer of $f$.
\end{theorem}

That is why, in this section we mainly focus on the methods that provide optimal convergence rates for the gradient norm. In particular, we consider Recursive Regularization Meta-Algorithm from (see Algorithm~\ref{Alg:RRMA-AC-SA}) \cite{foster2019complexity} with {\tt AC-SA$^2$} (see Algorithm~\ref{Alg:AC-SA2}) as a subroutine (i.e.\ {\tt RRMA-AC-SA$^2$}) which is based on {\tt AC-SA} algorithm (see Algorithm~\ref{Alg:AC-SA}) from \cite{ghadimi2012optimal}. We notice that {\tt RRMA-AC-SA$^2$} is applied for a regularized dual function
\begin{equation}
    \tilde\psi(y) = \psi(y) + \frac{\lambda}{2}\|y - y^0\|_2^2,\label{eq:regularized_dual_function}
\end{equation}
where $\lambda > 0$ is some positive number which will be defined further. Function $\tilde{\psi}$ is $\lambda$-strongly convex and $\tilde{L}_\psi$-smooth in $\R^n$ where $\tilde{L}_\psi = L_\psi + \lambda$. For now, we just assume w.l.o.g.\ that $\tilde{\psi}$ is $(\mu_\psi + \lambda)$-strongly convex in $\R^n$, but we will go back to this question further.

In this section we consider the same oracle as in Section~\ref{sec:conv_dual}, but we additionally assume that $\delta = 0$, i.e., stochastic first-order oracle is unbiased. To define batched version of the stochastic gradient we will use the following notation:
\begin{equation}
    \nabla\Psi(y,\Bxi^{t},r_t) = \frac{1}{r_t}\sum\limits_{l=1}^{r_t}\nabla \psi(y,\xi^{l}), \quad x(y,\Bxi^t,r_t) = \frac{1}{r_t}\sum\limits_{l=1}^{r_t}x(y,\xi^l).\label{eq:batched_biased_stoch_gradient_restarts}
\end{equation}
As before, in the cases when the batch-size $r_t$ can be restored from the context, we will use simplified notation $\nabla\Psi(y,\Bxi^{t})$ and $x(y,\Bxi^t)$. 
\begin{algorithm}[h]
\caption{{\tt RRMA-AC-SA$^2$} \cite{foster2019complexity}}
\label{Alg:RRMA-AC-SA}   
 \begin{algorithmic}[1]
\Require $y^0$~--- starting point, $m$~--- total number of iterations
\State $\psi_0 \leftarrow \tilde\psi$, $\hat y^0 \leftarrow y^0$, $T \leftarrow \left\lfloor\log_2\frac{\tilde L_\psi}{\lambda}\right\rfloor$
\For{$k=1,\ldots, T$}
\State Run {\tt AC-SA$^2$} for $\nicefrac{m}{T}$ iterations to optimize $\psi_{k-1}$ with $\hat y^{k-1}$ as a starting point and get the output $\hat y^k$ 
\State $\psi_k(y) \leftarrow \tilde\psi(y) + \lambda\sum_{l=1}^k 2^{l-1}\|y - \hat y^l\|_2^2$
\EndFor
\Ensure $\hat y^T$. 
\end{algorithmic}
\end{algorithm}

In the {\tt AC-SA} algorithm we use batched stochastic gradients of functions $\psi_k$ which are defined as follows:
\begin{eqnarray}
    \nabla\Psi_k(y,\Bxi^{t}) &=& \frac{1}{r_t}\sum\limits_{l=1}^{r_t}\nabla \psi_k(y,\xi^{l}),\label{eq:batch_stoch_regularized_func}\\
    \nabla\psi_k(y,\xi) &=& \nabla\psi(y,\xi) + \lambda(y - y^0) + \lambda\sum\limits_{l=1}^k2^l (y - \hat{y}^l).\notag
\end{eqnarray}

\begin{algorithm}[h]
\caption{{\tt AC-SA} \cite{ghadimi2012optimal}}
\label{Alg:AC-SA}   
 \begin{algorithmic}[1]
\Require $z^0$~--- starting point, $m$~--- number of iterations, $\psi_k$~--- objective function
\State $y^0_{ag} \leftarrow z^0$, $y^0_{md} \leftarrow z^0$
\For{$t=1,\ldots, m$}
\State $\alpha_t \leftarrow \frac{2}{t+1}$, $\gamma_t \leftarrow \frac{4\tilde L_\psi}{t(t+1)}$
\State $y^t_{md} \leftarrow \frac{(1-\alpha_t)(\lambda + \gamma_t)}{\gamma_t + (1-\alpha_t^2)\lambda}y^{t-1}_{ag} + \frac{\alpha_t((1-\alpha_t)\lambda + \gamma_t)}{\gamma_t + (1-\alpha_t^2)\lambda}z^{t-1}$
\State $z^t \leftarrow \frac{\alpha_t\lambda}{\lambda + \gamma_t}y^t_{md} + \frac{(1-\alpha_t)\lambda + \gamma_t}{\lambda + \gamma_t}z^{t-1} - \frac{\alpha_t}{\lambda + \gamma_t}\nabla\Psi_k(y^t_{md}, \Bxi^t)$
\State $y^t_{ag} \leftarrow \alpha_t z^t + (1-\alpha_t)x^{t-1}_{ag}$
\EndFor
\Ensure $y^m_{ag}$. 
\end{algorithmic}
\end{algorithm}

\begin{algorithm}[h]
\caption{{\tt AC-SA$^2$} \cite{foster2019complexity}}
\label{Alg:AC-SA2}   
 \begin{algorithmic}[1]
\Require $z^0$~--- starting point, $m$~--- number of iterations, $\psi_k$~--- objective function
\State Run {\tt AC-SA} for $\nicefrac{m}{2}$ iterations to optimize $\psi_{k}$ with $z^0$ as a starting point and get the output $y^1$ 
\State Run {\tt AC-SA} for $\nicefrac{m}{2}$ iterations to optimize $\psi_{k}$ with $y^1$ as a starting point and get the output $y^2$
\Ensure $y^2$. 
\end{algorithmic}
\end{algorithm}

The following theorem states the main result for {\tt RRMA-AC-SA$^2$} that we need in the section.
\begin{theorem}[Corollary~1 from \cite{foster2019complexity}]\label{thm:rrma-ac-sa2_convergence}
    Let $\psi$ be $L_\psi$-smooth and $\mu_\psi$-strongly convex function and $\lambda = \Theta\left(\nicefrac{(L_\psi \ln^2 N)}{N^2}\right)$ for some $N > 1$. If the Algorithm~\ref{Alg:RRMA-AC-SA} performs $N$ iterations in total\footnote{The overall number of performed iterations during the calls of {\tt AC-SA$^2$} equals $N$.} with batch size $r$ for all iterations, then it will provide such a point $\hat{y}$ that
    \begin{equation}
        \EE\left[\|\nabla\psi(\hat{y})\|_2^2\mid y^0, r\right] \le C\left(\frac{L_\psi^2\|y^0 - y^*\|_2^2\ln^{4}N}{N^4} + \frac{\sigma_\psi^2\ln^6N}{rN}\right),\label{eq:rrma-ac-sa2_guarantee}
    \end{equation}
    where $C>0$ is some positive constant and $y^*$ is a solution of the dual problem \eqref{DP}.
\end{theorem}

The following result shows that w.l.o.g.\ we can assume that function $\psi$ defined in \eqref{eq:dual_function} is $\mu_\psi$-strongly convex everywhere with $\mu_\psi = \nicefrac{\lambda_{\min}^+(A^\top A)}{L}$. In fact, from $L$-smoothness of $f$ we have only that $\psi$ is $\mu_\psi$-strongly convex in $y^0 + \left(\text{Ker}(A^\top)\right)^\perp$ (see \cite{kakade2009duality,Rockafellar2015} for the details). However, the structure of the considered here methods is such that all points generated by the {\tt RRMA-AC-SA$^2$} and, in particular, {\tt AC-SA} lie in $y^0 + \left(\text{Ker}(A^\top)\right)^\perp$.

\begin{theorem}[Theorem 5.4 from \cite{gorbunov2019optimal}]\label{thm:ac-sa_points}
    Assume that Algorithm~\ref{Alg:AC-SA} is run for the objective $\psi_k(y) = \tilde\psi(y) + \lambda\sum_{l=1}^k 2^{l-1}\|y - \hat y^l\|_2^2$ with $z^0$ as a starting point, where $z^0,\hat y^1,\ldots,\hat y^k$ are some points from $y^0 + \left(\text{Ker}(A^\top)\right)^\perp$ and $y^0\in\R^n$. Then for all $t\ge 0$ we have $y_{md}^t, z^t, y_{ag}^t \in y^0 + \left(\text{Ker}(A^\top)\right)^\perp$.
\end{theorem}
\begin{corollary}[Corollary 5.5 from \cite{gorbunov2019optimal}]\label{cor:rrma-ac-sa2_points}
    Assume that Algorithm~\ref{Alg:RRMA-AC-SA} is run for the objective $\psi_k(y) = \tilde\psi(y) + \lambda\sum_{l=1}^k 2^{l-1}\|y - \hat y^l\|_2^2$ with $y^0$ as a starting point. Then for all $k\ge 0$ we have $\hat y^k \in y^0 + \left(\text{Ker}(A^\top)\right)^\perp$.
\end{corollary}

Now we are ready to present our approach\footnote{This approach was described in \cite{dvinskikh2019decentralized} and formally proved in \cite{gorbunov2019optimal}.} of constructing an accelerated method for the strongly convex dual problem using restarts of {\tt RRMA-AC-SA$^2$}. To explain the main idea we start with the simplest case: $\sigma_\psi^2 = 0$, $r = 0$. It means that there is no stochasticity in the method and the bound \eqref{eq:rrma-ac-sa2_guarantee} can be rewritten in the following form:
\begin{equation}
    \|\nabla \psi(\hat{y})\|_2 \le \frac{\sqrt{C}L_\psi\|y^0 - y^*\|_2\ln^2 N}{N^2} \le \frac{\sqrt{C}L_\psi\|\nabla\psi(y^0)\|_2\ln^2 N}{\mu_\psi N^2},\label{eq:rrma-ac-sa2_guarantee_simplified}
\end{equation}
where we used inequality $\|\nabla \psi(y^0)\| \ge \mu_\psi\|y^0 - y^*\|$ which follows from the $\mu_\psi$-strong convexity of $\psi$. It implies that after $\bar{N} = \tilde{O}(\sqrt{\nicefrac{L_\psi}{\mu_\psi}})$ iterations of {\tt RRMA-AC-SA$^2$} the method returns such $\bar{y}^1 = \hat{y}$ that $\|\nabla\psi(\bar{y}^1)\|_2 \le \frac{1}{2}\|\nabla\psi(y^0)\|_2$. Next, applying {\tt RRMA-AC-SA$^2$} with $\bar{y}^1$ as a starting point for the same number of iterations we will get new point $\bar{y}^2$ such that $\|\nabla\psi(\bar{y}^2)\|_2 \le \frac{1}{2}\|\nabla\psi(\bar{y}^1)\|_2 \le \frac{1}{4}\|\nabla\psi(y^0)\|_2$. Then, after $l = O(\ln(\nicefrac{R_y\|\nabla\psi(y^0)\|_2}{\e}))$ of such restarts we can get the point $\bar{y}^l$ such that $\|\nabla\psi(\bar y^l)\|_2 \le \nicefrac{\e}{R_y}$ with total number of gradients computations $\bar{N}l = \tilde{O}\left(\sqrt{\nicefrac{L_\psi}{\mu_\psi}}\ln(\nicefrac{R_y\|\nabla\psi(y^0)\|_2}{\e})\right)$.

When $\sigma_\psi^2 \neq 0$ we need to modify this approach. The first ingredient to handle the stochasticity is large enough batch size for the $l$-th restart: $r_l$ should be $\Omega\left(\nicefrac{\sigma_\psi^2}{(\bar{N}\|\nabla \psi(\bar{y}^{l-1})\|_2^2)}\right)$. However, in the stochastic case we do not have an access to the $\nabla \psi(\bar y^{l-1})$, so, such batch size is impractical. One possible way to fix this issue is to independently sample large enough number $\hat{r}_l \sim \nicefrac{R_y^2}{\e^2}$ of stochastic gradients additionally, which is the second ingredient of our approach, in order to get good enough approximation $\nabla\Psi(\bar{y}^{l-1},\Bxi^{l-1},\hat{r}_l)$ of $\nabla \psi(\bar y^{l-1})$ and use the norm of such an approximation which is close to the norm of the true gradient with big enough probability in order to estimate needed batch size $r^l$ for the optimization procedure. Using this, we can get the bound of the following form:
\begin{eqnarray*}
    \EE\left[\|\nabla \psi(\bar y^l)\|_2^2 \mid \bar y^{l-1}, r_l, \hat{r}_l\right] \le A_l &\eqdef& \frac{\|\nabla \psi(\bar y^{l-1})\|_2^2}{8}+ \frac{\|\nabla\Psi(\bar{y}^{l-1},\Bxi^{l-1},\hat{r}_l)-\nabla \psi(\bar y^{l-1})\|_2^2}{32}.
\end{eqnarray*}
The third ingredient is the amplification trick: we run $p_l = \Omega(\ln(\nicefrac{1}{\beta}))$ independent trajectories of {\tt RRMA-AC-SA$^2$}, get points $\bar{y}^{l,1},\ldots, \bar{y}^{l,p_l}$ and choose such $\bar{y}^{l,p(l)}$ among of them that $\|\nabla \psi(\bar{y}^{l,p(l)})\|_2$ is \textit{close enough} to $\min_{p=1,\ldots,p_l}\|\nabla \psi(\bar{y}^{l, p})\|_2$ with high probability, i.e., $\|\nabla \psi(\bar{y}^{l,p(l)})\|_2^2 \le 2\min_{p=1,\ldots,p_l}\|\nabla \psi(\bar{y}^{l, p})\|_2^2 + \nicefrac{\e^2}{8R_y^2}$ with probability at least $1-\beta$ for fixed $\nabla\Psi(\bar{y}^{l-1},\Bxi^{l-1},\hat{r}_l)$. We achieve it due to additional sampling of $\bar{r}_l \sim \nicefrac{R_y^2}{\e^2}$ stochastic gradients at $\bar{y}^{l,p}$ for each trajectory and choosing such $p(l)$ corresponding to the smallest norm of the obtained batched stochastic gradient. By Markov's inequality for all $p=1,\ldots,p_l$
\begin{equation*}
	\PP\left\{\|\nabla\psi(\bar{y}^{l,p})\|_2^2 \ge 2A_l\mid \bar y^{l-1}, r_l, \bar{r}_l\right\} \le \frac{1}{2},
\end{equation*}
hence
\begin{equation*}
	\PP\left\{\min_{p=1,\ldots,p_l}\|\nabla \psi(\bar{y}^{l, p})\|_2^2 \ge 2A_l\mid \bar y^{l-1}, r_l, \bar{r}_l\right\} \le \frac{1}{2^{p_l}}.
\end{equation*}
That is, for $p_l = \log_2(\nicefrac{1}{\beta})$ we have that with probability at least $1-2\beta$
\begin{equation*}
	\|\nabla\psi(\bar{y}^{l,p(l)})\|_2^2 \le \frac{\|\nabla \psi(\bar y^{l-1})\|_2^2}{2} + \frac{\|\nabla\Psi(\bar{y}^{l-1},\Bxi^{l-1},\hat{r}_l)-\nabla \psi(\bar y^{l-1})\|_2^2}{8} + \frac{\e^2}{8R_y^2}
\end{equation*}
for fixed $\nabla\Psi(\bar{y}^{l-1},\Bxi^{l-1},\hat{r}_l)$ which means that
\begin{equation*}
	\|\nabla\psi(\bar{y}^{l,p(l)})\|_2^2 \le \frac{\|\nabla \psi(\bar y^{l-1})\|_2^2}{2} + \frac{\e^2}{4R_y^2}
\end{equation*}
with probability at least $1 - 3\beta$. Therefore, after $l = \log_2(\nicefrac{2R_y^2\|\nabla\psi(y^0)\|_2^2}{\e^2})$ of such restarts our method provides the point $\bar{y}^{l,p(l)}$ such that with probability at least $1 - 3l\beta$
\begin{eqnarray*}
	\|\nabla \psi(\bar{y}^{l,p(l)})\|_2^2 &\le& \frac{\|\nabla\psi(y^{0})\|_2^2}{2^{l}} + \frac{\e^2}{4R_y^2}\sum\limits_{k=0}^{l-1}2^{-k} \le \frac{\e^2}{2R_y^2} + \frac{\e^2}{4R_y^2}\cdot 2 = \frac{\e^2}{R_y^2}.
\end{eqnarray*}

The approach informally described above is stated as Algorithm~\ref{Alg:Restarted-RRMA-AC-SA2}.
\begin{algorithm}[h]
\caption{{\tt Restarted-RRMA-AC-SA$^2$}}
\label{Alg:Restarted-RRMA-AC-SA2}   
 \begin{algorithmic}[1]
\Require $y^0$~--- starting point, $l$~--- number of restarts, $\{\hat r_k\}_{k=1}^l$, $\{\bar{r}_k\}_{k=1}^l$~--- batch-sizes, $\{p_k\}_{k=1}^{l}$~--- amplification parameters
\State Choose the smallest integer $\bar{N} > 1$ such that $\frac{CL_\psi^2\ln^4\bar{N}}{\mu_\psi^2\bar{N}^4} \le \frac{1}{32}$
\State $\bar{y}^{0,p(0)} \leftarrow y^0$ 
\For{$k=1,\ldots,l$}
\State Compute $\nabla\Psi(\bar{y}^{k-1,p(k-1)},\Bxi^{k-1,p(k-1)}, \hat{r}_k)$
\State $r_k \leftarrow \max\left\{1, \frac{64C\sigma_\psi^2\ln^6\bar{N}}{\bar{N}\|\nabla\Psi(\bar{y}^{k-1,p(k-1)},\Bxi^{k-1,p(k-1)}, \hat{r}_k)\|_2^2}\right\}$
\State Run $p_k$ independent trajectories of {\tt RRMA-AC-SA$^2$} for $\bar{N}$ iterations with batch-size $r_k$ with $\bar{y}^{k-1,p(k-1)}$ as a starting point and get outputs $\bar{y}^{k,1},\ldots,\bar{y}^{k,p_k}$
\State Compute $\nabla\Psi(\bar{y}^{k,1},\Bxi^{k,1}, \bar{r}_k),\ldots, \nabla\Psi(\bar{y}^{k,p_k},\Bxi^{k,p_k}, \bar{r}_k)$
\State $p(k) \leftarrow \argmin_{p=1,\ldots,p_k}\|\nabla\Psi(\bar{y}^{k,p},\Bxi^{k,p}, \bar{r}_k)\|_2$
\EndFor
\Ensure $\bar{y}^{l,p(l)}$. 
\end{algorithmic}
\end{algorithm}

\begin{theorem}[Theorem 5.6 from \cite{gorbunov2019optimal}]\label{thm:restarted-rrma-ac-sa2_convergence}
	Assume that $\psi$ is $\mu_\psi$-strongly convex and $L_\psi$-smooth. If Algorithm~\ref{Alg:Restarted-RRMA-AC-SA2} is run with 
	\begin{eqnarray}
	   l &=& \max\left\{1,\log_2\frac{2R_y^2\|\nabla\psi(y^0)\|_2^2}{\e^2}\right\}\notag\\
	   \hat{r}_k &=& \max\left\{1, \frac{4\sigma_\psi^2\left(1 + \sqrt{3\ln\frac{l}{\beta}}\right)^2R_y^2}{\e^2}\right\},\quad r_k = \max\left\{1, \frac{64C\sigma_\psi^2\ln^6\bar{N}}{\bar{N}\|\nabla\Psi(\bar{y}^{k-1,p(k-1)},\Bxi^{k-1,p(k-1)}, \hat{r}_k)\|_2^2}\right\},\notag\\ 
	   p_k &=& \max\left\{1, \log_2\frac{l}{\beta}\right\},\quad \bar{r}_k = \max\left\{1,\frac{128\sigma_\psi^2\left(1 + \sqrt{3\ln\frac{lp_k}{\beta}}\right)^2R_y^2}{\e^2}\right\}\label{eq:r-rrma-ac-sa2_params}
	\end{eqnarray}
	for all $k=1,\ldots,l$ where $\bar{N} > 1$ is such that $\frac{CL_\psi^2\ln^4\bar{N}}{\mu_\psi^2\bar{N}^4} \le \frac{1}{32}$, $\beta \in (0,\nicefrac{1}{3})$ and $\e > 0$, then with probability at least $1 - 3\beta$
	\begin{equation}
		\|\nabla \psi(\bar{y}^{l,p(l)})\|_2 \le \frac{\e}{R_y}\label{eq:restarted-rrma-ac-sa2_grad_norm}
	\end{equation}
	and the total number of the oracle calls equals
	\begin{equation}
	    \sum\limits_{k=1}^{l}(\hat{r}_k + \bar{N}p_kr_k + p_k\bar{r}_k) = \widetilde{O}\left(\max\left\{\sqrt{\frac{L_\psi}{\mu_\psi}}, \frac{\sigma_\psi^2R_y^2}{\e^2}\right\}\right).\label{eq:restarted_number_of_oracle_calls}
	\end{equation}
\end{theorem}

\begin{corollary}[Corollary 5.7 from \cite{gorbunov2019optimal}]\label{cor:r-rrma-ac-sa2_norm}
	Under assumptions of Theorem~\ref{thm:restarted-rrma-ac-sa2_convergence} we get that with probability at least $1 - 3\beta$
	\begin{equation}
		\|\bar{y}^{l,p(l)} - y^*\|_2 \le \frac{\e}{\mu_\psi R_y},\label{eq:restarted_norm_difference}
	\end{equation}
	where $\beta\in(0,\nicefrac{1}{3})$ the total number of the oracle calls is defined in \eqref{eq:restarted_number_of_oracle_calls}.	
\end{corollary}

Now we are ready to present convergence guarantees for the primal function and variables.
\begin{corollary}[Corollary 5.8 from \cite{gorbunov2019optimal}]\label{cor:r-rrma-ac-sa2_connect_with_primal}
	Let the assumptions of Theorem~\ref{thm:restarted-rrma-ac-sa2_convergence} hold. Assume that $f$ is $L_f$-Lipschitz continuous on $B_{R_f}(0)$ where $$R_f = \left(\frac{\mu_\psi}{8\sqrt{\lambda_{\max}(A^\top A)}} + \frac{\sqrt{\lambda_{\max}(A^\top A)}}{\mu} + \frac{R_x}{R_y}\right)R_y$$ and $R_x = \|x(A^\top y^*)\|_2$. Then, with probability at least $1 - 4\beta$
	\begin{equation}
		f(x^l) - f(x^*) \le \left(2 + \frac{L_f}{8R_y\sqrt{\lambda_{\max}(A^\top A)}}\right)\e,\quad \|Ax^l\|_2 \le \frac{9\e}{8R_y},\label{eq:restarted_primal_guarantees}
	\end{equation}
	where $\beta\in(0,\nicefrac{1}{4})$, $\e\in(0,\mu_\psi R_y^2)$ $x^l \eqdef x(A^\top\bar{y}^{l,p(l)},\Bxi^{l,p(l)}, \bar{r}_l)$ and to achieve it we need the following number of oracle calls:
	\begin{equation}
	    \sum\limits_{k=1}^{l}(\hat{r}_k + \bar{N}p_kr_k + p_k\bar{r}_k) = \widetilde{O}\left(\max\left\{\sqrt{\frac{L}{\mu}\chi(A^\top A)}, \frac{\sigma_x^2M^2}{\e^2}\chi(A^\top A)\right\}\right)\label{eq:restarted_number_of_oracle_calls_connextion_with_primal}
	\end{equation}
	where $M = \|\nabla f(x^*)\|_2$.
\end{corollary}

\subsubsection{Direct Acceleration for Strongly Convex Dual Function}\label{sec:str_cvx_dual}
First of all, we consider the following minimization problem:
\begin{equation}
    \min_{y\in\R^n}\psi(y),\label{eq:main_problem_str_cvx}
\end{equation}
where $\psi(y)$ is $\mu_\psi$-strongly convex and $L_\psi$-smooth. We use the same notation to define the objective in \eqref{eq:main_problem_str_cvx} as for the dual function from \eqref{DP} because later in the section we apply the algorithm introduced below to the \eqref{DP}, but for now it is not important that $\psi$ is a dual function for \eqref{PP} and we prefer to consider more general situation. As in Section~\ref{sec:conv_dual}, we do not assume that we have an access to the exact gradient of $\psi(y)$ and consider instead of it biased stochastic gradient $\tnabla\psi(y,\xi)$ satisfying inequalities \eqref{eq:bias_stoch_grad} and \eqref{eq:super_exp_moment_stoch_grad} with $\delta \ge 0$ and $\sigma_\psi \ge 0$. In the main method of this section batched version of the stochastic gradient is used:
\begin{equation}
    \tnabla\Psi(y,\Bxi^{k}) = \frac{1}{r_k}\sum\limits_{l=1}^{r_k}\tnabla \psi(y,\xi^{l}),\label{eq:batched_stoch_grad_str_cvx}
\end{equation}
where $r_k$ is the batch-size that we leave unspecified for now. Note that $\tnabla\Psi(y,\Bxi^{k})$ satisfies inequalities \eqref{eq:bias_batched_stoch_grad} and \eqref{eq:super_exp_moment_batched_stoch_grad}.

We use Stochastic Similar Triangles Method which is stated in this section as Algorithm~\ref{Alg:STM_str_cvx} to solve problem \eqref{eq:main_problem_str_cvx}. To define the iterate $z^{k+1}$ we use the following sequence of functions:
\begin{eqnarray}
    \tilde{g}_{0}(z) &\eqdef& \frac{1}{2}\|z-z^0\|_2^2 + \alpha_0\left(\psi(y^0) + \la\tnabla \Psi(y^0,\Bxi^0), z - y^0\ra + \frac{\mu_\psi}{2}\|z - y^{0}\|_2^2\right),\notag\\
    \tilde{g}_{k+1}(z) &\eqdef& \tilde{g}_{k}(z) + \alpha_{k+1}\Big(\psi(\tilde{y}^{k+1}) + \la\tnabla \Psi(\tilde{y}^{k+1},\Bxi^{k+1}),z - \tilde{y}^{k+1}\ra+ \frac{\mu_\psi}{2}\|z - \tilde{y}^{k+1}\|_2^2\Big)\notag\\
    &=& \frac{1}{2}\|z - z^0\|_2^2 + \sum\limits_{l=0}^{k+1}\alpha_{l}\left(\psi(\tilde{y}^{l}) + \la\tnabla \Psi(\tilde{y}^{l},\Bxi^{l}),z - \tilde{y}^{l}\ra + \frac{\mu_\psi}{2}\|z - \tilde{y}^{l}\|_2^2\right)\label{eq:g_k+1_sequence_str_cvx}
\end{eqnarray}
We notice that $\tg_{k}(z)$ is $(1+A_k\mu_\psi)$-strongly convex.

\begin{algorithm}[h]
\caption{Stochastic Similar Triangles Methods for strongly convex problems ({\tt SSTM{\_}sc})}
\label{Alg:STM_str_cvx}   
 \begin{algorithmic}[1]
\Require $\tilde{y}^0 = z^0 = y^0$~--- starting point, $N$~--- number of iterations
\State Set $\alpha_0 = A_0 = \nicefrac{1}{L_\psi}$
\State Get $\tnabla\Psi(y^0,\Bxi^0)$ to define $\tg_0(z)$
\For{$k=0,1,\ldots, N-1$}
\State Choose $\alpha_{k+1}$ such that $A_{k+1} = A_k + \alpha_{k+1}$, $A_{k+1}(1+A_k\mu_\psi) = \alpha_{k+1}^2L_\psi$
\State $\tilde{y}^{k+1} = \nicefrac{(A_ky^k+\alpha_{k+1}z^k)}{A_{k+1}}$
\State $z^{k+1} = \argmin_{z\in \R^n} \tilde{g}_{k+1}(z)$, where $\tilde{g}_{k+1}(z)$ is defined in \eqref{eq:g_k+1_sequence_str_cvx}
\State $y^{k+1} = \nicefrac{(A_ky^k+\alpha_{k+1}z^{k+1})}{A_{k+1}}$
\EndFor
\Ensure    $ x^N$ 
\end{algorithmic}
 \end{algorithm}

For this algorithm we have the following convergence result.
\begin{theorem}[Theorem 5.11 from \cite{gorbunov2019optimal}]\label{thm:str_cvx_biased_main_result}
    Assume that the function $\psi$ is $\mu_\psi$-strongly convex and $L_\psi$-smooth, 
    $$
    r_k = \Theta\left(\max\left\{1, \left(\frac{\mu_\psi}{L_\psi}\right)^{\nicefrac{3}{2}}\frac{N^2\sigma_\psi^2\ln\frac{N}{\beta}}{\e}\right\}\right),
    $$ 
    i.e. $r_k \ge \frac{1}{C}\max\left\{1,\left(\frac{\mu_\psi}{L_\psi}\right)^{\nicefrac{3}{2}}\frac{N^2\sigma_\psi^2\left(1+\sqrt{3\ln\frac{N}{\beta}}\right)^2}{\e}\right\}$ with positive constants $C > 0$, $\e>0$ and $N \ge 1$. If additionally $\delta \le \frac{GR_0}{N\sqrt{A_N}}$ and $\e \le \frac{HR_0^2}{A_N}$ where $R_0 = \|y^* - y^0\|_2$ and Algorithm~\ref{Alg:STM_str_cvx} is run for $N$ iterations, then with probability at least $1-3\beta$
    \begin{equation}
    	\|y^N - y^*\|_2^2 \le \frac{\hat{J}^2R_0^2}{A_N},
    \end{equation}
    where $\beta\in(0,\nicefrac{1}{3})$, $$\hat{g}(N) = \frac{\ln\left(\frac{N}{\beta}\right) + \ln\ln\left(\frac{\hat{B}}{b}\right)}{\left(1+\sqrt{3\ln\left(\frac{N}{\beta}\right)}\right)^2},\quad b = \frac{2\sigma_1^2\alpha_{1}^2R_0^2}{r_1},\quad D = 1+\frac{\mu_\psi}{L_\psi} + \sqrt{1+\frac{\mu_\psi}{L_\psi}},$$
	\begin{eqnarray*}
	    \hat{B} &=& 8H C\left(\frac{L_\psi}{\mu_\psi}\right)^{\nicefrac{3}{2}}DR_0^4\left(N\left(\frac{3}{2}\right)^N + 1\right)\left(\hat{A} + 2Dh^2G^2+ 2C\left(\frac{L_\psi}{\mu_\psi}\right)^{\nicefrac{3}{2}}\left(c+2Du^2\right)H\right),
	\end{eqnarray*}
	$$h = u = \frac{2}{\mu_\psi},\quad c = \frac{2}{\mu_\psi^2},$$
	$$
	\hat{A} = \frac{1}{\mu_\psi} + \frac{2G }{L_\psi\mu_\psi N\sqrt{A_N}} + \frac{2G^2}{\mu_\psi^2 N^2} + \left(\frac{L_\psi}{\mu_\psi}\right)^{\nicefrac{3}{4}}\frac{2\sqrt{2CH}}{L_\psi\mu_\psi N\sqrt{A_N}} + \left(\frac{L_\psi}{\mu_\psi}\right)^{\nicefrac{3}{2}}\frac{4CH}{L_\psi\mu_\psi^2N^2A_N},
	$$
$$\hat{J} = \max\left\{\sqrt{\frac{1}{L_\psi}}, \frac{3\hat{B}_1D + \sqrt{9\hat{B}_1^2D^2 + 4\hat{A}+8cHC\left(\frac{L_\psi}{\mu_\psi}\right)^{\nicefrac{3}{2}}}}{2}\right\},\quad \hat{B}_1 = hG + uC_1\sqrt{2HC\left(\frac{L_\psi}{\mu_\psi}\right)^{\nicefrac{3}{2}}\hat{g}(N)}$$
and $C_1$ is some positive constant.
In other words, to achieve $\|y^N - y^*\|_2^2 \le \e$ with probability at least $1-3\beta$ Algorithm~\ref{Alg:STM_str_cvx} needs $N = \widetilde{O}\left(\sqrt{\frac{L_\psi}{\mu_\psi}}\right)$ iterations and $\widetilde{O}\left(\max\left\{\sqrt{\frac{L_\psi}{\mu_\psi}},\frac{\sigma_\psi^2}{\e}\right\}\right)$ oracle calls where $\widetilde{O}(\cdot)$ hides polylogarithmic factors depending on $L_\psi, \mu_\psi, R_0, \e$ and $\beta$.
\end{theorem}

Next, we apply the {\tt SSTM{\_}sc} to the problem \eqref{DP} when the objective of the primal problem \eqref{PP} is $L$-smooth, $\mu$-strongly convex and $L_f$-Lipschitz continuous on some ball which will be specified next, i.e., we consider the same setup as in Section~\ref{sec:conv_dual} but we additionally assume that the primal functional $f$ has $L$-Lipschitz continuous gradient. As in Section~\ref{sec:conv_dual} we also consider the case when the gradient of the dual functional is known only through biased stochastic estimators, see \eqref{eq:dual_stoch_func}--\eqref{eq:super_exp_moment_batched_stoch_grad} and the paragraphs containing these formulas.

In Section~\ref{sec:conv_dual} and \ref{sec:restarts} we mentioned that in the considered case dual function $\psi$ is $L_\psi$-smooth on $\R^n$ and $\mu_\psi$-strongly convex on $y^0 + (\text{Ker} A^\top)^{\perp}$ where $L_\psi = \nicefrac{\lambda_{\max}(A^\top A)}{\mu}$ and $\mu_\psi = \nicefrac{\lambda_{\min}^+(A^\top A)}{L}$. Using the same technique as in the proof of Theorem~\ref{thm:ac-sa_points} we show next that w.l.o.g.\ one can assume that $\psi$ is $\mu_\psi$-strongly convex on $\R^n$ since $\tnabla\Psi(y,\Bxi^k)$ lies in $\text{Im}A = (\text{Ker} A^\top)^{\perp}$ by definition of $\tnabla\Psi(y,\Bxi^k)$. For this purposes we need the explicit formula for $z^{k+1}$ which follows from the equation $\nabla \tg_{k+1}(z^{k+1}) = 0$:
\begin{equation}
    z^{k+1} = \frac{z^0}{1+A_{k+1}\mu_\psi} + \sum\limits_{l=0}^{k+1}\frac{\alpha_l\mu_\psi}{1+A_{k+1}\mu_\psi}\ty^l - \frac{1}{1+A_{k+1}\mu_\psi}\sum\limits_{l=0}^{k+1}\alpha_l\tnabla\Psi(\ty^l,\Bxi^l).\label{eq:z^k+1_str_cvx_explicit}
\end{equation}
\begin{theorem}[Theorem 5.12 from \cite{gorbunov2019optimal}]\label{thm:sstm_str_cvx_points}
    For all $k\ge 0$ we have that the iterates of Algorithm~\ref{Alg:STM_str_cvx} $\ty^k, z^k, y^k$ lie in $y^0 + \left(\text{Ker}(A^\top)\right)^\perp$.
\end{theorem}

This theorem makes it possible to apply the result from Theorem~\ref{thm:str_cvx_biased_main_result} for {\tt SSTM{\_}sc} which is run on the problem \eqref{DP}.
\begin{corollary}[Corollary 5.13 from \cite{gorbunov2019optimal}]\label{cor:radius_grad_norm_guarantee_str_cvx}
Under assumptions of Theorem~\ref{thm:str_cvx_biased_main_result} we get that after $N = \widetilde{O}\left(\sqrt{\frac{L_\psi}{\mu_\psi}}\ln\frac{1}{\e}\right)$ iterations of Algorithm~\ref{Alg:STM_str_cvx} which is run on the problem \eqref{DP} with probability at least $1-3\beta$
    \begin{equation}
        \|\nabla \psi(y^N)\|_2 \le \frac{\e}{R_y}, \label{eq:grad_norm_str_cvx}
    \end{equation}
    where $\beta \in\left(0,\nicefrac{1}{3}\right)$ and the total number of oracles calls equals
    \begin{equation}
        \widetilde{O}\left(\max\left\{\sqrt{\frac{L_\psi}{\mu_\psi}},\frac{\sigma_\psi^2R_y^2}{\e^2}\right\}\right). \label{eq:radius_grad_norm_oracle_calls_str_cvx}
    \end{equation}
    If additionally $\e \le \mu_\psi R_y^2$, then with probability at least $1-3\beta$
    \begin{eqnarray}
        \|y^N - y^*\|_2 &\le& \frac{\e}{\mu_\psi R_y},\label{eq:radius_grad_norm_main_str_cvx}\\
        \|y^N\|_2 &\le& 2R_y\label{eq:radius_dual_str_cvx}
    \end{eqnarray}
\end{corollary}

\begin{corollary}[Corollary 5.14 from \cite{gorbunov2019optimal}]\label{cor:sstm_str_cvx_connect_with_primal}
	Let the assumptions of Theorem~\ref{thm:str_cvx_biased_main_result} hold. Assume that $f$ is $L_f$-Lipschitz continuous on $B_{R_f}(0)$ where 
	$$R_f = \left(\sqrt{\frac{2C}{\lambda_{\max}(A^\top A)}} + G_1 + \frac{\sqrt{\lambda_{\max}(A^\top A)}}{\mu}\right)\frac{\e}{R_y} + R_x,$$
	$R_x = \|x(A^\top y^*)\|_2$, $\e \le \mu_\psi R_y^2$ and $\delta_y \le \frac{G_1\e}{NR_y}$ for some positive constant $G_1$. Assume additionally that the last batch-size $r_N$ is slightly bigger than other batch-sizes, i.e.\
	\begin{eqnarray}
	   r_N &\ge& \frac{1}{C}\max\left\{1,\left(\frac{\mu_\psi}{L_\psi}\right)^{\nicefrac{3}{2}}\frac{N^2\sigma_\psi^2\left(1+\sqrt{3\ln\frac{N}{\beta}}\right)^2R_y^2}{\e^2},\frac{\sigma_\psi^2\left(1+\sqrt{3\ln\frac{N}{\beta}}\right)^2R_y^2}{\e^2}\right\}.\label{eq:sstm_sc_last_batch} 
	\end{eqnarray}
	Then, with probability at least $1 - 4\beta$
	\begin{eqnarray}
	    f(\tx^N) - f(x^*) &\le& \left(2 + \left(\sqrt{\frac{2C}{\lambda_{\max}(A^\top A)}} + G_1\right)\frac{L_f}{R_y}\right)\e,\label{eq:stm_str_cvx_primal_guarantees_func}\\
	    \|A\tx^N\|_2 &\le& \left(1 + \sqrt{2C} + G_1\sqrt{\lambda_{\max}(A^\top A)}\right)\frac{\e}{R_y},\label{eq:stm_str_cvx_primal_guarantees_norm}
	\end{eqnarray}
	where $\beta\in(0,\nicefrac{1}{4})$, $\tilde{x}^N \eqdef \tilde{x}(A^\top y^{N},\Bxi^{N}, r_N)$ and to achieve it we need the total number of oracle calls including the cost of computing $\tilde{x}^N$ equals 
	\begin{equation}
        \widetilde{O}\left(\max\left\{\sqrt{\frac{L}{\mu}\chi(A^\top A)},\frac{\sigma_x^2M^2}{\e^2}\chi(A^\top A)\right\}\right) \label{eq:primal_oracle_calls_str_cvx}
    \end{equation}
    where $M = \|\nabla f(x^*)\|_2$.
\end{corollary}

\subsection{Applications to Decentralized Distributed Optimization}\label{sec:distributed_opt}

In this section, we apply our results to the decentralized optimization problems. First of all, we want to add additional motivation to the problem we are focusing on. As it was stated in the introductory part of this work, we are interested in the convex optimization problem
\begin{equation}
     \min\limits_{x\in Q\subseteq \R^n}f(x),\label{eq:main_problem}
\end{equation}
where $f$ is a convex function and $Q$ is closed and convex subset of $\R^n$. More precisely, we study particular case of \eqref{eq:main_problem} when the objective function $f$ could be represented as a mathematical expectation
\begin{equation}
    f(x) = \EE_\xi\left[f(x,\xi)\right]\label{eq:objectve_expectation},
\end{equation}
where $\xi$ is a random variable. Typically $x$ represents the feature vector defining the model, only samples of $\xi$ are available and the distribution of $\xi$ is unknown. One possible way to minimize generalization error \eqref{eq:objectve_expectation} is to solve empirical risk minimization or finite-sum minimization problem instead, i.e., solve \eqref{eq:main_problem} with the objective
\begin{equation}
    \hat f(x) = \frac{1}{m}\sum\limits_{i=1}^m f(x,\xi_i),\label{eq:erm_problem}
\end{equation}
where $m$ should be sufficiently large to approximate the initial problem. Indeed, if $f(x,\xi)$ is convex and $M$-Lipschitz continuous for all $\xi$, $Q$ has finite diameter $D$ and $\hat x = \argmin_{x\in Q}\hat f(x)$, then (see \cite{cesa-bianchi2002generalization,shalev2009stochastic}) with probability at least $1-\beta$
\begin{equation}
    f(\hat x) - \min\limits_{x\in Q} f(x) = O\left(\sqrt{\frac{M^2D^2n\ln(m)\ln\left(\nicefrac{n}{\beta}\right)}{m}}\right),\label{eq:convex_erm_argmin_property}
\end{equation}
and if additionally $f(x,\xi)$ is $\mu$-strongly convex for all $\xi$, then (see \cite{feldman2019high}) with probability at least $1-\beta$
\begin{equation}
    f(\hat x) - \min\limits_{x\in Q} f(x) = O\left(\frac{M^2D^2\ln(m)\ln\left(\nicefrac{m}{\beta}\right)}{\mu m} + \sqrt{\frac{M^2D^2\ln\left(\nicefrac{1}{\beta}\right)}{m}}\right).\label{eq:str_convex_erm_argmin_property}
\end{equation}
In other words, to solve \eqref{eq:main_problem}+\eqref{eq:objectve_expectation} with $\varepsilon$ functional accuracy via minimization of empirical risk \eqref{eq:erm_problem} it is needed to have $m = \widetilde{\Omega}\left(\nicefrac{M^2D^2n}{\varepsilon^2}\right)$ in the convex case and $m = \widetilde{\Omega}\left(\max\left\{\nicefrac{M^2D^2}{\mu\varepsilon},\nicefrac{M^2D^2}{\varepsilon^2}\right\}\right)$ in the $\mu$-strongly convex case where $\widetilde{\Omega}(\cdot)$ hides a constant factor, a logarithmic factor of $\nicefrac{1}{\beta}$ and a polylogarithmic factor of $\nicefrac{1}{\e}$.

Stochastic first-order methods such as Stochastic Gradient Descent ({\tt SGD}) \cite{gower2019sgd,nemirovski2009robust,nguyen2018sgd,RobbinsMonro:1951,vaswani2019fast} or its accelerated variants like {\tt AC-SA} \cite{lan2012optimal} or Similar Triangles Method ({\tt STM}) \cite{dvurechensky2017randomized,gasnikov2018universal, nesterov2018lectures} are very popular choice to solve either \eqref{eq:main_problem}+\eqref{eq:objectve_expectation} or \eqref{eq:main_problem}+\eqref{eq:erm_problem}. In contrast with their cheap iterations in terms of computational cost, these methods converge only to the neighbourhood of the solution, i.e., to the ball centered at the optimality and radius proportional to the standard deviation of the stochastic estimator. For the particular case of finite-sum minimization problem one can solve this issue via variance-reduction trick \cite{defazio2014saga, gorbunov2019unified, johnson2013accelerating, schmidt2017minimizing} and its accelerated variants \cite{allen2016katyusha,zhou2018direct,zhou2018simple}. Unfortunately, this technique is not applicable in general for the problems of type \eqref{eq:main_problem}+\eqref{eq:objectve_expectation}. Another possible way to reduce the variance is mini-batching. When the objective function is $L$-smooth one can accelerate the computations of batches using parallelization \cite{devolder2013exactness,dvurechensky2016stochastic,gasnikov2018universal,ghadimi2013stochastic}, and it is one of the examples where centralized distributed optimization appears naturally \cite{bertsekas1989parallel}.

In other words, in some situations, e.g., when the number of samples $m$ is too big, it is preferable in practice to split the data into $q$ blocks, assign each block to the separate worker, e.g., processor, and organize computation of the gradient or stochastic gradient in the parallel or distributed manner. Moreover, in view of \eqref{eq:convex_erm_argmin_property}-\eqref{eq:str_convex_erm_argmin_property} sometimes to solve an expectation minimization problem it is needed to have such a big number of samples that corresponding information (e.g.\ some objects like images, videos and etc.) cannot be stored on $1$ machine because of the memory limitations (see Section~\ref{sec:wasserstein} for the detailed example of such a situation). Then, we can rewrite the objective function in the following form
\begin{equation}
    f(x) = \frac{1}{q}\sum\limits_{i=1}^q f_i(x),\quad f_i(x) = \EE_{\xi_i}\left[f(x,\xi_i)\right] \text{ or } f_i(x) = \frac{1}{s_i}\sum\limits_{j=1}^{s_i}f(x,\xi_{ij}).\label{eq:finite_sum_minimization}
\end{equation}
Here $f_i$ corresponds to the loss on the $i$-th data block and could be also represented as an expectation or a finite sum. So, the general idea for parallel optimization is to compute gradients or stochastic gradients by each worker, then aggregate the results by the master node and broadcast new iterate or needed information to obtain the new iterate back to the workers.

The visual simplicity of the parallel scheme hides synchronization drawback and high requirement to master node \cite{scaman2017optimal}. The big line of works is aimed to solve this issue via periodical synchronization \cite{bayoumi2020tighter, stich2018local, yu2019linear, woodworth2020local, woodworth2020minibatch, koloskova2020unified, gorbunov2020local}, error-compensation \cite{karimireddy2019error, stich2018sparsified,beznosikov2020biased,gorbunov2020linearly}, quantization \cite{alistarh2017qsgd, horvath2019natural, horvath2019stochastic, mishchenko2019distributed, wen2017terngrad} or combination of these techniques \cite{basu2019qsparse,liu2019double}. 

However, in this work we mainly focus on another approach to deal with aforementioned drawbacks~--- decentralized distributed optimization \cite{bertsekas1989parallel,kibardin1979decomposition}. It is based on two basic principles: every node communicates only with its neighbours and communications are performed simultaneously. Moreover, this architecture is more robust, e.g., it can be applied to time-varying (wireless) communication networks \cite{rogozin2021towards}.

But let us consider first the centralized or parallel architecture. As we mentioned in the introduction, when the objective function is $L$-smooth one can compute batches in parallel \cite{devolder2013exactness,dvurechensky2016stochastic,gasnikov2018universal,ghadimi2013stochastic} in order to accelerate the work of the method and get the method (see Section~3 from \cite{gorbunov2019optimal} for the details) using
\begin{equation}
    O\left(\frac{\nicefrac{\sigma^2R^2}{\e^2}}{\sqrt{\nicefrac{LR^2}{\e}}}\right) \text{ or } O\left(\frac{\nicefrac{\sigma^2}{\mu\e}}{\sqrt{\nicefrac{L}{\mu}}\ln\left(\nicefrac{\mu R^2}{\e}\right)}\right)\label{eq:parallel_opt_number_of_workers}
\end{equation}
workers and having the working time proportional to the number of iterations of an accelerated first-order method. However, the number of workers defined in \eqref{eq:parallel_opt_number_of_workers} could be too big in order to use such an approach in practice. But still computing the batches in parallel even with much smaller number of workers could reduce the working time of the method if the communication is fast enough.

Besides the computation of batches in parallel for the general type of problem \eqref{eq:main_problem}+\eqref{eq:objectve_expectation}, parallel optimization is often applied to the finite-sum minimization problems \eqref{eq:main_problem}+\eqref{eq:erm_problem} or \eqref{eq:main_problem}+\eqref{eq:finite_sum_minimization} that we rewrite here in the following form:
\begin{equation}
    \min\limits_{x\in Q\subseteq \R^n}f(x) = \frac{1}{m}\sum\limits_{k=1}^m f_k(x).\label{eq:main_problem_decentralized_sec}
\end{equation}
We notice that in this section $m$ is a number of workers and $f_k(x)$ is known only for the $k$-th worker. Consider the situation when workers are connected in a network and one can construct a spanning tree for this network. Assume that the diameter of the obtained graph equals $d$, i.e., the height of the tree~--- maximal distance (in terms of connections) between the root and a leaf \cite{scaman2017optimal}. If we run Similar Triangles Methods ({\tt STM}, \cite{gasnikov2018universal}) on such a spanning tree then we will get that the number of communication rounds will be
\begin{equation*}
   O\left(dN + d\min\left\{\frac{\sigma^2 R^2}{\e^2}\ln\left(\frac{\sqrt{\nicefrac{LR^2}{\e}}}{\beta}\right), \frac{\sigma^2}{\mu\e}\ln\left(\frac{LR^2}{\e}\right)\ln\left(\frac{\sqrt{\nicefrac{L}{\mu}}}{\beta}\right)\right\}\right),
\end{equation*}
where 
\begin{equation*}
    N = O\left(\min\left\{\sqrt{\frac{LR^2}{\e}}, \sqrt{\frac{L}{\mu}}\ln\left(\frac{LR^2}{\e}\right)\right\}\right).
\end{equation*}

Now let us consider the decentralized case when workers can communicate only with their neighbours. Next, we describe the method of how to reflect this restriction in the problem \eqref{eq:main_problem_decentralized_sec}. Consider the Laplacian matrix $\overline{W}\in\R^{m\times m}$ of the network with vertices $V$ and edges $E$ which is defined as follows:
\begin{equation}
    \overline{W}_{ij} = \begin{cases} 
    -1, &\text{if } (i,j)\in E,\\
    \deg(i), &\text{if } i=j,\\
    0 &\text{otherwise},
    \end{cases}\label{eq:laplacian_matrix}
\end{equation}
where $\deg(i)$ is degree of $i$-th node, i.e.\ number of neighbours of the $i$-th worker. Since we consider only connected networks the matrix $\overline{W}$ has unique eigenvector $\bld{1}_m \eqdef (1,\ldots,1)^\top \in \R^m$ corresponding to the eigenvalue $0$. It implies that for all vectors $a = (a_1,\ldots,a_m)^\top\in \R^m$ the following equivalence holds:
\begin{equation}
    a_1 = \ldots = a_m \; \Longleftrightarrow \; \overline{W}a = 0.\label{eq:main_property_of_laplacian_simple}
\end{equation}
Now let us think about $a_i$ as a number that $i$-th node stores. Then, using \eqref{eq:main_property_of_laplacian_simple} we can use Laplacian matrix to express in the short matrix form the fact that all nodes of the network store the same number. In order to generalize it for the case when $a_i$ are vectors from $\R^n$ we should consider the matrix $W \eqdef \overline{W} \otimes I_n$ where $\otimes$ represents the Kronecker product. Indeed, if we consider vectors $x_1,\ldots,x_m\in\R^n$ and $\x = \left(x_1^\top,\ldots,x_m^\top\right)\in \R^{nm}$, then \eqref{eq:main_property_of_laplacian_simple} implies
\begin{equation}
    x_1 = \ldots = x_m \; \Longleftrightarrow \; W\x = 0.\label{eq:main_property_of_laplacian}
\end{equation}
For simplicity, we also call $W$ as a Laplacian matrix and it does not lead to misunderstanding since everywhere below we use $W$ instead of $\overline{W}$. The key observation here that computation of $Wx$ requires one round of communications when the $k$-th worker sends $x_k$ to all its neighbours and receives $x_j$ for all $j$ such that $(k,j)\in E$, i.e.\ $k$-th worker gets vectors from all its neighbours. Note, that $W$ is symmetric and positive semidefinite \cite{scaman2017optimal} and, as a consequence, $\sqrt{W}$ exists. Moreover, we can replace $W$ by $\sqrt{W}$ in \eqref{eq:main_property_of_laplacian} and get the equivalent statement:
\begin{equation}
    x_1 = \ldots = x_m \; \Longleftrightarrow \; \sqrt{W}\x = 0.\label{eq:main_property_of_laplacian_sqrt}
\end{equation}

Using this we can rewrite the problem \eqref{eq:main_problem_decentralized_sec} in the following way:
\begin{equation}
    \min\limits_{\substack{\sqrt{W}\x = 0, \\ x_1,\ldots, x_m \in Q \subseteq \R^n}}f(\x) = \frac{1}{m}\sum\limits_{k=1}^m f_k(x_k).\label{eq:main_problem_decentralized_sec_rewritten}
\end{equation}
We are interested in the general case when $f_k(x_k) = \EE_{\xi_k}\left[f_k(x_k,\xi_k)\right]$ where $\{\xi_k\}_{k=1}^m$ are independent. This type of objective can be considered as a special case of \eqref{eq:finite_sum_minimization}. Then, as it was mentioned in the introduction it is natural to use stochastic gradients $\nabla f_k(x_k,\xi_k)$ that satisfy
\begin{eqnarray}
    \left\|\EE_{\xi_k}\left[\nabla f_k(x_k,\xi_k)\right] - \nabla f_k(x_k)\right\|_2 &\le& \delta, \label{eq:primal_bias_in_stoch_grad_distrib}\\
    \EE_{\xi_k}\left[\exp\left(\frac{\left\|\nabla f_k(x_k,\xi_k) - \EE_{\xi_k}\left[\nabla f_k(x_k,\xi_k)\right]\right\|_2^2}{\sigma^2}\right)\right] &\le& \exp(1). \label{eq:primal_light_tails_stoch_grad_distrib}
\end{eqnarray}
Then, the stochastic gradient
\begin{equation*}
    \nabla f(\x,\xi) \eqdef \nabla f(\x,\{\xi_k\}_{k=1}^m) \eqdef \frac{1}{m}\sum\limits_{k=1}^m \nabla f_k(x_k,\xi_k)
\end{equation*}
satisfies (see also \eqref{eq:super_exp_moment_batched_stoch_grad})
\begin{equation*}
    \EE_{\xi}\left[\exp\left(\frac{\left\|\nabla f(\x,\xi) - \EE_{\xi}\left[\nabla f(\x,\xi)\right]\right\|_2^2}{\sigma_f^2}\right)\right] \le \exp(1)
\end{equation*}
with $\sigma_f^2 = O\left(\nicefrac{\sigma^2}{m}\right)$. 

As always, we start with the smooth case with $Q = \R^n$ and assume that each $f_k$ is $L$-smooth, $\mu$-strongly convex and satisfies $\|\nabla_{k}f_k(x_k)\|_2 \le M$ on some ball $B_{R_M}(x^*)$ where we use $\nabla_{k}f(x_k)$ to emphasize that $f_k$ depends only on the $k$-th $n$-dimensional block of $\x$. Since the functional $f(\x)$ in \eqref{eq:main_problem_decentralized_sec_rewritten} has separable structure, it implies that $f$ is $\nicefrac{L}{m}$-smooth, $\nicefrac{\mu}{m}$-strongly convex and satisfies $\|\nabla f(\x)\|_2 \le \nicefrac{M}{\sqrt{m}}$ on $B_{\sqrt{m}R_M}(\x^*)$. Indeed, for all $\x,\y \in \R^n$
\begin{eqnarray*}
    \|\x - \y\|_2^2 &=& \sum\limits_{k=1}^m \|x_k - y_k\|_2^2,\\
    \|\nabla f(\x) - \nabla f(\y)\|_2 &=& \sqrt{\frac{1}{m^2}\sum\limits_{k=1}^m\|\nabla_{k} f_k(x_k) - \nabla_{k} f_k(y_k)\|_2^2} \le \sqrt{\frac{L^2}{m^2}\sum\limits_{k=1}^m\|x_k - y_k\|_2^2} = \frac{L}{m}\|\x - \y\|_2,\\
    f(\x) &=& \frac{1}{m}\sum\limits_{k=1}^m f_k(x_k) \ge \frac{1}{m}\sum\limits_{k=1}^m \left(f(y_k) + \la\nabla_k f_k(y_k), x_k - y_k \ra + \frac{\mu}{2}\|x^k - y^k\|_2^2\right)\\
    &=& f(\y) + \la\nabla f(\y), \x - \y \ra + \frac{\mu}{2m}\|\x - \y\|_2^2,\\\
    \|\nabla f(\x)\|_2^2 &=& \frac{1}{m^2}\sum\limits_{k=1}^m\|\nabla_k f_k(x_k)\|_2^2.
\end{eqnarray*}

Therefore, one can consider the problem \eqref{eq:main_problem_decentralized_sec_rewritten} as \eqref{PP} with $A = \sqrt{W}$ and $Q = \R^{nm}$. Next, if the starting point $\x^0$ is such that $\x^0 = ((x^0)^\top,\ldots,(x^0)^\top)^\top$ then 
\begin{eqnarray*}
    \Rbf^2 \eqdef \|\x^0 - \x^*\|_2^2 = m\|x^0 - x^*\|_2^2 = mR^2,\quad R_\y^2 \eqdef \|\y^*\|_2^2 \le \frac{\|\nabla f(\x^*)\|_2^2}{\lambda_{\min}^+(W)} \le \frac{M^2}{m\lambda_{\min}^+(W)}.
\end{eqnarray*}
Now it should become clear why in Section~\ref{sec:primal} we paid most of our attention on number of $A^\top A\x$ calculations. In this particular scenario $A^\top A\x = \sqrt{W}^\top \sqrt{W}x = Wx$ which can be computed via one round of communications of each node with its neighbours as it was mentioned earlier in this section. That is, for the primal approach we can simply use the results discussed in Section~\ref{sec:primal}. For convenience, we summarize them in Tables~\ref{tab:deterministic_bounds_primal} and \ref{tab:stochastic_bounds_primal} which are obtained via plugging the parameters that we obtained above in the bounds from Section~\ref{sec:primal}. Note that the results presented in this match the lower bounds obtained in \cite{arjevani2015communication} in terms of the number of communication rounds up to logarithmic factors and and there is a conjecture \cite{dvinskikh2019decentralized} that these bounds are also optimal in terms of number of oracle calls per node for the class of methods that require optimal number of communication rounds. Recently, the very similar result about the optimal balance between number of oracle calls per node and number of communication round was proved for the case when the primal functional is convex and $L$-smooth and deterministic first-order oracle is available \cite{xu2019accelerated}.
\begin{table}[ht!]
    \centering
    \begin{tabular}{|c|c|c|c|}
         \hline
         Assumptions on $f_k$ & Method & \makecell{\# of communication\\ rounds} & \makecell{\# of $\nabla f_k(x)$ oracle\\ calls per node}\\
         \hline
         \makecell{ $\mu${-strongly convex,}\\ $L$-smooth} & \makecell{{\tt D-MASG},\\$Q = \R^n$,\\{\cite{fallah2019robust}}} & $\widetilde{O}\left(\sqrt{\frac{L}{\mu}\chi}\right)$ & $\widetilde{O}\left(\sqrt{\frac{L}{\mu}}\right)$ \\
         \hline
         $L$-smooth & \makecell{{\tt STP{\_}IPS} with \\ {\tt STP} as a subroutine,\\ $Q = \R^n$,\\{ \cite{gorbunov2019optimal}}}  & $\widetilde{O}\left(\sqrt{\frac{LR^2}{\e}\chi}\right)$ & $\widetilde{O}\left(\sqrt{\frac{LR^2}{\e}}\right)$ \\
         \hline
         \makecell{ $\mu${-strongly convex,}\\ $\|\nabla f_k(x)\|_2 \le M$}& \makecell{\tt R-Sliding,\\{\makecell{\cite{dvinskikh2019decentralized,Lan2019lectures,lan2016gradient,lan2017communication} }}} & $\widetilde{O}\left(\sqrt{\frac{M^2}{\mu\e}\chi}\right)$ & $\widetilde{O}\left(\frac{M^2}{\mu\e}\right)$ \\
         \hline
         $\|\nabla f_k(x)\|_2 \le M$ & \makecell{\tt Sliding,\\\makecell{\cite{Lan2019lectures,lan2016gradient,lan2017communication}}} & $ O\left(\sqrt{\frac{M^2R^2}{\e^2}\chi}\right)$ & $ O\left(\frac{M^2R^2}{\e^2}\right)$ \\
         \hline
    \end{tabular}
    \caption{\small Summary of the covered results in this paper for solving \eqref{eq:main_problem_decentralized_sec_rewritten} using primal deterministic approach from Section~\ref{sec:primal}. First column contains assumptions on $f_k$, $k=1,\ldots,m$ in addition to the convexity, $\chi = \chi(W) = \nicefrac{\lambda_{\max}(W)}{\lambda_{\min}^+(W)}$, where $\lambda_{\max}(W)$ and $\lambda_{\min}^+(W)$ are maximal and minimal positive eigenvalues of matrix $W$. All methods except {\tt D-MASG} should be applied to solve \eqref{penalty}.}
    \label{tab:deterministic_bounds_primal}
\end{table}
\begin{table}[t!]
    \centering
    \begin{tabular}{|c|c|c|c|}
         \hline
         Assumptions on $f_k$ & Method & \makecell{\# of communication\\ rounds} & \makecell{\# of $\nabla f_k(x,\xi)$ oracle\\ calls per node}\\
         \hline
         \makecell{ $\mu${-strongly convex,}\\ $L$-smooth} & \makecell{{\tt D-MASG},\\in expectation, \\$Q = \R^n$,\\{\cite{fallah2019robust}}} & $\widetilde{O}\left(\sqrt{\frac{L}{\mu}\chi}\right)$ & $\widetilde{O}\left(\max\left\{\sqrt{\frac{L}{\mu}}, \frac{\sigma^2}{\mu\e} \right\}\right)$\\
         \hline
         $L$-smooth & \makecell{{\tt SSTP{\_}IPS} with \\ {\tt STP} as a subroutine,\\ $Q = \R^n$,\\{\makecell{conjecture, \\ \cite{dvinskikh2019decentralized,gorbunov2019optimal}}}} & $\widetilde{O}\left(\sqrt{\frac{LR^2}{\e}\chi}\right)$ & $\widetilde{O}\left(\max\left\{\sqrt{\frac{LR^2}{\e}}, \frac{\sigma^2 R^2}{\e^2}\right\}\right)$\\
         \hline
         \makecell{ $\mu${-strongly convex,}\\ $\|\nabla f_k(x)\|_2 \le M$}& \makecell{{\tt RS-Sliding}\\ $Q$ is bounded,\\{ \makecell{\cite{dvinskikh2019decentralized,Lan2019lectures,lan2016gradient,lan2017communication} }}} & $\widetilde{O}\left(\sqrt{\frac{M^2}{\mu\e}\chi}\right)$ & $\widetilde{O}\left(\frac{M^2 + \sigma^2}{\mu\e}\right)$ \\
         \hline
         $\|\nabla f_k(x)\|_2 \le M$ & \makecell{{\tt S-Sliding} \\ $Q$ is bounded,\\\makecell{\cite{Lan2019lectures,lan2016gradient,lan2017communication}}} & $ \widetilde{O}\left(\sqrt{\frac{M^2R^2}{\e^2}\chi}\right)$ & $\widetilde{O}\left(\frac{(M^2+\sigma^2)R^2}{\e^2}\right)$\\
         \hline
    \end{tabular}
    \caption{\small Summary of the covered results in this paper for solving \eqref{eq:main_problem_decentralized_sec_rewritten} using primal stochastic approach from Section~\ref{sec:primal} with the stochastic oracle satisfying \eqref{eq:primal_bias_in_stoch_grad_distrib}-\eqref{eq:primal_light_tails_stoch_grad_distrib} with $\delta = 0$. First column contains assumptions on $f_k$, $k=1,\ldots,m$ in addition to the convexity, $\chi = \chi(W) = \nicefrac{\lambda_{\max}(W)}{\lambda_{\min}^+(W)}$, where $\lambda_{\max}(W)$ and $\lambda_{\min}^+(W)$ are maximal and minimal positive eigenvalues of matrix $W$. All methods except {\tt D-MASG} should be applied to solve \eqref{penalty}. The bounds from the last two rows hold even in the case when $Q$ is unbounded, but in the expectation (see \cite{lan2016algorithms}). 
    }
    \label{tab:stochastic_bounds_primal}
\end{table}

Finally, consider the situation when $Q = \R^n$ and each $f_k$ from \eqref{eq:main_problem_decentralized_sec_rewritten} is dual-friendly, i.e.\ one can construct dual problem for \eqref{eq:main_problem_decentralized_sec_rewritten}
\begin{eqnarray}
\min_{\y\in\R^{nm}}\Psi(\y), && \text{where } \y =(y_1^\top,\ldots, y_m^\top)^\top \in \R^{nm},\; y_1,\ldots,y_m\in\R^{n},\label{eq:dual_problem_distributed}\\
\varphi_k(y_k) &=& \max_{x_k\in \R^n}\left\{\langle y_k,x_k\rangle - f_k(x_k)\right\},\label{eq:dual_phi_k_function_distributed}\\
\Phi(\y) &=& \frac{1}{m}\sum\limits_{k=1}^m\varphi_k(my_k),\; \Psi(\y) = \Phi(\sqrt{W}\y) = \frac{1}{m}\sum\limits_{k=1}^m\varphi_k(m[\sqrt{W}\x]_k),\label{eq:dual_phi_psi_function_distributed}
\end{eqnarray}
where $[\sqrt{W}\x]_k$ is the $k$-th $n$-dimensional block of $\sqrt{W}x$. Note that
\begin{eqnarray*}
    \max\limits_{\x\in\R^{nm}}\left\{\la\y , \x \ra - f(\x)\right\} &=& \max\limits_{\x\in\R^{nm}}\left\{\sum\limits_{k=1}^m\la y_k, x_k\ra - \frac{1}{m}\sum\limits_{k=1}^m f_k(x_k)\right\} \\
    &=&\frac{1}{m}\sum\limits_{k=1}^m \max\limits_{x_k\in\R^n}\left\{\la my_k,x_k\ra - f_k(x_k)\right\} =  \frac{1}{m}\sum\limits_{k=1}^m\varphi_k(my_k) = \Phi(\y),
\end{eqnarray*}
so, $\Phi(\y)$ is a dual function for $f(\x)$. As for the primal approach, we are interested in the general case when $\varphi_k(y_k) = \EE_{\xi_k}\left[\varphi_k(y_k,\xi_k)\right]$ where $\{\xi_k\}_{k=1}^m$ are independent and stochastic gradients $\nabla \varphi_k(x_k,\xi_k)$ satisfy
\begin{eqnarray}
    \left\|\EE_{\xi_k}\left[\nabla \varphi_k(y_k,\xi_k)\right] - \nabla \varphi_k(y_k)\right\|_2 &\le& \delta_\varphi, \label{eq:dual_bias_in_stoch_grad_distrib}\\
    \EE_{\xi_k}\left[\exp\left(\frac{\left\|\nabla \varphi_k(y_k,\xi_k) - \EE_{\xi_k}\left[\nabla \varphi_k(y_k,\xi_k)\right]\right\|_2^2}{\sigma^2}\right)\right] &\le& \exp(1). \label{eq:dual_light_tails_stoch_grad_distrib}
\end{eqnarray}
Consider the stochastic function $f_k(x_k,\xi_k)$ which is defined implicitly as follows:
\begin{equation}
    \varphi_k(y_k,\xi_k) = \max\limits_{x_k\in \R^n}\left\{\la y_k, x_k \ra - f(x_k,\xi_k)\right\}.\label{eq:dual_stoch_func_distrib}
\end{equation}
Since
\begin{eqnarray*}
    \nabla \Phi(\y) = \sum\limits_{k=1}^m\nabla\varphi_k(my_k) \overset{\eqref{eq:gradient_dual_function}}{=} \sum\limits_{k=1}^m x_k(my_k) \eqdef \x(\y),\quad x_k(y_k) \eqdef \argmax_{x_k\in\R^n}\left\{\langle y_k,x_k\rangle - f_k(x_k)\right\}
\end{eqnarray*}
it is natural to define the stochastic gradient $\nabla \Phi(\y,\xi)$ as follows:
\begin{eqnarray*}
    \nabla \Phi(\y,\xi) &\eqdef& \nabla \Phi(\y,\{\xi_k\}_{k=1}^m) \eqdef \sum\limits_{k=1}^m \nabla \varphi_k(my_k,\xi_k)\overset{\eqref{eq:gradient_dual_function}}{=} \sum\limits_{k=1}^m x_k(my_k,\xi_k)\eqdef \x(\y,\xi),\\
    x_k(y_k,\xi_k) &\eqdef& \argmax_{x_k\in\R^n}\left\{\langle y_k,x_k\rangle - f_k(x_k,\xi_k)\right\}.
\end{eqnarray*}
It satisfies (see also \eqref{eq:super_exp_moment_batched_stoch_grad})
\begin{eqnarray*}
    \left\|\EE_{\xi}\left[\nabla \Phi(\y,\xi)\right] - \nabla \Phi(\y)\right\|_2 &\le& \delta_\Phi,\\
    \EE_{\xi}\left[\exp\left(\frac{\left\|\nabla \Phi(\y,\xi) - \EE_{\xi}\left[\nabla \Phi(\y,\xi)\right]\right\|_2^2}{\sigma_\Phi^2}\right)\right] &\le& \exp(1)
\end{eqnarray*}
with $\delta_\Phi = m\delta_\varphi$ and $\sigma_\Phi^2 = O\left(m\sigma^2\right)$. Using this, we define the stochastic gradient of $\Psi(\y)$ as $\nabla \Psi(\y, \xi) \eqdef \sqrt{W}\nabla \Phi(\sqrt{W}\y,\xi) = \sqrt{W}\x(\sqrt{W}\y,\xi)$ and, as a consequence, we get
\begin{eqnarray*}
    \left\|\EE_{\xi}\left[\nabla \Psi(\y,\xi)\right] - \nabla \Psi(\y)\right\|_2 &\le& \delta_\Psi,\\
    \EE_{\xi}\left[\exp\left(\frac{\left\|\nabla \Psi(\y,\xi) - \EE_{\xi}\left[\nabla \Psi(\y,\xi)\right]\right\|_2^2}{\sigma_\Psi^2}\right)\right] &\le& \exp(1)
\end{eqnarray*}
with $\delta_\Psi = \sqrt{\lambda_{\max}(W)}\delta_\Phi$ and $\sigma_\Psi = \sqrt{\lambda_{\max}(W)}\sigma_\Phi$.

Taking all of this into account we conclude that problem \eqref{eq:dual_problem_distributed} is a special case of \eqref{DP} with $A = \sqrt{W}$. To make the algorithms from Section~\ref{sec:dual} distributed we should change the variables in those methods via multiplying them by $\sqrt{W}$ from the left \cite{dvinskikh2019decentralized,dvinskikh2019dual,uribe2017optimal}, e.g.\ for the iterates of {\tt SPDSTM} we will get
\begin{equation*}
    \tilde{y}^{k+1} := \sqrt{W}\tilde{y}^{k+1}, \quad z^{k+1} := \sqrt{W}z^{k+1},\quad y^{k+1} := \sqrt{W}y^{k+1},
\end{equation*}
which means that it is needed to multiply lines 4-6 of Algorithm~\ref{Alg:PDSTM} by $\sqrt{W}$ from the left. After such a change of variables all methods from Section~\ref{sec:dual} become suitable to run them in the distributed fashion. Besides that, it does not spoil the ability of recovering the primal variables since before the change of variables all of the methods mentioned in Section~\ref{sec:dual} used $\x(\sqrt{W}\y)$ or $\x(\sqrt{W}\y, \xi)$ where points $y$ were some dual iterates of those methods, so, after the change of variables we should use $\x(\y)$ or $\x(\y, \xi)$ respectively. Moreover, it is also possible to compute $\|\sqrt{W}x\|_2^2 = \la\x, W\x \ra$ in the distributed fashion using consensus type algorithms: one communication step is needed to compute $W\x$, then each worker computes $\la x_k, [W\x]_k \ra$ locally and after that it is needed to run consensus algorithm. We summarize the results for this case in Table~\ref{tab:stochastic_bounds_dual}. Note that the proposed bounds are optimal in terms of the number of communication rounds up to polylogarithmic factors \cite{arjevani2015communication,scaman2017optimal,scaman2019optimal,scaman2018optimal}. Note that the lower bounds from \cite{scaman2017optimal,scaman2019optimal,scaman2018optimal} are presented for the convolution of two criteria: number of oracle calls per node and communication rounds. One can obtain lower bounds for the number of communication rounds itself using additional assumption that time needed for one communication is big enough and the term which corresponds to the number of oracle calls can be neglected. Regarding the number of oracle calls there is a conjecture \cite{dvinskikh2019decentralized} that the bounds that we present in this paper are also optimal up to polylogarithmic factors for the class of methods that require optimal number of communication rounds.

\begin{table}[ht!]
    \centering
    \begin{tabular}{|c|c|c|c|c|}
         \hline
         Assumptions on $f_k$ & Method  & \makecell{\# of communication\\ rounds} & \makecell{\# of $\nabla \varphi_k(y,\xi)$ oracle\\ calls per node}\\
         \hline
         \makecell{ $\mu${-strongly convex,}\\ $L$-smooth,\\
         $\|\nabla f_k(x)\|_2 \le M$} & \makecell{{\tt R-RRMA-AC-SA$^2$} \\ (Algorithm~\ref{Alg:Restarted-RRMA-AC-SA2}),\\Corollary~\ref{cor:r-rrma-ac-sa2_connect_with_primal}, \\
         {\tt SSTM{\_}sc} \\ (Algorithm~\ref{Alg:STM_str_cvx}),\\Corollary~\ref{cor:sstm_str_cvx_connect_with_primal}}& 
         $\widetilde{O}\left(\sqrt{\frac{L}{\mu}\chi}\right)$ & $\widetilde{O}\left(\max\left\{\sqrt{\frac{L}{\mu}\chi}, \frac{\sigma_\Phi^2M^2}{\e^2}\chi \right\}\right)$\\
         \hline
         \makecell{ $\mu${-strongly convex,}\\
         $\|\nabla f_k(x)\|_2 \le M$} & \makecell{{\tt SPDSTM} \\ (Algorithm~\ref{Alg:PDSTM}),\\ Theorem~\ref{thm:spdtstm_smooth_cvx_dual_biased}} & 
         $\widetilde{O}\left(\sqrt{\frac{M^2}{\mu\e}\chi}\right)$ & $\widetilde{O}\left(\max\left\{\sqrt{\frac{M^2}{\mu\e}\chi}, \frac{\sigma_\Phi^2M^2}{\e^2}\chi \right\}\right)$\\
         \hline
    \end{tabular}
    \caption{\small Summary of the covered results in this paper for solving \eqref{eq:dual_problem_distributed} using dual stochastic approach from Section~\ref{sec:dual} with the stochastic oracle satisfying \eqref{eq:primal_bias_in_stoch_grad_distrib}-\eqref{eq:primal_light_tails_stoch_grad_distrib} with $\delta = 0$ for {\tt R-RRMA-AC-SA$^2$} and $\delta_\varphi = \widetilde{O}\left(\nicefrac{\e }{(M\sqrt{m\chi})}\right)$ for {\tt SSTM{\_}sc} and {\tt SPDSTM}. First column contains assumptions on $f_k$, $k=1,\ldots,m$ in addition to the convexity, $\chi = \chi(W)$. }
    \label{tab:stochastic_bounds_dual}
\end{table}

\subsection{Discussion}\label{sec:discussion}
In this section, we want to discuss some aspects of the proposed results that were not covered in the main part of this paper. First of all, we should say that in the smooth case for the primal approach our bounds for the number of communication steps coincides with the optimal bounds for the number of communication steps for parallel optimization if we substitute the diameter $d$ of the spanning tree in the bounds for parallel optimization by $\widetilde{O}(\sqrt{\chi(W)})$.

However, we want to discuss another interesting difference between parallel and decentralized optimization in terms of the complexity results which was noticed in \cite{dvinskikh2019decentralized}. From the line of works \cite{kulunchakov2019estimate1,kulunchakov2019estimate2,kulunchakov2019generic,lan2018random} it is known that for the problem \eqref{eq:main_problem}+\eqref{eq:finite_sum_minimization} (here we use $m$ instead of $q$ and iterator $k$ instead of $i$ for consistency) with $L$-smooth and $\mu$-strongly convex $f_k$ for all $k=1,\ldots, m$ the optimal number of oracle calls, i.e.\ calculations of of the stochastic gradients of $f_k$ with $\sigma^2$-subgaussian variance is
\begin{equation}
    \widetilde{O}\left(m + \sqrt{m\frac{L}{\mu}} + \frac{\sigma^2}{\mu\e}\right).\label{eq:optimal_bound_stoch_str_cvx}
\end{equation}
The bad news is that \eqref{eq:optimal_bound_stoch_str_cvx} does not work with full parallelization trick and the best possible way to parallelize it is described in \cite{lan2018random}. However, standard accelerated scheme using mini-batched versions of the stochastic gradients without variance-reduction technique and incremental oracles which gives the bound
\begin{equation}
    \widetilde{O}\left(m\sqrt{\frac{L}{\mu}} + \frac{\sigma^2}{\mu\e}\right)\label{eq:classical_accelerated_scheme}
\end{equation}
for the number of oracle calls and it admits full parallelization. It means that in the parallel optimization setup when we have computational network with $m$ nodes and the spanning tree for it with diameter $d$ the number of oracle calls per node is
\begin{equation}
    \widetilde{O}\left(\sqrt{\frac{L}{\mu}} + \frac{\sigma^2}{m\mu\e}\right) = \widetilde{O}\left(\max\left\{\sqrt{\frac{L}{\mu}},\frac{\sigma^2}{m\mu\e}\right\}\right)\label{eq:paralel_opt_oracle_per_node}
\end{equation}
and the number of communication steps is
\begin{equation}
    \widetilde{O}\left(d\sqrt{\frac{L}{\mu}}\right).\label{eq:parallel_opt_communications}
\end{equation}
However, for the decentralized setup the second row of Table~\ref{tab:stochastic_bounds_primal} states that the number of communication rounds is the same as in \eqref{eq:parallel_opt_communications} up to substitution of $d$ by $\sqrt{\chi(W)}$ and the number of oracle calls per node is
\begin{equation}
    \widetilde{O}\left(\max\left\{\sqrt{\frac{L}{\mu}},\frac{\sigma^2}{\mu\e}\right\}\right)\label{eq:decentralized_opt_oracle_per_node}
\end{equation}
which has $m$ times bigger statistical term under the maximum than in \eqref{eq:paralel_opt_oracle_per_node}. What is more, recently it was shown that there exists such a decentralized distributed method that requires
$$\widetilde{O}\left( \frac{\sigma^2}{m\mu\e}\right)$$ 
stochastic gradient oracle calls per node \cite{olshevsky2019asymptotic,olshevsky2019non}, but it is not optimal in terms of the number of communications. 
\ar{Recently a stochastic optimization method with consensus subroutine for time-varying graphs requiring $\tilde O\left(\sigma^2/(n\mu\eps)\right)$ oracle calls and $\tilde O\left(\sqrt{{L}/{\mu}}\chi\right)$ communications was proposed in \cite{rogozin2021accelerated}. The results of \cite{rogozin2021accelerated} can be easily extended to $\tilde O(\sqrt{{L}/{\mu}}\sqrt\chi)$ communication complexity in the time-static case via employing accelerated consensus with Chebyshev acceleration.} Moreover, there is a hypothesis \cite{dvinskikh2019decentralized} that in the smooth case the bounds from Tables~\ref{tab:deterministic_bounds_primal}~and~\ref{tab:stochastic_bounds_primal} (rows 2 and 3) are not optimal in terms of the number of oracle calls per node and optimal ones can be found in Table~\ref{T:stoch_prima_oracle}.


\subsection{Application for Population Wasserstein Barycenter Calculation}\label{sec:wasserstein}

In this section we consider the problem of calculation of population Wasserstein barycenter since this example hides different interesting details connected with the theory discussed in this paper. In our presentation of this example we rely mostly on the recent works \dm{\cite{dvinskikh2020sa,dvinskikh2021decentralized}}.

\subsubsection{Definitions and Properties}\label{sec:wass_defs}
We define the probability simplex in $\R^n$ as $S_n(1) = \left\{x\in\R_{+}^n\mid \sum_{i=1}^n x_i  = 1\right\}$. One can interpret the elements of $S_n(1)$ as discrete probability measures with $n$ shared atoms. For an arbitrary pair of measures $p,q\in S_n(1)$ we introduce the set $\Pi(p,q) = \left\{\pi\in\R_+^{n\times n}\mid \pi\one = p,\; \pi^\top\one = q\right\}$ called transportation polytope. Optimal transportation (OT) problem between measures $p,q\in S_n(1)$ is defined as follows
\begin{equation}
    \cW(p,q) = \min\limits_{\pi\in\Pi(p,q)}\langle C, \pi\rangle = \min\limits_{\pi\in\Pi(p,q)}\sum\limits_{i,j=1}^n C_{ij}\pi_{ij}\label{eq:wasserstein_distance}
\end{equation}
where $C$ is a transportation cost matrix. That is, $(i,j)$-th component $C_{ij}$ of $C$ is a cost of transportation of the unit mass from point $x_i$ to the point $x_j$ where points
are atoms of measures from $S_n(1)$.

Next, we consider the entropic OT problem (see \cite{peyre2019computational, rigollet2018entropic})
\begin{equation}
    \cW_\mu (p,q) = \min_{\pi\in \Pi(p,q)}\sum\limits_{i,j=1}^n \left(C_{ij}\pi_{ij} + \mu\pi_{ij}\ln\pi_{ij}\right).\label{eq:entropic_wasserstein_distance}
\end{equation}
Consider some probability measure $\PP$ on $S_n(1)$. Then one can define population barycenter of measures from $S_n(1)$ as
\begin{equation}
    p_\mu^* = \argmin\limits_{p\in S_n(1)}\int_{q\in S_n(1)}\cW_\mu(p,q)d\PP(q) = \argmin\limits_{p\in S_n(1)}\underbrace{\EE_q\left[\cW_\mu(p,q)\right]}_{\cW_\mu(p)}.\label{eq:population_barycenter}
\end{equation}
For a given set of samples $q^1,\ldots, q^m$ we introduce empirical barycenter as
\begin{equation}
    \hat p_\mu^* = \argmin\limits_{p\in S_n(1)}\underbrace{\frac{1}{m}\sum\limits_{i=1}^m\cW_\mu(p,q^i)}_{\hat\cW(p)}.\label{eq:empirical_barycenter}
\end{equation}
We consider the problem \eqref{eq:population_barycenter} of finding population barycenter with some accuracy and discuss possible approaches to solve this problem in the following subsections.

However, before that, we need to mention some useful properties of $\cW_\mu(p,q)$. First of all, one can write explicitly the dual function of $W_\mu(p,q)$ for a fixed $q\in S_n(1)$ (see \cite{cuturi2016smoothed,dvinskikh2020sa}):
\begin{eqnarray}
    \cW_\mu(p,q) &=& \max\limits_{\lambda \in \R^n}\left\{\la\lambda, p\ra - \cW_{q,\mu}^*(\lambda)\right\}\label{eq:wasserstein_dist_sual_reformulation}\\
    \cW_{q,\mu}^*(\lambda) &=& \mu\sum\limits_{j=1}^n q_j\ln\left(\frac{1}{q_j}\sum\limits_{i=1}^n\exp\left(\frac{-C_{ij} + \lambda_i}{\mu}\right)\right).\label{eq:dual_function_wasserstein_distance}
\end{eqnarray}
Using this representation one can deduce the following theorem.
\begin{theorem}[\cite{dvinskikh2020sa}]\label{thm:wasserstein_dist_properties}
    For an arbitrary $q\in S_n(1)$ the entropic Wasserstein distance $\cW_\mu(\cdot,q): S_n(1) \to \R$ is $\mu$-strongly convex w.r.t.\ $\ell_2$-norm and $M$-Lipschitz continuous w.r.t.\ $\ell_2$-norm. Moreover,  $M \le \sqrt{n}M_\infty$ where $M_\infty$ is Lipschitz constant of $\cW_\mu(\cdot,q)$ w.r.t.\ $\ell_\infty$-norm and\footnote{Under assumption that measures are separated from zero, see the details in \cite{blanchet2018towards} and the proof of Proposition~2.5 from \cite{dvinskikh2020sa}.} $M_\infty = \widetilde{O}(\|C\|_\infty)$.
\end{theorem}

We also want to notice that function $\cW_{q,\mu}^*(\lambda)$ is only strictly convex and the minimal eigenvalue of its hessian $\gamma \eqdef \lambda_{\min}(\nabla^2\cW_{q,\mu}(\lambda^*))$ evaluated in the solution $\lambda^* \eqdef \argmax_{\lambda\in \R^n}\left\{\la\lambda, p\ra - \cW_{q,\mu}^*(\lambda)\right\}$ is very small and there exist only such bounds that are exponentially small in $n$.

We will also use another useful relation (see \cite{dvinskikh2020sa}):
\begin{equation}
    \nabla\cW_\mu(p,q) = \lambda^*,\quad \la\lambda^*, \one \ra = 0 \label{eq:dual_solution_grad_wasser_dist_relation}
\end{equation}
where the gradient $\nabla\cW_\mu(p,q)$ is taken w.r.t.\ the first argument.

\subsubsection{SA Approach}\label{sec:sa_approach_barycenter}
Assume that one can obtain and use fresh samples $q^1,q^2,\ldots$ in online regime. This approach is called Stochastic Approximation (SA). It implies that at each iteration one can draw a fresh sample $q^k$ and compute the gradient w.r.t.\ $p$ of function $\cW_\mu(p,q^k)$ which is $\mu$-strongly convex and $M$-Lipschitz continuous with $M = \widetilde{O}(\sqrt{n}\|C\|_\infty)$. Optimal methods for this case are based on iterations of the following form
\begin{equation*}
    p^{k+1} = \text{proj}_{S_n(1)}\left(p^k - \eta_k \nabla \cW_\mu(p^k,q^k)\right)
\end{equation*}
where $\text{proj}_{S_n(1)}(x)$ is a projection of $x\in\R^n$ on $S_n(1)$ and the gradient $\nabla \cW_\mu(p^k,q^k)$ is taken w.r.t.\ the first argument. One can show that {\tt restarted-SGD} ({\tt R-SGD}) from \cite{juditsky2014deterministic} that using biased stochastic gradients (see also \cite{juditsky2012first-order,gasnikov2016gradient-free,dvinskikh2020sa}) $\tnabla\cW_\mu(p,q)$ such that
\begin{equation}
    \|\tnabla\cW_\mu(p,q) - \nabla \cW_\mu(p,q)\|_2 \le \delta\label{eq:barycenter_inexact_grad}
\end{equation}
for some $\delta \ge 0$ and for all $p,q\in S_n(1)$ after $N$ calls of this oracle produces such a point $p^N$ that with probability at least $1-\beta$ the following inequalities hold:
\begin{equation}
    \cW_\mu(p^N) - \cW_\mu(p_\mu^*) = O\left(\frac{n\|C\|_\infty^2\ln(\nicefrac{N}{\alpha})}{\mu N} + \delta\right)\label{eq:r-sgd_pop_barycenter_func_guarantee}
\end{equation}
and, as a consequence of $\mu$-strong convexity of $\cW_\mu(p,q)$ for all $q$,
\begin{equation}
    \|p^N - p_\mu^*\|_2 = O\left(\sqrt{\frac{n\|C\|_\infty^2\ln(\nicefrac{N}{\alpha})}{\mu^2 N} + \frac{\delta}{\mu}}\right).\label{eq:r-sgd_pop_barycenter_distance_guarantee}
\end{equation}
That is, to guarantee
\begin{equation}
    \|p^N - p_\mu^*\|_2 \le \e\label{eq:r-sgd_pop_barycenter_distance_guarantee_eps}
\end{equation}
with probability at least $1-\beta$, {\tt R-SGD} requires
\begin{equation}
    \widetilde{O}\left(\frac{n\|C\|_\infty^2}{\mu^2\e^2}\right) \quad \tnabla \cW_\mu(p,q) \text{ oracle calls}\label{eq:r-sgd_number_of_oracle_calls}
\end{equation}
under additional assumption that $\delta = O(\mu\e^2)$.

However, it is computationally hard problem to find $\nabla \cW_\mu(p,q)$ with high-accuracy, i.e.\ find $\tnabla \cW_\mu(p,q)$ satisfying \eqref{eq:barycenter_inexact_grad} with $\delta = O(\mu\e^2)$. Taking into account the relation \eqref{eq:dual_solution_grad_wasser_dist_relation} we get that it is needed to solve the problem \eqref{eq:wasserstein_dist_sual_reformulation} with accuracy $\delta = O(\mu\e^2)$ in terms of the distance to the optimum. i.e.\ it is needed to find such $\tilde\lambda$ that $\|\tilde\lambda - \lambda^*\|_2 \le \delta$ and set $\tnabla \cW_\mu(p,q) = \tilde\lambda$. Using variants of Sinkhorn algorithm \cite{kroshnin2019complexity,stonyakin2019gradient,guminov2019accelerated} one can show \cite{dvinskikh2020sa} that {\tt R-SGD} finds point $p^N$ such that \eqref{eq:r-sgd_pop_barycenter_distance_guarantee_eps} holds with probability at least $1-\beta$ and it requires
\begin{equation}
    \widetilde{O}\left(\frac{n^3\|C\|_\infty^2}{\mu^2\e^2}\min\left\{\exp\left(\frac{\|C\|_\infty}{\mu}\right)\left(\frac{\|C\|_\infty}{\mu} + \ln\left(\frac{\|C\|_\infty}{\gamma\mu^2\e^4}\right)\right), \sqrt{\frac{n}{\gamma\mu^3\e^4}}\right\}\right)\label{eq:sa_overall_complexity}
\end{equation}
arithmetical operations.

\subsubsection{SAA Approach}\label{sec:saa_approach_barycenter}
Now let us assume that large enough collection of samples $q^1,\ldots,q^m$ is available. Our goal is to find such $p\in S_n(1)$ that $\|\hat p - p_\mu^*\|_2 \le \e$ with high probability, i.e.\ $\e$-approximation of the population barycenter, via solving empirical barycenter problem \eqref{eq:empirical_barycenter}. This approach is called Stochastic Average Approximation (SAA). Since $\cW_\mu(p,q^i)$ is $\mu$-strongly convex and $M$-Lipschitz in $p$ with $M = \widetilde{O}(\sqrt{n}\|C\|_\infty)$ for all $i=1,\ldots,m$ we can conclude that with probability $\ge 1-\beta$
\begin{equation}
    \cW_\mu(\hat p_\mu^*) - \cW_\mu(p_\mu^*) \overset{\eqref{eq:str_convex_erm_argmin_property}}{=} O\left(\frac{n\|C\|_\infty^2\ln(m)\ln\left(\nicefrac{m}{\beta}\right)}{\mu m} + \sqrt{\frac{n\|C\|_\infty^2\ln\left(\nicefrac{1}{\beta}\right)}{m}}\right)\label{eq:erm_rm_difference_no_beta}
\end{equation}
where we use that the diameter of $S_n(1)$ is $O(1)$. Moreover, in \cite{shalev2009stochastic} it was shown that one can guarantee that with probability $\ge 1-\beta$
\begin{equation}
    \cW_\mu(\hat p_\mu^*) - \cW_\mu(p_\mu^*) \overset{\eqref{eq:str_convex_erm_argmin_property}}{=} O\left(\frac{n\|C\|_\infty^2}{\beta\mu m}\right).\label{eq:erm_rm_difference_beta}
\end{equation}
Taking advantages of both inequalities we get that if
\begin{equation}
    m = \widetilde{\Omega}\left(\min\left\{\max\left\{\frac{n\|C\|_\infty^2}{\mu^2\varepsilon^2},\frac{n\|C\|_\infty^2}{\mu^2\varepsilon^4}\right\},\frac{n\|C\|_\infty^2}{\beta\mu^2\e^2}\right\}\right) = \widetilde{\Omega}\left(n\min\left\{\frac{\|C\|_\infty^2}{\mu^2\varepsilon^4},\frac{\|C\|_\infty^2}{\beta\mu^2\e^2}\right\}\right)\label{eq:barycenters_needed_sample_size}
\end{equation}
then with probability at least $1-\frac{\beta}{2}$
\begin{equation}
    \|\hat p_\mu^* - p_\mu^*\|_2 \le \sqrt{\frac{2}{\mu}\left(\cW_\mu(\hat p_\mu^*) - \cW_\mu(p_\mu^*)\right)} \overset{\eqref{eq:erm_rm_difference_no_beta},\eqref{eq:erm_rm_difference_beta},\eqref{eq:barycenters_needed_sample_size}}{\le} \frac{\e}{2}.\label{eq:population_and_empirical_risks}
\end{equation}
Assuming that we have such $\hat p\in S_n(1)$ that with probability at least $1-\frac{\beta}{2}$ the inequality 
\begin{equation}
    \|\hat p - \hat p_\mu^*\|_2 \le \frac{\e}{2}\label{eq:empirical_barycenter_eps_solution}
\end{equation}
holds, we apply the union bound and get that with probability $\ge 1 - \beta$
\begin{equation}
    \|\hat p - p_\mu^*\|_2 \le \|\hat p - \hat p_\mu^*\|_2 + \|\hat p_\mu^* - p_\mu^*\|_2 \le \e.\label{eq:e_solution_population_barycenter}
\end{equation}

It remains to describe the approach that finds such $\hat p \in S_n(1)$ that satisfies \eqref{eq:e_solution_population_barycenter} with probability at least $1-\beta$. Recall that in this subsection we consider the following problem
\begin{equation}
    \hat\cW_\mu(p) = \frac{1}{m}\sum\limits_{i=1}^m \cW_\mu(p,q^i) \to \min\limits_{p\in S_n(1)}.
\end{equation}
For each summand $\cW_\mu(p,q^i)$ in the sum above we have the explicit formula \eqref{eq:dual_function_wasserstein_distance} for the dual function $\cW_{q^i,\mu}^*(\lambda)$. Note that one can compute the gradient of $\cW_{q^i,\mu}^*(\lambda)$ via $O(n^2)$ arithmetical operations. What is more, $\cW_{q^i,\mu}^*(\lambda)$ has a finite-sum structure, so, one can sample $j$-th component of $q^i$ with probability $q_j^i$ and get stochastic gradient
\begin{equation}
    \nabla \cW_{q^i,\mu}^*(\lambda,j) = \mu\nabla\left(\ln\left(\frac{1}{q_j^i}\sum\limits_{i=1}^n\exp\left(\frac{-C_{ij} + \lambda_i}{\mu}\right)\right)\right)\label{eq:stoch_grad_dual_wasserstein}
\end{equation}
which requires $O(n)$ arithmetical operations to be computed.

We start with the simple situation. Assume that each measures $q^i$ are stored on $m$ separate machines that form some network with Laplacian matrix $\overline{W}\in\R^{m\times m}$. For this scenario we can apply the dual approach described in Section~\ref{sec:distributed_opt} and apply bounds from Table~\ref{tab:stochastic_bounds_dual}. If for all $i= 1,\ldots, m$ the $i$-th node computes the full gradient of dual functions $\cW_{q^i,\mu}$ at each iteration then in order to find such a point $\hat p$ that with probability at least $1-\frac{\beta}{2}$
\begin{equation}
    \hat\cW_{\mu}(\hat p) - \hat\cW_{\mu}(\hat p_\mu^*) \le \hat\e,\label{eq:empirical_functional_gap}
\end{equation}
where $W = \overline{W}\otimes I_n$, this approach requires $\widetilde{O}\left(\sqrt{\frac{n\|C\|_\infty^2}{\mu\hat\e}\chi(W)}\right)$ communication rounds and $\widetilde{O}\left(n^{2.5}\sqrt{\frac{\|C\|_\infty^2}{\mu\hat\e}\chi(W)}\right)$ arithmetical operations per node to find gradients $\nabla \cW_{q^i,\mu}^*(\lambda)$. If instead of full gradients workers use stochastic gradients $\nabla \cW_{q^i,\mu}^*(\lambda,j)$ defined in \eqref{eq:stoch_grad_dual_wasserstein} and these stochastic gradients have light-tailed distribution, i.e.\ satisfy the condition \eqref{eq:dual_light_tails_stoch_grad_distrib} with parameter $\sigma > 0$, then to guarantee \eqref{eq:empirical_functional_gap} with probability $\ge 1-\frac{\beta}{2}$ the aforementioned approach needs the same number of communications rounds and $\widetilde{O}\left(n\max\left\{\sqrt{\frac{n\|C\|_\infty^2}{\mu\hat\e}\chi(W)}, \frac{m\sigma^2n\|C\|_\infty^2}{\hat\e^2}\chi(W)\right\}\right)$ arithmetical operations per node to find gradients $\nabla \cW_{q^i,\mu}^*(\lambda, j)$. Using $\mu$-strong convexity of $\cW_\mu(p,q^i)$ for all $i=1,\ldots,m$ and taking $\hat\e = \frac{\mu\e^2}{8}$ we get that our approach finds such a point $\hat p$ that satisfies \eqref{eq:empirical_barycenter_eps_solution} with probability at least $1-\frac{\beta}{2}$ using
\begin{equation}
    \widetilde{O}\left(\frac{\sqrt{n}\|C\|_\infty}{\mu\e}\sqrt{\chi(W)}\right) \quad \text{communication rounds}\label{eq:communic_rounds_dual_barycenters}
\end{equation}
and 
\begin{equation}
    \widetilde{O}\left(n^{2.5}\frac{\|C\|_\infty}{\mu\e}\sqrt{\chi(W)}\right)\label{eq:arithm_opers_deterministic_dual}
\end{equation}
arithmetical operations per node to find gradients in the deterministic case and
\begin{equation}
    \widetilde{O}\left(n\max\left\{\frac{\sqrt{n}\|C\|_\infty}{\mu\e}\sqrt{\chi(W)}, \frac{m\sigma^2n\|C\|_\infty^2}{\mu^2\e^4}\chi(W)\right\}\right)\notag
\end{equation}
arithmetical operations per node to find stochastic gradients in the stochastic case. However, the state-of-the-art theory of learning states (see \eqref{eq:barycenters_needed_sample_size}) that $m$ should so large that in the stochastic case the second term in the bound for arithmetical operations typically dominates the first term and the dimensional dependence reduction from $n^{2.5}$ in the deterministic case to $n^{1.5}$ in the stochastic case is typically negligible in comparison with how much $\frac{m\sigma^2\sqrt{n}\|C\|_\infty^2}{\mu^2\e^4}\chi(W)$ is larger than $\frac{\|C\|_\infty}{\mu\e}\sqrt{\chi(W)}$. That is, our theory says that it is better to use full gradients in the particular example considered in this section (see also Section~\ref{sec:discussion}). Therefore, further in the section we will assume that $\sigma^2 = 0$, i.e.\ workers use full gradients of dual functions $\cW_{q^i,\mu}^*(\lambda)$.

However, bounds \eqref{eq:communic_rounds_dual_barycenters}-\eqref{eq:arithm_opers_deterministic_dual} were obtained under very restrictive at the first sight assumption that we have $m$ workers and each worker stores only one measure which is unrealistic. One can relax this assumption in the following way. Assume that we have $\hat l < m$ machines connected in a network with Laplacian matrix $\hat{W}$ and $j$-th machine stores $\hat m_j \ge 1$ measures for $j=1,\ldots, \hat l$ and $\sum_{j=1}^{\hat l}\hat m_j = m$. Next, for $j$-th machine we introduce $\hat m_j$ virtual workers also connected in some network that $j$-th machine can emulate along with communication between virtual workers and for every virtual worker we arrange one measure, e.g.\ it can be implemented as an array-like data structure with some formal rules for exchanging the data between cells that emulates communications. We also assume that inside the machine we can set the preferable network for the virtual nodes in such a way that each machine emulates communication between virtual nodes and computations inside them fast enough. Let us denote the Laplacian matrix of the obtained network of $m$ virtual nodes as $\overline{W}$. Then, our approach finds such a point $\hat p$ that satisfies \eqref{eq:empirical_barycenter_eps_solution} with probability at least $1-\frac{\beta}{2}$ using
\begin{equation}
    \widetilde{O}\left(\underbrace{\left(\max\limits_{j=1,\ldots,\hat l}T_{\text{cm},j}\right)}_{T_{\text{cm},\max}}\frac{\sqrt{n}\|C\|_\infty}{\mu\e}\sqrt{\chi(W)}\right)\label{eq:communic_rounds_dual_barycenters_virtual}
\end{equation}
time to perform communications and 
\begin{equation}
    \widetilde{O}\left(\underbrace{\left(\max\limits_{j=1,\ldots,\hat l}T_{\text{cp},j}\right)}_{T_{\text{cp},\max}} n^{2.5}\frac{\|C\|_\infty}{\mu\e}\sqrt{\chi(W)}\right)\label{eq:arithm_opers_deterministic_dual_virtual}
\end{equation}
time for arithmetical operations per machine to find gradients where $T_{\text{cm},j}$ is time needed for $j$-th machine to emulate communication between corresponding virtual nodes at each iteration and $T_{\text{cp},j}$ is time required by $j$-th machine to perform $1$ arithmetical operation for all corresponding virtual nodes in the gradients computation process at each iteration. For example, if we have only one machine and network of virtual nodes forms a complete graph than $\chi(W) = 1$, but $T_{\text{cm},\max}$ and $T_{\text{cp},\max}$ can be large and to reduce the running time one should use more powerful machine. In contrast, if we have $m$ machines connected in a star-graph than $T_{\text{cm},\max}$ and $T_{\text{cp},\max}$ will be much smaller, but $\chi(W)$ will be of order $m$ which is large. Therefore, it is very important to choose balanced architecture of the network at least for virtual nodes per machine if it is possible. This question requires a separate thorough study and lies out of scope of this paper.

\subsubsection{SA vs SAA comparison}\label{sec:sa_vs_saa_comparison}
Recall that in SA approach we assume that it is possible to sample new measures in online regime which means that the computational process is performed on one machine, whereas in SAA approach we assume that large enough collection of measures is distributed among the network of machines that form some computational network. In practice measures from $S_n(1)$ correspond to some images. As one can see from the complexity bounds, both SA and SAA approaches require large number of samples to learn the population barycenter defined in \eqref{eq:population_barycenter}. If these samples are images, then they typically cannot be stored in RAM of one computer. Therefore, it is natural to use distributed systems to store the data. 

Now let us compare complexity bounds for SA and SAA. We summarize them in Table~\ref{tab:sa_saa_comparison}.
\begin{table}[ht!]
    \centering
    \begin{tabular}{|c|c|c|c|}
         \hline
         Approach & Complexity\\
         \hline
         SA &  \makecell{$\widetilde{O}\left(\frac{n^3\|C\|_\infty^2}{\mu^2\e^2}\min\left\{\exp\left(\frac{\|C\|_\infty}{\mu}\right)\left(\frac{\|C\|_\infty}{\mu} + \ln\left(\frac{\|C\|_\infty}{\gamma\mu^2\e^4}\right)\right), \sqrt{\frac{n}{\gamma\mu^3\e^4}}\right\}\right)$\\
         arithmetical operations}\\
         \hline
         \makecell{SA,\\
         the 2-d term\\
         is smaller} &  \makecell{$\widetilde{O}\left(\frac{n^{3.5}\|C\|_\infty^2}{\sqrt{\gamma}\mu^{3.5}\e^4}\right)$ arithmetical operations}\\
         \hline
         SAA & \makecell{$\widetilde{O}\left(T_{\text{cm},\max}\frac{\sqrt{n}\|C\|_\infty}{\mu\e}\sqrt{\chi(W)}\right)$ time to perform communications,\\
         $ \widetilde{O}\left(T_{\text{cp},\max} n^{2.5}\frac{\|C\|_\infty}{\mu\e}\sqrt{\chi(W)}\right)$ time for arithmetical operations per machine,\\
         where $m = \widetilde{\Omega}\left(n\min\left\{\frac{\|C\|_\infty^2}{\mu^2\varepsilon^4},\frac{\|C\|_\infty^2}{\beta\mu^2\e^2}\right\}\right)$} \\
         \hline
         \makecell{SAA,\\
         $\chi(W) = \Omega(m)$,\\
         $T_{\text{cm},\max} = O(1)$,\\
         $T_{\text{cp},\max} = O(1)$,\\
         $\sqrt{\beta} \ge \e$}& \makecell{$\widetilde{O}\left(\frac{n\|C\|_\infty^2}{\sqrt{\beta}\mu^2\e^2}\right)$ communication rounds,\\
         $\widetilde{O}\left(\frac{n^3\|C\|_\infty^2}{\sqrt{\beta}\mu^2\e^2}\right)$ arithmetical operations per machine} \\
         \hline
    \end{tabular}
    \caption{Complexity bounds for SA and SAA approaches for computation of population barycenter defined in \eqref{eq:population_barycenter} with accuracy $\e$. The third row states the complexity bound for SA approach when the second term under the minimum in \eqref{eq:sa_overall_complexity} is dominated by the first one, e.g.\ when $\mu$ is small enough. The last row corresponds to the case when $T_{\text{cm},\max} = O(1)$, $T_{\text{cp},\max} = O(1)$, $\sqrt{\beta} \ge \e$, e.g.\ $\beta = 0.01$ and $\e \le 0.1$, and the communication network is star-like, which implies $\chi(W) = \Omega(m)$}
    \label{tab:sa_saa_comparison}
\end{table}
When the communication is fast enough and $\mu$ is small we typically have that SAA approach significantly outperforms SA approach in terms of the complexity as well even for communication architectures with big $\chi(W)$. Therefore, for balanced architecture one can expect that SAA approach will outperform SA even more.

To conclude, we state that population barycenter computation is a natural example when it is typically much more preferable to use distributed algorithms with dual oracle instead of SA approach in terms of memory and complexity bounds.



\section{Derivative-Free Distributed Optimization}\label{beznosikov}

As mentioned above in Section \ref{gorbunov}, the decentralized optimization problem can be rewritten as a problem with affine constraints:
\begin{equation}
    \label{temp1}
    \min\limits_{\substack{\sqrt{W}\x = 0, \\ x_1,\ldots, x_m \in Q}} f(\x) = \frac{1}{m}\sum\limits_{i=1}^m f_i(x_i),
\end{equation}
where we use matrix $W \eqdef \overline{W}\otimes I_n$ for Laplacian matrix $\overline{W} = \|\overline{W}_{ij}\|_{i,j=1,1}^{m,m} \in \R^{m\times m}$ of the connection graph. In turn, the problem with affine constraints:
\begin{equation*}
\min_{Ax=0, x\in Q} f(x),    
\end{equation*}
is rewritten in a penalized form as follows:
\begin{equation}
\label{temp2}
\min_{x\in Q} F(x) = f(x) + \frac{R_y^2}{\varepsilon}\| Ax\|_2^2,  
\end{equation}
with some positive constants $\varepsilon$ and $R_y$ (for details see Section \ref{gorbunov}). As a result, we have a classical composite optimization problem, therefore this section will focus on this problem. In what follows, we will rely on work \cite{beznosikov2019derivative}. Note that the work \cite{stepanov2021one} with a similar results has recently appeared (unlike work \cite{beznosikov2019derivative}, it considers a more practical one-point feedback -- for a more detailed explanation of the difference, see \cite{stepanov2021one}). \arrev{Note also, that results of \cite{beznosikov2019derivative,stepanov2021one} can be generalized for saddle-point problems by using proper version of Sliding technique \cite{lan2021mirror}.} We will find out a method based on the Sliding Algorithm (see \cite{lan2016gradient} and Section \ref{gorbunov}) for the convex composite optimization problem with smooth and non-smooth terms. One can find gradient-free methods for distributed optimization in the literature (see \cite{ligf2014,tang2020distributed}), but the method that will be discussed further is the first, which combines zeroth-order and first-order oracles. Its uses the first-order oracle for the smooth part and the zeroth-order oracle for the non-smooth part. 

\subsection{Theoretical part}

\subsubsection{Convex Case}\label{sec:main_res_cvx}

We consider\footnote{The narrative in this section follows \cite{beznosikov2019derivative}.} the composite optimization problem
\begin{equation}
    \label{problem_orig} 
    \min\limits_{x \in Q} \Psi_0(x) = f(x) + g(x).
\end{equation}

In this part of paper, we will work not in the Euclidean norm $\|\cdot\|_2$, but in a certain norm $\| \cdot\|$ (and the dual norm $\|\cdot\|_*$ for the norm $\|\cdot\|$). Also define the Bregman divergence associated with some function $\nu(x)$, which is $1$-strongly convex w.r.t. $\|\cdot\|$-norm and differentiable on $Q$, as follows
    \begin{equation*}
        V(x,y) = \nu(y) - \nu(x) - \la\nabla \nu(x), y-x \ra ,\quad \forall x,y\in Q.
    \end{equation*}
The use of Bregman divergence and special norms allows taking into account the geometric setup of the problem. For example, when we work with the problem in a probability simplex, it seems natural to use the $\|\cdot\|_1$-norm and the Kullback--Leibler divergence.

Next, we introduce some assumptions for problem \eqref{problem_orig}: $Q\subseteq \R^n$ is a compact and convex set with diameter $D_Q$ in $\|\cdot\|$-norm, function $g$ is convex and $L$-smooth on $Q$ w.r.t. norm $\|\cdot\|$, i.e.
    \begin{equation*}
        \|\nabla g(x) - \nabla g(y)\|_* \le L\|x - y\|,\quad \forall x,y\in Q,
    \end{equation*}
$f$ is convex differentiable function on $Q$. 

Assume that we have an access to the first-order oracle for $g$, i.e. gradient $\nabla g(x)$ is available, and to the biased stochastic zeroth-order oracle for $f$ (see also \cite{gorbunov2018accelerated, beznosikov2020gradient}) that for a given point $x$ returns noisy value $\tilde f(x, \xi)$ such that
\begin{equation}
    \label{tilde_f} 
    \tilde{f}(x, \xi) = f(x, \xi) + \Delta(x),
\end{equation}
where $\Delta(x)$ is a bounded noise of unknown nature
\begin{equation*}
    |\Delta(x)| \leq \Delta
\end{equation*}
and random variable $\xi$ is such that
\begin{equation*}
    \mathbb{E}[f(x,\xi)] = f(x).
\end{equation*}
Additionally, we assume that for all $x \in Q_s$ ($s \le D_Q$)
\begin{equation*}
    \|\nabla f(x,\xi) \|_2 \leq M(\xi),\quad
    \mathbb{E}[M^2(\xi)] = M^2.
\end{equation*}
It is important to note that for the function $f(x)$ these assumptions are made only for theoretical estimates; we have no real access to $\nabla f(x)$. The question is how to replace the gradient of the function $f(x)$. The easiest way is to collect gradient completely using finite differences:
\begin{eqnarray}
    \label{full_coor}
     f'_{\text{full}}\arrev{(x, \xi)}= \frac{1}{{r}}\sum\limits_{i=1}^{n}\left(\tilde f(x+r h_i, \xi) - \tilde f(x-r h_i, \xi) \right)h_i,
\end{eqnarray}
here we consider a standard orthogonal normalized basis $\{h_1, \ldots, h_{n}\}$. This way we really get a vector close to the gradient. The obvious disadvantage of this method is that one need to call the oracle for $\tilde f(x, \xi)$ $2n$ times. Another way is to use random direction $e$ uniformly distributed on the Euclidean sphere (see \cite{Nesterov,Shamir15}):
\begin{equation}
    \label{oracle_f} 
    \tilde{f}_r'(x, \xi , e) = \frac{n}{2 r} (\tilde{f}(x + r e , \xi) - \tilde{f}(x - r e, \xi) ) e.
\end{equation}
In particular, the authors of \cite{beznosikov2019derivative} use this approximation. 

Now another problem arises -- we need to combine the zeroth-order and first-order oracles for different parts of the composite problem. It seems natural that the gradient-free oracle should be called more often than the gradient one. The authors of paper \cite{beznosikov2019derivative} solve this problem and propose to apply the algorithm based on Lan's Sliding \cite{lan2016gradient}. The basic idea is that we fix $\nabla g$ and iterate through the inner loop ($\text{\tt PS}$ procedure), changing only the point $x$ in $\tilde{f}_r'(x, \xi, e)$.
\begin{algorithm} [H]
	\caption{Zeroth-Order Sliding Algorithm ({\tt zoSA})}
	\label{alg}
	\begin{algorithmic}
\State
\noindent {\bf Input:} Initial point $x_0 \in Q$ and iteration limit $N$.
\State Let $\beta_k \in \RR_{++}, \gamma_k \in \RR_+$, and $T_k \in {\mathbb N}$, $k = 1, 2, \ldots$,
be given and
set $\overline x_0 = x_0$. 
\For {$k=1, 2, \ldots, N$ }
    \State 1. Set $\underline x_k = (1 - \gamma_k) \overline x_{k-1} + \gamma_k x_{k-1}$,
    and let $h_k(\cdot) \equiv l_g(\underline x_{k}, \cdot)$ be defined in \eqref{lg}.
    \State 2. Set
    \begin{equation*}
        (x_k, \tilde x_k) = \text{\tt PS}(h_k, x_{k-1}, \beta_k, T_k);
    \end{equation*}
    \State 3. Set $\overline x_k = (1-\gamma_k) \overline x_{k-1} + \gamma_k \tilde x_k$. 
\EndFor
\State 
\noindent {\bf Output:} $\overline x_N$.

\Statex
\Statex The  $\text{\tt{PS} }$(prox-sliding) procedure.

\State {\bf procedure:} {$(x^+, \tilde x^+) = \text{\tt{PS}}$($h$, $x$, $\beta$, $T$)}
\State Let the parameters $p_t \in \R_{++}$ and $\theta_t \in [0,1]$,
$t = 1, \ldots$, be given. Set $u_{0} = \tilde u_0 = x$.
\State {\bf for} $t = 1, 2, \ldots, T$ {\bf do}
    \begin{eqnarray*}
        u_{t} &=& \argmin_{u \in Q} \Big\{h(u) + \langle \tilde{f}_r'(u_{t-1}, \xi_{t-1}, e_{t-1}), u \rangle+\beta V(x, u) + \beta p_t V(u_{t-1}, u)\Big\},\\
        \tilde u_t &=&  (1-\theta_t) \tilde u_{t-1} + \theta_t u_t.
    \end{eqnarray*}
\State {\bf end for}
\State Set $x^+ = u_T$ and  $\tilde x^+ = \tilde u_T$.
\State {\bf end procedure:}
\end{algorithmic}
\end{algorithm}
In the Algorithm~\ref{alg} we need the following function
\begin{equation}
    \label{lg} 
    l_g(x,y) = g(x) + \langle \nabla g(x), y-x \rangle.
\end{equation}
It is important that the random variables $\xi_t$ are independent, and also $e_t$ is sampled independently from previous iterations. 

We also note that {\tt zoSA} (in contrast to the basic version -- Algorithm \ref{alg:sliding}) takes into account the geometric setting of the problem and uses Bregman divergence $V(x,y)$ instead of the standard Euclidean distance in prox-sliding procedure.

Next, we will briefly talk about the convergence of this method (see the full version of the analysis in \cite{beznosikov2019derivative}). First of all, we note the universal technical lemmas that forms a general approach to working with gradient-free methods for non-smooth functions. But before that we introduce a new notation:
\begin{equation}
   \label{F} 
    F(x) = \EE_e [f(x + r e)].
\end{equation}
$F (x)$ is called the smoothed function of $f(x)$. It is important to note that the function $F (x)$ is not calculated by the algorithm, this object is needed only for theoretical analysis. The first lemma states some properties of $F (x)$:
\begin{lemma}\label{lem:lemma_8_shamir_main}
    Assume that differentiable function $f$ defined on $Q_s$ satisfy $\|\nabla f(x)\|_2\le M$ with some constant $M > 0$. Then $F(x)$ defined in \eqref{F} is convex, differentiable and $F(x)$ satisfies
\begin{eqnarray*}
    \sup_{x \in Q} |F(x) - f(x)| \leq rM, \quad \nabla F(x) = \mathbb{E}_{e} \left[\frac{n}{r} f(x + r e)e\right], \quad \|\nabla F(x)\|_* \le \tilde c p_* \sqrt{n}M,
\end{eqnarray*}
where $\tilde c$ is some positive constant independent of $n$ and $p_*$ is determined by the following relation: 
$ \sqrt[4]{\EE[\|e\|_*^4]} \le p_*.$
\end{lemma}
In other words, $F(x)$ provides a good approximation of $f(x)$ for small enough $r$. 
\begin{lemma}\label{lem:second_lemma}
For $\tilde f'_r(x, \xi , e)$ defined in \eqref{oracle_f} the following inequalities hold: 
\begin{equation*}
    \|\EE[ \tilde{f}_r'(x, \xi, e)] - \nabla F(x)\|_* \leq  \frac{n\Delta p_*}{r}, \quad 
    \EE[\|\tilde{f}_r'(x, \xi, e)\|^2_{*}] \leq 2p_*^2\left(cnM^2 + \frac{n^2\Delta^2}{r^2}\right),
\end{equation*}
where $c$ is some positive constant independent of $n$.
\end{lemma}
In other words, one can consider $\tilde{f}_r'(x, \xi, e)$ as a biased stochastic gradient of $F(x)$ with bounded second moment. Therefore, instead of solving \eqref{problem_orig} directly one can focus on the problem
\begin{equation}
    \label{problem} 
    \min\limits_{x \in Q} \Psi(x) = F(x) + g(x)
\end{equation}
with small enough $r$. As mentioned earlier, this approach is universal. In particular, the analysis of gradient-free methods for non-smooth saddle-point problems can be carried out in a similar way \cite{beznosikov2020gradient}. 

Now we will give the main facts from \cite{beznosikov2019derivative} for {\tt zoSA} algorithm itself.
The following theorem states convergence guarantees:
\begin{theorem}\label{cor:main} Suppose that $\{ p_t\}_{t\ge 1}$, $\{\theta_t\}_{t\ge 1}$ are
\begin{equation}
    \label{pt_main}
    p_t = \frac{t}{2}, ~~~ \theta_t = \frac{2(t+1)}{t(t+3)}, ~~~\text{for all } t \geq 1,
\end{equation}
$N$ is given, $\{\beta_k \}$, $\{\gamma_k\}$, $\{T_k\}$ are
\begin{equation}
    \label{k_main}
    \beta_k = \frac{2L}{k},~~~\gamma_k = \frac{2}{k+1},~~~T_k = \frac{C N p_*^2\left(nM^2 + \frac{n^2\Delta^2}{r^2} \right)k^2}{\tilde D L^2}
\end{equation}
with $\tilde D = \nicefrac{3D_{Q,V}^2}{4}$, $D_{Q,V} = \max\{\sqrt{2V(x,y)}\mid x,y \in Q\}$, $D_{Q} = \max\{\|x-y\|\mid x,y \in Q\}$, with some positive constant $C$. Then for all $ N \geq 1$
\begin{equation*}
    \mathbb{E}[\Psi(\overline x_N)- \Psi(x^*)] \leq\frac{12LD_{Q,V}^2}{N(N+1)}+ \frac{n \Delta D_Q p_*}{r}.
\end{equation*}
\end{theorem}

Finally, need to connect the result above to the initial problem \eqref{problem_orig}.
\begin{corollary}\label{cor:main_corollary} Under the assumptions of Theorem~\ref{cor:main} we have that the following inequality holds for all $N \ge 1$:
    \begin{eqnarray}
        \mathbb{E}[\Psi_0(\overline x_N)- \Psi_0(x^*)] &\leq& 2rM + \frac{12LD_{Q,V}^2}{N(N+1)}+ \frac{n \Delta D_Q p_*}{r}.\label{original_main}
    \end{eqnarray}
From \eqref{original_main} it follows that if
\begin{eqnarray*}
   r &=& \Theta\left(\frac{\varepsilon}{M}\right),\quad \Delta = O\left(\frac{\varepsilon^2}{nMD_Q \min\{p_*,1\}}\right)\label{r_delta_main}
\end{eqnarray*}
and $\varepsilon = O\left(\sqrt{n}MD_Q\right)$, then the number of evaluations for $\nabla g$ and $\tilde f'_r$, respectively, required by Algorithm~\ref{alg} to find an $\varepsilon$-solution of \eqref{problem_orig}, i.e.\ such $\overline x_N$ that $\EE[\Psi_0(\overline x_N)] - \Psi_0(x^*) \le \varepsilon$, can be bounded by
\begin{equation}
    \label{bound_in_main}
    O\left(\sqrt{\frac{L D_{Q,V}^2}{\varepsilon}} \right) \text{and} \quad 
  O\left(\sqrt{\frac{L D_{Q,V}^2}{\varepsilon}} + \frac{D_{Q,V}^2p_*^2 nM^2}{\varepsilon^2}\right).
    \end{equation}
\end{corollary}
It is interesting to analyze the obtained results depending on $p^*$, and these constants are determined depending on what geometry we have defined for our problem. For example, if we consider Euclidean proximal setup, i.e. $\|\cdot\| = \|\cdot\|_2$, $V(x,y) = \frac{1}{2}\|x-y\|_2^2,$ $D_{Q,V} = D_Q$. In this case we have $p_*$ and bound \eqref{bound_in_main} for the number of \eqref{tilde_f} oracle calls reduces to
\begin{equation*}
  O\left(\sqrt{\frac{L D_Q^2}{\varepsilon}} + \frac{D_Q^2nM^2}{\varepsilon^2}\right)
\end{equation*}
and the number of $\nabla g(x)$ computations remains the same. It means that our result gives the same number of first-order oracle calls as in the original Gradient Sliding algorithm, while the number of the biased stochastic zeroth-order oracle calls is $n$ times larger in the leading term than in the analogous bound from the original first-order method. In the Euclidean case our bounds reflect the classical dimension dependence for the derivative-free optimization (see \cite{larson2019derivative}).

But if we work on the probability simplex in $\R^n$ and the proximal setup is entropic: $V(x,y)$ is the Kullback--Leibler divergence, i.e.\ $V(x,y) = \sum_{i=1}^n x_i\ln\frac{x_i}{y_i}$. In this situation we have $D_{Q,V} = \sqrt{2\log n}$, $D_{Q} = 2$, $p_* = O\left(\nicefrac{\log(n)}{n}\right)$ \cite{gorbunov2019upper}. Then number of $\nabla g(x)$ calculations is bounded by $O\left( \sqrt{\nicefrac{(L\log^2 n)}{\varepsilon}}\right)$. As for the number of $\tilde{f}'_r(x, \xi, e)$ computations, we get the following bound:
\begin{equation}
\label{temp144}
  O\left(\sqrt{\frac{L \log n}{\varepsilon}} + \frac{M^2 \log^2 n}{\varepsilon^2}\right).
\end{equation}

\subsubsection{Strongly Convex Case}\label{sec:main_res_str_cvx}
In this section we additionally assume that $g$ is $\mu$-strongly convex w.r.t. Bregman divergence $V(x,y)$ \cite{stonyakin2020inexact}, i.e. for all $ x,y\in Q$ 
\begin{equation*}
    g(x) \geq g(y) + \langle \nabla g(y), x - y\rangle + \mu V(x,y).
\end{equation*}

The authors of \cite{beznosikov2019derivative} use restarts technique and get Algorithm~\ref{alg2}. 
\begin{algorithm} [H]
	\caption{The Multi-phase Zeroth-Order Sliding Algorithm ({\tt M-zoSA})}
	\label{alg2}
	\begin{algorithmic}
\State
{\bf Input:} Initial point $y_0 \in Q$ and iteration limit $N_0$, initial estimate $\rho_0$ (s.t. $\Psi (y_0) - \Psi(y^*) \leq \rho_0$)
\For {$i = 1, 2, \ldots, I$ }
    \State Run {\tt zoSA} with $x_0 = y_{i-1}$, $N = N_0$, $\{p_t\}$ and $\{\theta_t\}$ in \eqref{pt_main}, $\{\beta_k\}$ and $\{\gamma_k\}$, $\{T_k\}$ in \eqref{k_main} with $\tilde D = \nicefrac{\rho_0}{\mu 2^i}$, and $y_i$ is output.
\EndFor
\State {\bf Output:} $y_I$.
	\end{algorithmic}
\end{algorithm}

The following theorem states the main complexity results for {\tt M-zoSA}.
\begin{theorem}\label{thm:mzosa_main}
    For {\tt M-zoSA} with $N_0 = 2 \lceil \sqrt{\nicefrac{5L}{\mu}}\rceil$ we have
    \begin{equation*}
        \EE{[\Psi (y_i) - \Psi(y^*)]} \leq \frac{\rho_0}{2^i} + \frac{2 n \Delta D_Q p_*}{r}.
    \end{equation*}
\end{theorem}
Using this we derive the complexity bounds for {\tt M-zoSA}.
\begin{corollary}\label{cor:mzosa_main}
For all $N \ge 1$ the iterates of {\tt M-zoSA} satisfy
    \begin{equation}
        \label{original_strconv}
        \mathbb{E}[\Psi_0(y_i)- \Psi_0(y^*)] \leq 2rM + \frac{\rho_0}{2^i} + \frac{2n \Delta D_Q p_*}{r}.
    \end{equation}
From \eqref{original_strconv} it follows that if
\begin{eqnarray*}
    r = \Theta\left(\frac{\varepsilon}{M}\right),\quad \Delta = O\left(\frac{\varepsilon^2}{nMD_Q\min\{p_*,1\}}\right)
\end{eqnarray*}
and $\varepsilon = O\left(\sqrt{n}MD_Q\right)$, then the number of evaluations for $\nabla g$ and $\tilde f'_r$, respectively, required by Algorithm~\ref{alg2} to find a $\varepsilon$-solution of \eqref{problem_orig} can be bounded by
\begin{equation*}
    O\left(\sqrt{\frac{L}{\mu}} \log_2 \max\left[ 1, \nicefrac{\rho_0}{\varepsilon}\right]\right), \quad 
   O\left(\sqrt{\frac{L}{\mu}} \log_2 \max\left[ 1, \nicefrac{\rho_0}{\varepsilon}\right] + \frac{p_*^2nM^2}{\mu\varepsilon}\right).
    \end{equation*}
\end{corollary}

\subsubsection{From Composite Optimization to Decentralized Distributed Optimization}

Finally, we get an estimate for solving the decentralized optimization problem. With the help of \eqref{temp1} and \eqref{temp2}, we reduce the original decentralized problem to the penalized problem. Next, we need to define parameters of $f$ using parameters of local functions $f_i$. Assume that for each $f_i$ we have $\|\nabla f_i(x_i)\|_2 \le M$ for all $x_i \in Q$, all $f_i$ are convex functions, the starting point is $\x_0^\top = (x_0^\top,\ldots,x_0^\top)^\top$ and $\x_*^\top = (x_*^\top,\ldots, x_*^\top)^\top$ is the optimality point for \eqref{temp1}. Then, one can show that $\|\nabla f(\x)\|_2 \le \nicefrac{M}{\sqrt{m}}$ on the set of such $\x$ that $x_1,\ldots, x_m\in Q$, $D_{Q^m}^2 = mD_Q^2$, $D_{Q^m,V}^2 = mD_{Q,V}^2$ and $R_{\y}^2$ from \eqref{temp2} is $R_{\y}^2 \le \nicefrac{M^2}{m\lambda_{\min}^+(W)}.$ And we have estimates in the Euclidean case: 
\begin{equation*}
    O\left(\sqrt{\frac{\chi(W) M^2 D_Q^2}{\varepsilon^2}} \right) \text{ communication rounds and } \quad 
    O\left(\sqrt{\frac{\chi(W) M^2 D_Q^2}{\varepsilon^2}} +  \frac{nD_Q^2M^2}{\varepsilon^2}\right)  \text{ calculations of $\tilde f(x, \xi)$ per node.}
\end{equation*}
At the same time, when we work on a simplex and use the Kullback-Leibler divergence, we get estimates similar to \eqref{temp144}:
\begin{equation*}
    O\left(\sqrt{\frac{\chi(W) M^2  \log n}{\varepsilon^2}} \right) \text{ communication rounds and } \quad 
    O\left(\sqrt{\frac{\chi(W) M^2 \log n}{\varepsilon^2}} +  \frac{M^2 \log^2 n }{\varepsilon^2}\right)  \text{ calculations of $\tilde f(x, \xi)$ per node.}
\end{equation*}

The bound for the communication rounds matches the lower bound from \cite{scaman2018optimal,scaman2019optimal} and one can note that under above  assumptions the obtained bound for zeroth-order oracle calculations per node is optimal up to polylogarithmic factors in the class of methods with optimal number of communication rounds (see also \cite{dvinskikh2019decentralized,gorbunov2019optimal}). In particular, in the Euclidean case, we lose $n$ times (which corresponds to the case if we were to restore the gradient in the way \eqref{full_coor}), and in the case of a simplex, only in the $\log n$ times.

\section*{Acknowledgements}
Authors are express gratitude to A.~Nazin, A.~Nedich, G.~Scutari, C.~Uribe and P.~Dvurechensky for fruitful discussions.  

The research of A. Gasnikov, A. Beznosikov and. A. Rogozin was partially supported by RFBR, project number 19-31-51001. The research of E. Gorbunov and D. Dvinskikh was partially supported by the Ministry of Science and Higher Education of the Russian Federation (Goszadaniye) № 075-00337-20-03, project No. 0714-2020-0005.

\bibliographystyle{abbrv}
\bibliography{refs}

\begin{thebibliography}{100}

\bibitem{Abadeh_et_al_2015}
S.~Abadeh, P.~Esfahani, and D.~Kuhn.
\newblock Distributionally robust logistic regression.
\newblock In {\em Advances in Neural Information Processing Systems
  (NeurIPS))}, pages 1576--1584, 2015.

\bibitem{aghajan2020distributed}
A.~Aghajan and B.~Touri.
\newblock Distributed optimization over dependent random networks.
\newblock {\em arXiv preprint arXiv:2010.01956}, 2020.

\bibitem{alghunaim2020decentralized}
S.~A. Alghunaim, E.~K. Ryu, K.~Yuan, and A.~H. Sayed.
\newblock Decentralized proximal gradient algorithms with linear convergence
  rates.
\newblock {\em IEEE Transactions on Automatic Control}, 66(6):2787--2794, 2020.

\bibitem{alistarh2017qsgd}
D.~Alistarh, D.~Grubic, J.~Li, R.~Tomioka, and M.~Vojnovic.
\newblock {QSGD}: {C}ommunication-efficient {SGD} via gradient quantization and
  encoding.
\newblock In {\em Advances in Neural Information Processing Systems}, pages
  1709--1720, 2017.

\bibitem{allen2016katyusha}
Z.~Allen-Zhu.
\newblock Katyusha: The first direct acceleration of stochastic gradient
  methods.
\newblock In {\em Proceedings of the 49th Annual ACM SIGACT Symposium on Theory
  of Computing}, STOC 2017, pages 1200--1205, New York, NY, USA, 2017. ACM.
\newblock arXiv:1603.05953.

\bibitem{allen2018make}
Z.~Allen-Zhu.
\newblock How to make the gradients small stochastically: Even faster convex
  and nonconvex sgd.
\newblock In {\em Advances in Neural Information Processing Systems}, pages
  1157--1167, 2018.

\bibitem{anikin2017dual}
A.~S. Anikin, A.~V. Gasnikov, P.~E. Dvurechensky, A.~I. Tyurin, and A.~V.
  Chernov.
\newblock Dual approaches to the minimization of strongly convex functionals
  with a simple structure under affine constraints.
\newblock {\em Computational Mathematics and Mathematical Physics},
  57(8):1262--1276, Aug 2017.

\bibitem{arjevani2015communication}
Y.~Arjevani and O.~Shamir.
\newblock Communication complexity of distributed convex learning and
  optimization.
\newblock In {\em Advances in neural information processing systems}, pages
  1756--1764, 2015.

\bibitem{Arjovsky_et_al2017}
M.~Arjovsky, S.~Chintala, and L.~Bottou.
\newblock Wasserstein generative adversarial networks.
\newblock {\em Proceedings of the 34th International Conference on Machine
  Learning (ICML)}, 70(1):214--223, 2017.

\bibitem{bansal2019potential}
N.~Bansal and A.~Gupta.
\newblock Potential-function proofs for gradient methods.
\newblock {\em Theory of Computing}, 15(1):1--32, 2019.

\bibitem{basu2019qsparse}
D.~Basu, D.~Data, C.~Karakus, and S.~Diggavi.
\newblock Qsparse-local-sgd: Distributed sgd with quantization, sparsification,
  and local computations.
\newblock {\em arXiv preprint arXiv:1906.02367}, 2019.

\bibitem{bayandina2018mirror}
A.~Bayandina, P.~Dvurechensky, A.~Gasnikov, F.~Stonyakin, and A.~Titov.
\newblock Mirror descent and convex optimization problems with non-smooth
  inequality constraints.
\newblock In {\em Large-Scale and Distributed Optimization}, pages 181--213.
  Springer, 2018.

\bibitem{bertsekas1989parallel}
D.~P. Bertsekas and J.~N. Tsitsiklis.
\newblock {\em Parallel and distributed computation: numerical methods},
  volume~23.
\newblock Prentice hall Englewood Cliffs, NJ, 1989.

\bibitem{beznosikov2021decentralized}
A.~Beznosikov, P.~Dvurechensky, A.~Koloskova, V.~Samokhin, S.~U. Stich, and
  A.~Gasnikov.
\newblock Decentralized local stochastic extra-gradient for variational
  inequalities.
\newblock {\em arXiv preprint arXiv:2106.08315}, 2021.

\bibitem{beznosikov2019derivative}
A.~Beznosikov, E.~Gorbunov, and A.~Gasnikov.
\newblock Derivative-free method for composite optimization with applications
  to decentralized distributed optimization.
\newblock {\em IFAC-PapersOnLine}, 53(2):4038--4043, 2020.

\bibitem{beznosikov2020biased}
A.~Beznosikov, S.~Horv{\'a}th, P.~Richt{\'a}rik, and M.~Safaryan.
\newblock On biased compression for distributed learning.
\newblock {\em arXiv preprint arXiv:2002.12410}, 2020.

\bibitem{beznosikov2021optimal_}
A.~Beznosikov, D.~Kovalev, A.~Sadiev, P.~Richtarik, and A.~Gasnikov.
\newblock Optimal distributed algorithms for stochastic variational
  inequalities.
\newblock {\em arXiv preprint}, 2021.

\bibitem{beznosikov2021optimal}
A.~Beznosikov, A.~Rogozin, D.~Kovalev, and A.~Gasnikov.
\newblock Near-optimal decentralized algorithms for saddle point problems over
  time-varying networks.
\newblock In {\em International Conference on Optimization and Applications},
  pages 246--257. Springer, 2021.

\bibitem{beznosikov2020gradient}
A.~Beznosikov, A.~Sadiev, and A.~Gasnikov.
\newblock Gradient-free methods with inexact oracle for convex-concave
  stochastic saddle-point problem.
\newblock In {\em International Conference on Mathematical Optimization Theory
  and Operations Research}, pages 105--119. Springer, 2020.

\bibitem{beznosikov2021distributed}
A.~Beznosikov, G.~Scutari, A.~Rogozin, and A.~Gasnikov.
\newblock Distributed saddle-point problems under data similarity.
\newblock {\em Advances in Neural Information Processing Systems}, 34, 2021.

\bibitem{blanchet2018towards}
J.~Blanchet, A.~Jambulapati, C.~Kent, and A.~Sidford.
\newblock Towards optimal running times for optimal transport.
\newblock {\em arXiv preprint arXiv:1810.07717}, 2018.

\bibitem{boyd2006randomized}
S.~Boyd, A.~Ghosh, B.~Prabhakar, and D.~Shah.
\newblock Randomized gossip algorithms.
\newblock {\em IEEE transactions on information theory}, 52(6):2508--2530,
  2006.

\bibitem{cesa-bianchi2002generalization}
N.~Cesa-bianchi, A.~Conconi, and C.~Gentile.
\newblock On the generalization ability of on-line learning algorithms.
\newblock In T.~G. Dietterich, S.~Becker, and Z.~Ghahramani, editors, {\em
  Advances in Neural Information Processing Systems 14}, pages 359--366. MIT
  Press, 2002.

\bibitem{chambolle2011first}
A.~Chambolle and T.~Pock.
\newblock A first-order primal-dual algorithm for convex problems with
  applications to imaging.
\newblock {\em Journal of mathematical imaging and vision}, 40(1):120--145,
  2011.

\bibitem{cuturi2016smoothed}
M.~Cuturi and G.~Peyré.
\newblock A smoothed dual approach for variational wasserstein problems.
\newblock {\em SIAM Journal on Imaging Sciences}, 9(1):320--343, 2016.

\bibitem{defazio2014saga}
A.~Defazio, F.~Bach, and S.~Lacoste-Julien.
\newblock Saga: A fast incremental gradient method with support for
  non-strongly convex composite objectives.
\newblock In {\em Proceedings of the 27th International Conference on Neural
  Information Processing Systems}, NIPS'14, pages 1646--1654, Cambridge, MA,
  USA, 2014. MIT Press.

\bibitem{devolder2013exactness}
O.~Devolder.
\newblock {\em Exactness, inexactness and stochasticity in first-order methods
  for large-scale convex optimization}.
\newblock PhD thesis, PhD thesis, ICTEAM and CORE, Universit{\'e} Catholique de
  Louvain, 2013.

\bibitem{devolder2013first}
O.~Devolder, F.~Glineur, and Y.~Nesterov.
\newblock First-order methods with inexact oracle: the strongly convex case.
\newblock {\em CORE Discussion Papers}, 2013016:47, 2013.

\bibitem{devolder2014first}
O.~Devolder, F.~Glineur, and Y.~Nesterov.
\newblock First-order methods of smooth convex optimization with inexact
  oracle.
\newblock {\em Mathematical Programming}, 146(1):37--75, 2014.

\bibitem{dvinskikh2020sa}
D.~Dvinskikh.
\newblock Stochastic approximation versus sample average approximation for
  population wasserstein barycenters.
\newblock {\em arXiv preprint arXiv:2001.07697}, 2020.

\bibitem{dvinskikh2021decentralized}
D.~Dvinskikh.
\newblock Decentralized algorithms for wasserstein barycenters.
\newblock {\em arXiv preprint arXiv:2105.01587}, 2021.

\bibitem{dvinskikh2019decentralized}
D.~Dvinskikh and A.~Gasnikov.
\newblock Decentralized and parallel primal and dual accelerated methods for
  stochastic convex programming problems.
\newblock {\em Journal of Inverse and Ill-posed Problems}, 29(3):385--405,
  2021.

\bibitem{dvinskikh2020parallel}
D.~Dvinskikh, A.~Gasnikov, A.~Rogozin, and A.~Beznosikov.
\newblock Parallel and distributed algorithms for ml problems.
\newblock {\em arXiv preprint arXiv:2010.09585}, 2020.

\bibitem{dvinskikh2019dual}
D.~Dvinskikh, E.~Gorbunov, A.~Gasnikov, P.~Dvurechensky, and C.~A. Uribe.
\newblock On primal and dual approaches for distributed stochastic convex
  optimization over networks.
\newblock In {\em 2019 IEEE 58th Conference on Decision and Control (CDC)},
  pages 7435--7440. IEEE, 2019.

\bibitem{dvinskikh2020improved}
D.~Dvinskikh and D.~Tiapkin.
\newblock Improved complexity bounds in wasserstein barycenter problem.
\newblock In {\em International Conference on Artificial Intelligence and
  Statistics}, pages 1738--1746. PMLR, 2021.

\bibitem{dvinskikh2020accelerated}
D.~M. Dvinskikh, A.~I. Turin, A.~V. Gasnikov, and S.~S. Omelchenko.
\newblock Accelerated and non accelerated stochastic gradient descent in model
  generality.
\newblock {\em Matematicheskie Zametki}, 108(4):515--528, 2020.

\bibitem{dvurechenskii2018decentralize}
P.~Dvurechenskii, D.~Dvinskikh, A.~Gasnikov, C.~Uribe, and A.~Nedich.
\newblock Decentralize and randomize: Faster algorithm for wasserstein
  barycenters.
\newblock In {\em Advances in Neural Information Processing Systems}, pages
  10760--10770, 2018.

\bibitem{dvurechensky2016stochastic}
P.~Dvurechensky and A.~Gasnikov.
\newblock Stochastic intermediate gradient method for convex problems with
  stochastic inexact oracle.
\newblock {\em Journal of Optimization Theory and Applications},
  171(1):121--145, 2016.

\bibitem{dvurechensky2017randomized}
P.~Dvurechensky, A.~Gasnikov, and A.~Tiurin.
\newblock Randomized similar triangles method: A unifying framework for
  accelerated randomized optimization methods (coordinate descent, directional
  search, derivative-free method).
\newblock {\em arXiv:1707.08486}, 2017.

\bibitem{facchinei2007finite}
F.~Facchinei and J.~Pang.
\newblock {\em Finite-Dimensional Variational Inequalities and Complementarity
  Problems}.
\newblock Springer Series in Operations Research and Financial Engineering.
  Springer New York, 2007.

\bibitem{fallah2019robust}
A.~Fallah, M.~Gurbuzbalaban, A.~Ozdaglar, U.~Simsekli, and L.~Zhu.
\newblock Robust distributed accelerated stochastic gradient methods for
  multi-agent networks.
\newblock {\em arXiv preprint arXiv:1910.08701}, 2019.

\bibitem{feldman2019high}
V.~Feldman and J.~Vondrak.
\newblock High probability generalization bounds for uniformly stable
  algorithms with nearly optimal rate.
\newblock {\em arXiv preprint arXiv:1902.10710}, 2019.

\bibitem{foster2019complexity}
D.~Foster, A.~Sekhari, O.~Shamir, N.~Srebro, K.~Sridharan, and B.~Woodworth.
\newblock The complexity of making the gradient small in stochastic convex
  optimization.
\newblock {\em arXiv preprint arXiv:1902.04686}, 2019.

\bibitem{gasnikov2017modern}
A.~Gasnikov.
\newblock Universal gradient descent.
\newblock {\em arXiv preprint arXiv:1711.00394}, 2017.

\bibitem{gasnikov2021accelerated}
A.~Gasnikov, D.~Dvinskikh, P.~Dvurechensky, D.~Kamzolov, V.~Matyukhin,
  D.~Pasechnyuk, N.~Tupitsa, and A.~Chernov.
\newblock Accelerated meta-algorithm for convex optimization problems.
\newblock {\em Computational Mathematics and Mathematical Physics},
  61(1):17--28, 2021.

\bibitem{gasnikov2016gradient-free}
A.~V. Gasnikov, A.~A. Lagunovskaya, I.~N. Usmanova, and F.~A. Fedorenko.
\newblock Gradient-free proximal methods with inexact oracle for convex
  stochastic nonsmooth optimization problems on the simplex.
\newblock {\em Automation and Remote Control}, 77(11):2018--2034, Nov 2016.
\newblock arXiv:1412.3890.

\bibitem{gasnikov2018universal}
A.~V. Gasnikov and Y.~E. Nesterov.
\newblock Universal method for stochastic composite optimization problems.
\newblock {\em Computational Mathematics and Mathematical Physics},
  58(1):48--64, 2018.

\bibitem{ghadimi2012optimal}
S.~Ghadimi and G.~Lan.
\newblock Optimal stochastic approximation algorithms for strongly convex
  stochastic composite optimization i: A generic algorithmic framework.
\newblock {\em SIAM Journal on Optimization}, 22(4):1469--1492, 2012.

\bibitem{ghadimi2013stochastic}
S.~Ghadimi and G.~Lan.
\newblock Stochastic first- and zeroth-order methods for nonconvex stochastic
  programming.
\newblock {\em SIAM Journal on Optimization}, 23(4):2341--2368, 2013.
\newblock arXiv:1309.5549.

\bibitem{Bengio2014}
I.~Goodfellow, J.~Pouget-Abadie, M.~Mirza, B.~Xu, D.~Warde-Farley, S.~Ozair,
  A.~Courville, and Y.~Bengio.
\newblock Generative adversarial nets.
\newblock In {\em Advances in Neural Information Processing Systems
  (NeurIPS))}, pages 2672--2680, 2014.

\bibitem{gorbunov2019optimal}
E.~Gorbunov, D.~Dvinskikh, and A.~Gasnikov.
\newblock Optimal decentralized distributed algorithms for stochastic convex
  optimization.
\newblock {\em arXiv preprint arXiv:1911.07363}, 2019.

\bibitem{gorbunov2018accelerated}
E.~Gorbunov, P.~Dvurechensky, and A.~Gasnikov.
\newblock An accelerated method for derivative-free smooth stochastic convex
  optimization.
\newblock {\em SIOPT (in print)}, 2022.

\bibitem{gorbunov2020local}
E.~Gorbunov, F.~Hanzely, and P.~Richt{\'a}rik.
\newblock Local sgd: Unified theory and new efficient methods.
\newblock {\em arXiv preprint arXiv:2011.02828}, 2020.

\bibitem{gorbunov2019unified}
E.~Gorbunov, F.~Hanzely, and P.~Richt{\'a}rik.
\newblock A unified theory of sgd: Variance reduction, sampling, quantization
  and coordinate descent.
\newblock In {\em International Conference on Artificial Intelligence and
  Statistics}, pages 680--690. PMLR, 2020.

\bibitem{gorbunov2020linearly}
E.~Gorbunov, D.~Kovalev, D.~Makarenko, and P.~Richt{\'a}rik.
\newblock Linearly converging error compensated sgd.
\newblock {\em Advances in Neural Information Processing Systems}, 33, 2020.

\bibitem{gorbunov2019upper}
E.~Gorbunov, E.~A. Vorontsova, and A.~V. Gasnikov.
\newblock On the upper bound for the expectation of the norm of a vector
  uniformly distributed on the sphere and the phenomenon of concentration of
  uniform measure on the sphere.
\newblock {\em Mathematical Notes}, 106, 2019.

\bibitem{gower2019sgd}
R.~M. Gower, N.~Loizou, X.~Qian, A.~Sailanbayev, E.~Shulgin, and P.~Richtarik.
\newblock Sgd: General analysis and improved rates.
\newblock {\em arXiv preprint arXiv:1901.09401}, 2019.

\bibitem{guminov2019accelerated}
S.~Guminov, P.~Dvurechensky, N.~Tupitsa, and A.~Gasnikov.
\newblock On a combination of alternating minimization and nesterov’s
  momentum.
\newblock In {\em International Conference on Machine Learning}, pages
  3886--3898. PMLR, 2021.

\bibitem{hendrikx2020optimal}
H.~Hendrikx, F.~Bach, and L.~Massoulie.
\newblock An optimal algorithm for decentralized finite sum optimization.
\newblock {\em arXiv preprint arXiv:2005.10675}, 2020.

\bibitem{hendrikx2020statistically}
H.~Hendrikx, L.~Xiao, S.~Bubeck, F.~Bach, and L.~Massoulie.
\newblock Statistically preconditioned accelerated gradient method for
  distributed optimization.
\newblock {\em arXiv preprint arXiv:2002.10726}, 2020.

\bibitem{horvath2019natural}
S.~Horvath, C.-Y. Ho, L.~Horvath, A.~N. Sahu, M.~Canini, and P.~Richtarik.
\newblock Natural compression for distributed deep learning.
\newblock {\em arXiv preprint arXiv:1905.10988}, 2019.

\bibitem{horvath2019stochastic}
S.~Horv{\'a}th, D.~Kovalev, K.~Mishchenko, S.~Stich, and P.~Richt{\'a}rik.
\newblock Stochastic distributed learning with gradient quantization and
  variance reduction.
\newblock {\em arXiv preprint arXiv:1904.05115}, 2019.

\bibitem{jakovetic2014fast}
D.~Jakoveti{\'c}, J.~Xavier, and J.~M. Moura.
\newblock Fast distributed gradient methods.
\newblock {\em IEEE Transactions on Automatic Control}, 59(5):1131--1146, 2014.

\bibitem{johnson2013accelerating}
R.~Johnson and T.~Zhang.
\newblock Accelerating stochastic gradient descent using predictive variance
  reduction.
\newblock In {\em Advances in neural information processing systems}, pages
  315--323, 2013.

\bibitem{juditsky2012first-order}
A.~Juditsky and A.~Nemirovski.
\newblock First order methods for non-smooth convex large-scale optimization,
  i: General purpose methods.
\newblock In S.~W. Suvrit~Sra, Sebastian~Nowozin, editor, {\em Optimization for
  Machine Learning}, pages 121--184. Cambridge, MA: MIT Press, 2012.

\bibitem{juditsky2011solving}
A.~Juditsky, A.~Nemirovski, and C.~Tauvel.
\newblock Solving variational inequalities with stochastic mirror-prox
  algorithm.
\newblock {\em Stochastic Systems}, 1(1):17--58, 2011.

\bibitem{juditsky2014deterministic}
A.~Juditsky and Y.~Nesterov.
\newblock Deterministic and stochastic primal-dual subgradient algorithms for
  uniformly convex minimization.
\newblock {\em Stochastic Systems}, 4(1):44--80, 2014.

\bibitem{kairouz2019advances}
P.~Kairouz, H.~B. McMahan, B.~Avent, A.~Bellet, M.~Bennis, A.~N. Bhagoji,
  K.~Bonawitz, Z.~Charles, G.~Cormode, R.~Cummings, et~al.
\newblock Advances and open problems in federated learning.
\newblock {\em arXiv preprint arXiv:1912.04977}, 2019.

\bibitem{kakade2009duality}
S.~Kakade, S.~Shalev-Shwartz, and A.~Tewari.
\newblock On the duality of strong convexity and strong smoothness: Learning
  applications and matrix regularization.
\newblock {\em Unpublished Manuscript, http://ttic. uchicago.
  edu/shai/papers/KakadeShalevTewari09.pdf}, 2(1), 2009.

\bibitem{karimireddy2019scaffold}
S.~P. Karimireddy, S.~Kale, M.~Mohri, S.~J. Reddi, S.~U. Stich, and A.~T.
  Suresh.
\newblock Scaffold: Stochastic controlled averaging for federated learning.
\newblock {\em arXiv preprint arXiv:1910.06378}, 2019.

\bibitem{karimireddy2019error}
S.~P. Karimireddy, Q.~Rebjock, S.~U. Stich, and M.~Jaggi.
\newblock Error feedback fixes signsgd and other gradient compression schemes.
\newblock {\em arXiv preprint arXiv:1901.09847}, 2019.

\bibitem{bayoumi2020tighter}
A.~Khaled, K.~Mishchenko, and P.~Richt{\'a}rik.
\newblock Tighter theory for local sgd on identical and heterogeneous data.
\newblock In {\em International Conference on Artificial Intelligence and
  Statistics}, pages 4519--4529, 2020.

\bibitem{kibardin1979decomposition}
V.~Kibardin.
\newblock Decomposition into functions in the minimization problem.
\newblock {\em Avtomatika i Telemekhanika}, (9):66--79, 1979.

\bibitem{koloskova2021improved}
A.~Koloskova, T.~Lin, and S.~U. Stich.
\newblock An improved analysis of gradient tracking for decentralized machine
  learning.
\newblock {\em Advances in Neural Information Processing Systems}, 34, 2021.

\bibitem{koloskova2020unified}
A.~Koloskova, N.~Loizou, S.~Boreiri, M.~Jaggi, and S.~U. Stich.
\newblock A unified theory of decentralized sgd with changing topology and
  local updates.
\newblock {\em ICML 2020, arXiv preprint arXiv:2003.10422}, 2020.

\bibitem{kovalev2021lower}
D.~Kovalev, E.~Gasanov, A.~Gasnikov, and P.~Richtarik.
\newblock Lower bounds and optimal algorithms for smooth and strongly convex
  decentralized optimization over time-varying networks.
\newblock {\em Advances in Neural Information Processing Systems}, 34, 2021.

\bibitem{kovalev2020optimal}
D.~Kovalev, A.~Salim, and P.~Richt{\'a}rik.
\newblock Optimal and practical algorithms for smooth and strongly convex
  decentralized optimization.
\newblock {\em Advances in Neural Information Processing Systems}, 33, 2020.

\bibitem{kovalev2021adom}
D.~Kovalev, E.~Shulgin, P.~Richt{\'a}rik, A.~Rogozin, and A.~Gasnikov.
\newblock Adom: Accelerated decentralized optimization method for time-varying
  networks.
\newblock {\em arXiv preprint arXiv:2102.09234}, 2021.

\bibitem{kroshnin2019complexity}
A.~Kroshnin, N.~Tupitsa, D.~Dvinskikh, P.~Dvurechensky, A.~Gasnikov, and
  C.~Uribe.
\newblock On the complexity of approximating wasserstein barycenters.
\newblock In {\em International conference on machine learning}, pages
  3530--3540. PMLR, 2019.

\bibitem{kulunchakov2019estimate1}
A.~Kulunchakov and J.~Mairal.
\newblock Estimate sequences for stochastic composite optimization: Variance
  reduction, acceleration, and robustness to noise.
\newblock {\em arXiv preprint arXiv:1901.08788}, 2019.

\bibitem{kulunchakov2019estimate2}
A.~Kulunchakov and J.~Mairal.
\newblock Estimate sequences for variance-reduced stochastic composite
  optimization.
\newblock {\em arXiv preprint arXiv:1905.02374}, 2019.

\bibitem{kulunchakov2019generic}
A.~Kulunchakov and J.~Mairal.
\newblock A generic acceleration framework for stochastic composite
  optimization.
\newblock {\em arXiv preprint arXiv:1906.01164}, 2019.

\bibitem{lan2012optimal}
G.~Lan.
\newblock An optimal method for stochastic composite optimization.
\newblock {\em Mathematical Programming}, 133(1):365--397, Jun 2012.
\newblock Firs appeared in June 2008.

\bibitem{lan2016gradient}
G.~Lan.
\newblock Gradient sliding for composite optimization.
\newblock {\em Mathematical Programming}, 159(1):201--235, Sep 2016.

\bibitem{Lan2019lectures}
G.~Lan.
\newblock Lectures on optimization methods for machine learning.
\newblock {\em e-print}, 2019.

\bibitem{lan2020first}
G.~Lan.
\newblock {\em First-order and Stochastic Optimization Methods for Machine
  Learning}.
\newblock Springer, 2020.

\bibitem{lan2017communication}
G.~Lan, S.~Lee, and Y.~Zhou.
\newblock Communication-efficient algorithms for decentralized and stochastic
  optimization.
\newblock {\em Mathematical Programming}, pages 1--48, 2017.

\bibitem{lan2021mirror}
G.~Lan and Y.~Ouyang.
\newblock Mirror-prox sliding methods for solving a class of monotone
  variational inequalities.
\newblock {\em arXiv preprint arXiv:2111.00996}, 2021.

\bibitem{lan2018random}
G.~Lan and Y.~Zhou.
\newblock Random gradient extrapolation for distributed and stochastic
  optimization.
\newblock {\em SIAM Journal on Optimization}, 28(4):2753--2782, 2018.

\bibitem{lan2016algorithms}
G.~Lan and Z.~Zhou.
\newblock Algorithms for stochastic optimization with expectation constraints.
\newblock {\em arXiv:1604.03887}, 2016.

\bibitem{larson2019derivative}
J.~Larson, M.~Menickelly, and S.~M. Wild.
\newblock Derivative-free optimization methods.
\newblock {\em Acta Numerica}, 28:287--404, 2019.

\bibitem{lee2013distributed}
S.~Lee and A.~Nedic.
\newblock Distributed random projection algorithm for convex optimization.
\newblock {\em IEEE Journal of Selected Topics in Signal Processing},
  7(2):221--229, 2013.

\bibitem{li2020decentralized}
H.~Li, C.~Fang, W.~Yin, and Z.~Lin.
\newblock Decentralized accelerated gradient methods with increasing penalty
  parameters.
\newblock {\em IEEE Transactions on Signal Processing}, 68:4855--4870, 2020.

\bibitem{li2020revisiting}
H.~Li and Z.~Lin.
\newblock Revisiting extra for smooth distributed optimization.
\newblock {\em arXiv preprint arXiv:2002.10110}, 2020.

\bibitem{li2021accelerated}
H.~Li and Z.~Lin.
\newblock Accelerated gradient tracking over time-varying graphs for
  decentralized optimization.
\newblock {\em arXiv preprint arXiv:2104.02596}, 2021.

\bibitem{li2020optimal}
H.~Li, Z.~Lin, and Y.~Fang.
\newblock Optimal accelerated variance reduced extra and diging for strongly
  convex and smooth decentralized optimization.
\newblock {\em arXiv preprint arXiv:2009.04373}, 2020.

\bibitem{ligf2014}
J.~Li, C.~Wu, Z.~Wu, and Q.~Long.
\newblock Gradient-free method for nonsmooth distributed optimization.
\newblock {\em Journal of Global Optimization}, 61, 02 2014.

\bibitem{lin2015universal}
H.~Lin, J.~Mairal, and Z.~Harchaoui.
\newblock A universal catalyst for first-order optimization.
\newblock In {\em Proceedings of the 28th International Conference on Neural
  Information Processing Systems}, NIPS'15, pages 3384--3392, Cambridge, MA,
  USA, 2015. MIT Press.

\bibitem{lin2020near}
T.~Lin, C.~Jin, and M.~I. Jordan.
\newblock Near-optimal algorithms for minimax optimization.
\newblock In {\em Conference on Learning Theory}, pages 2738--2779. PMLR, 2020.

\bibitem{lin2021quasi}
T.~Lin, S.~P. Karimireddy, S.~U. Stich, and M.~Jaggi.
\newblock Quasi-global momentum: Accelerating decentralized deep learning on
  heterogeneous data.
\newblock {\em arXiv preprint arXiv:2102.04761}, 2021.

\bibitem{liu2011accelerated}
J.~Liu and A.~S. Morse.
\newblock Accelerated linear iterations for distributed averaging.
\newblock {\em Annual Reviews in Control}, 35(2):160--165, 2011.

\bibitem{liu2019decentralized}
M.~Liu, W.~Zhang, Y.~Mroueh, X.~Cui, J.~Ross, T.~Yang, and P.~Das.
\newblock A decentralized parallel algorithm for training generative
  adversarial nets.
\newblock In {\em Advances in Neural Information Processing Systems (NeurIPS)},
  2020.

\bibitem{liudecentralized}
W.~Liu, A.~Mokhtari, A.~Ozdaglar, S.~Pattathil, Z.~Shen, and N.~Zheng.
\newblock A decentralized proximal point-type method for non-convex non-concave
  saddle point problems.

\bibitem{liu2019decentralizedprox}
W.~Liu, A.~Mokhtari, A.~Ozdaglar, S.~Pattathil, Z.~Shen, and N.~Zheng.
\newblock A decentralized proximal point-type method for saddle point problems.
\newblock {\em arXiv preprint arXiv:1910.14380}, 2019.

\bibitem{liu2019double}
X.~Liu, Y.~Li, J.~Tang, and M.~Yan.
\newblock A double residual compression algorithm for efficient distributed
  learning.
\newblock {\em arXiv preprint arXiv:1910.07561}, 2019.

\bibitem{mateos2015distributed}
D.~Mateos-N{\'u}nez and J.~Cort{\'e}s.
\newblock Distributed subgradient methods for saddle-point problems.
\newblock In {\em 2015 54th IEEE Conference on Decision and Control (CDC)},
  pages 5462--5467. IEEE, 2015.

\bibitem{minty62}
G.~J. Minty.
\newblock Monotone (nonlinear) operators in {Hilbert} space.
\newblock {\em Duke Mathematical Journal}, 29(3):341 -- 346, 1962.

\bibitem{mishchenko2019distributed}
K.~Mishchenko, E.~Gorbunov, M.~Tak{\'a}{\v{c}}, and P.~Richt{\'a}rik.
\newblock Distributed learning with compressed gradient differences.
\newblock {\em arXiv preprint arXiv:1901.09269}, 2019.

\bibitem{muthukrishnan1998first}
S.~Muthukrishnan, B.~Ghosh, and M.~H. Schultz.
\newblock First-and second-order diffusive methods for rapid, coarse,
  distributed load balancing.
\newblock {\em Theory of computing systems}, 31(4):331--354, 1998.

\bibitem{nedic2020distributed}
A.~Nedic.
\newblock Distributed gradient methods for convex machine learning problems in
  networks: Distributed optimization.
\newblock {\em IEEE Signal Processing Magazine}, 37(3):92--101, 2020.

\bibitem{nedic2017achieving}
A.~Nedic, A.~Olshevsky, and W.~Shi.
\newblock Achieving geometric convergence for distributed optimization over
  time-varying graphs.
\newblock {\em SIAM Journal on Optimization}, 27(4):2597--2633, 2017.

\bibitem{nedic2009distributed}
A.~Nedi{\'c} and A.~Ozdaglar.
\newblock Distributed subgradient methods for multi-agent optimization.
\newblock {\em IEEE Transactions on Automatic Control}, 54(1):48--61, 2009.

\bibitem{nemirovski2004prox}
A.~Nemirovski.
\newblock Prox-method with rate of convergence $o(1/t)$ for variational
  inequalities with lipschitz continuous monotone operators and smooth
  convex-concave saddle point problems.
\newblock {\em SIAM Journal on Optimization}, 15(1):229--251, 2004.

\bibitem{nemirovski2009robust}
A.~Nemirovski, A.~Juditsky, G.~Lan, and A.~Shapiro.
\newblock Robust stochastic approximation approach to stochastic programming.
\newblock {\em SIAM Journal on Optimization}, 19(4):1574--1609, 2009.

\bibitem{nesterov2004introduction}
Y.~Nesterov.
\newblock {\em Introductory Lectures on Convex Optimization: a basic course}.
\newblock Kluwer Academic Publishers, Massachusetts, 2004.

\bibitem{nesterov2012make}
Y.~Nesterov.
\newblock How to make the gradients small.
\newblock {\em Optima}, 88:10--11, 2012.

\bibitem{nesterov2018lectures}
Y.~Nesterov.
\newblock {\em Lectures on convex optimization}, volume 137.
\newblock Springer, 2018.

\bibitem{Nesterov}
Y.~Nesterov and V.~G. Spokoiny.
\newblock Random gradient-free minimization of convex functions.
\newblock {\em Foundations of Computational Mathematics}, 17(2):527--566, 2017.

\bibitem{nguyen2018sgd}
L.~M. Nguyen, P.~H. Nguyen, M.~van Dijk, P.~Richt{\'a}rik, K.~Scheinberg, and
  M.~Tak{\'a}{\v{c}}.
\newblock Sgd and hogwild! convergence without the bounded gradients
  assumption.
\newblock {\em arXiv preprint arXiv:1802.03801}, 2018.

\bibitem{olshevsky2019asymptotic}
A.~Olshevsky, I.~C. Paschalidis, and S.~Pu.
\newblock Asymptotic network independence in distributed optimization for
  machine learning.
\newblock {\em arXiv preprint arXiv:1906.12345}, 2019.

\bibitem{olshevsky2019non}
A.~Olshevsky, I.~C. Paschalidis, and S.~Pu.
\newblock A non-asymptotic analysis of network independence for distributed
  stochastic gradient descent.
\newblock {\em arXiv preprint arXiv:1906.02702}, 2019.

\bibitem{peyre2019computational}
G.~Peyr{\'e}, M.~Cuturi, et~al.
\newblock Computational optimal transport.
\newblock {\em Foundations and Trends{\textregistered} in Machine Learning},
  11(5-6):355--607, 2019.

\bibitem{pu2021distributed}
S.~Pu and A.~Nedi{\'c}.
\newblock Distributed stochastic gradient tracking methods.
\newblock {\em Mathematical Programming}, 187(1):409--457, 2021.

\bibitem{qu2017harnessing}
G.~Qu and N.~Li.
\newblock Harnessing smoothness to accelerate distributed optimization.
\newblock {\em IEEE Transactions on Control of Network Systems},
  5(3):1245--1260, 2017.

\bibitem{qu2019accelerated}
G.~Qu and N.~Li.
\newblock Accelerated distributed nesterov gradient descent.
\newblock {\em IEEE Transactions on Automatic Control}, 2019.

\bibitem{rigollet2018entropic}
P.~Rigollet and J.~Weed.
\newblock Entropic optimal transport is maximum-likelihood deconvolution.
\newblock {\em Comptes Rendus Mathematique}, 356(11-12):1228--1235, 2018.

\bibitem{RobbinsMonro:1951}
H.~Robbins and S.~Monro.
\newblock A stochastic approximation method.
\newblock {\em Annals of Mathematical Statistics}, 22:400--407, 1951.

\bibitem{Rockafellar2015}
R.~T. Rockafellar.
\newblock {\em Convex analysis}.
\newblock Princeton university press, 2015.

\bibitem{rogozin2021decentralized}
A.~Rogozin, A.~Beznosikov, D.~Dvinskikh, D.~Kovalev, P.~Dvurechensky, and
  A.~Gasnikov.
\newblock Decentralized distributed optimization for saddle point problems.
\newblock {\em arXiv preprint arXiv:2102.07758}, 2021.

\bibitem{rogozin2021accelerated}
A.~Rogozin, M.~Bochko, P.~Dvurechensky, A.~Gasnikov, and V.~Lukoshkin.
\newblock An accelerated method for decentralized distributed stochastic
  optimization over time-varying graphs.
\newblock {\em Conference on decision and control}, 2021.

\bibitem{rogozin2019projected}
A.~Rogozin and A.~Gasnikov.
\newblock Projected gradient method for decentralized optimization over
  time-varying networks.
\newblock {\em arXiv preprint arXiv:1911.08527}, 2019.

\bibitem{rogozin2020penalty}
A.~Rogozin and A.~Gasnikov.
\newblock Penalty-based method for decentralized optimization over time-varying
  graphs.
\newblock In {\em International Conference on Optimization and Applications},
  pages 239--256. Springer, 2020.

\bibitem{rogozin2021towards}
A.~Rogozin, V.~Lukoshkin, A.~Gasnikov, D.~Kovalev, and E.~Shulgin.
\newblock Towards accelerated rates for distributed optimization over
  time-varying networks.
\newblock In {\em International Conference on Optimization and Applications},
  pages 258--272. Springer, 2021.

\bibitem{scaman2017optimal}
K.~Scaman, F.~Bach, S.~Bubeck, Y.~T. Lee, and L.~Massouli{\'e}.
\newblock Optimal algorithms for smooth and strongly convex distributed
  optimization in networks.
\newblock In {\em Proceedings of the 34th International Conference on Machine
  Learning-Volume 70}, pages 3027--3036. JMLR. org, 2017.

\bibitem{scaman2019optimal}
K.~Scaman, F.~Bach, S.~Bubeck, Y.~T. Lee, and L.~Massouli{\'e}.
\newblock Optimal convergence rates for convex distributed optimization in
  networks.
\newblock {\em Journal of Machine Learning Research}, 20(159):1--31, 2019.

\bibitem{scaman2018optimal}
K.~Scaman, F.~Bach, S.~Bubeck, L.~Massouli{\'e}, and Y.~T. Lee.
\newblock Optimal algorithms for non-smooth distributed optimization in
  networks.
\newblock In {\em Advances in Neural Information Processing Systems}, pages
  2740--2749, 2018.

\bibitem{schmidt2017minimizing}
M.~Schmidt, N.~Le~Roux, and F.~Bach.
\newblock Minimizing finite sums with the stochastic average gradient.
\newblock {\em Mathematical Programming}, 162(1-2):83--112, 2017.

\bibitem{shalev2014understanding}
S.~Shalev-Shwartz and S.~Ben-David.
\newblock {\em Understanding machine learning: From theory to algorithms}.
\newblock Cambridge university press, 2014.

\bibitem{shalev2009stochastic}
S.~Shalev-Shwartz, O.~Shamir, N.~Srebro, and K.~Sridharan.
\newblock Stochastic convex optimization.
\newblock In {\em COLT}, 2009.

\bibitem{shamir2017optimal}
O.~Shamir.
\newblock An optimal algorithm for bandit and zero-order convex optimization
  with two-point feedback.
\newblock {\em Journal of Machine Learning Research}, 18:52:1--52:11, 2017.
\newblock First appeared in arXiv:1507.08752.

\bibitem{Shamir15}
O.~Shamir.
\newblock An optimal algorithm for bandit and zero-order convex optimization
  with two-point feedback.
\newblock {\em Journal of Machine Learning Research}, 18(52):1--11, 2017.

\bibitem{shi2015extra}
W.~Shi, Q.~Ling, G.~Wu, and W.~Yin.
\newblock Extra: An exact first-order algorithm for decentralized consensus
  optimization.
\newblock {\em SIAM Journal on Optimization}, 25(2):944--966, 2015.

\bibitem{song2021optimal}
Z.~Song, L.~Shi, S.~Pu, and M.~Yan.
\newblock Optimal gradient tracking for decentralized optimization.
\newblock {\em arXiv preprint arXiv:2110.05282}, 2021.

\bibitem{song2021provably}
Z.~Song, L.~Shi, S.~Pu, and M.~Yan.
\newblock Provably accelerated decentralized gradient method over unbalanced
  directed graphs.
\newblock {\em arXiv preprint arXiv:2107.12065}, 2021.

\bibitem{spokoiny2012parametric}
V.~Spokoiny et~al.
\newblock Parametric estimation. finite sample theory.
\newblock {\em The Annals of Statistics}, 40(6):2877--2909, 2012.

\bibitem{stepanov2021one}
I.~Stepanov, A.~Voronov, A.~Beznosikov, and A.~Gasnikov.
\newblock One-point gradient-free methods for composite optimization with
  applications to distributed optimization.
\newblock {\em arXiv preprint arXiv:2107.05951}, 2021.

\bibitem{stich2018local}
S.~U. Stich.
\newblock Local sgd converges fast and communicates little.
\newblock {\em arXiv preprint arXiv:1805.09767}, 2018.

\bibitem{stich2018sparsified}
S.~U. Stich, J.-B. Cordonnier, and M.~Jaggi.
\newblock Sparsified sgd with memory.
\newblock In {\em Advances in Neural Information Processing Systems}, pages
  4447--4458, 2018.

\bibitem{stonyakin2019gradient}
F.~Stonyakin, D.~Dvinskikh, P.~Dvurechensky, A.~Kroshnin, O.~Kuznetsova,
  A.~Agafonov, A.~Gasnikov, A.~Tyurin, C.~A. Uribe, D.~Pasechnyuk, et~al.
\newblock Gradient methods for problems with inexact model of the objective.
\newblock {\em arXiv preprint arXiv:1902.09001}, 2019.

\bibitem{stonyakin2020inexact}
F.~Stonyakin, A.~Tyurin, A.~Gasnikov, P.~Dvurechensky, A.~Agafonov,
  D.~Dvinskikh, M.~Alkousa, D.~Pasechnyuk, S.~Artamonov, and V.~Piskunova.
\newblock Inexact model: A framework for optimization and variational
  inequalities.
\newblock {\em Optimization Methods and Software}, pages 1--47, 2021.

\bibitem{sun2019convergence}
Y.~Sun, A.~Daneshmand, and G.~Scutari.
\newblock Convergence rate of distributed optimization algorithms based on
  gradient tracking.
\newblock {\em arXiv preprint arXiv:1905.02637}, 2019.

\bibitem{sun2020convergence}
Y.~Sun, A.~Daneshmand, and G.~Scutari.
\newblock Distributed optimization based on gradient-tracking revisited:
  Enhancing convergence rate via surrogation.
\newblock {\em arXiv preprint arXiv:1905.02637}, 2020.

\bibitem{tang2020distributed}
Y.~Tang, J.~Zhang, and N.~Li.
\newblock Distributed zero-order algorithms for nonconvex multi-agent
  optimization.
\newblock {\em IEEE Transactions on Control of Network Systems}, 2020.

\bibitem{tian2021acceleration}
Y.~Tian, G.~Scutari, T.~Cao, and A.~Gasnikov.
\newblock Acceleration in distributed optimization under similarity.
\newblock {\em arXiv preprint arXiv:2110.12347}, 2021.

\bibitem{vladislav2021accelerated}
V.~Tominin, Y.~Tominin, E.~Borodich, D.~Kovalev, A.~Gasnikov, and
  P.~Dvurechensky.
\newblock On accelerated methods for saddle-point problems with composite
  structure.
\newblock {\em arXiv preprint arXiv:2103.09344}, 2021.

\bibitem{tsitsiklis1984problems}
J.~N. Tsitsiklis.
\newblock Problems in decentralized decision making and computation.
\newblock Technical report, Massachusetts Inst of Tech Cambridge Lab for
  Information and Decision Systems, 1984.

\bibitem{uribe2018distributed}
C.~A. Uribe, D.~Dvinskikh, P.~Dvurechensky, A.~Gasnikov, and A.~Nedi\'c.
\newblock Distributed computation of {W}asserstein barycenters over networks.
\newblock In {\em 2018 IEEE 57th Annual Conference on Decision and Control
  (CDC)}, 2018.
\newblock Accepted, arXiv:1803.02933.

\bibitem{uribe2017optimal}
C.~A. Uribe, S.~Lee, A.~Gasnikov, and A.~Nedi{\'c}.
\newblock Optimal algorithms for distributed optimization.
\newblock {\em arXiv preprint arXiv:1712.00232}, 2017.

\bibitem{uribe2020dual}
C.~A. Uribe, S.~Lee, A.~Gasnikov, and A.~Nedi{\'c}.
\newblock A dual approach for optimal algorithms in distributed optimization
  over networks.
\newblock {\em Optimization Methods and Software}, pages 1--40, 2020.

\bibitem{vaswani2019fast}
S.~Vaswani, F.~Bach, and M.~Schmidt.
\newblock Fast and faster convergence of sgd for over-parameterized models and
  an accelerated perceptron.
\newblock In {\em The 22nd International Conference on Artificial Intelligence
  and Statistics}, pages 1195--1204, 2019.

\bibitem{GT-book}
J.~von Neumann, O.~Morgenstern, and H.~Kuhn.
\newblock {\em Theory of Games and Economic Behavior (commemorative edition)}.
\newblock Princeton University Press, 2007.

\bibitem{wai2018multi}
H.-T. Wai, Z.~Yang, Z.~Wang, and M.~Hong.
\newblock Multi-agent reinforcement learning via double averaging primal-dual
  optimization.
\newblock {\em arXiv preprint arXiv:1806.00877}, 2018.

\bibitem{wen2017terngrad}
W.~Wen, C.~Xu, F.~Yan, C.~Wu, Y.~Wang, Y.~Chen, and H.~Li.
\newblock Terngrad: Ternary gradients to reduce communication in distributed
  deep learning.
\newblock In {\em Advances in Neural Information Processing Systems}, pages
  1509--1519, 2017.

\bibitem{woodworth2020minibatch}
B.~Woodworth, K.~K. Patel, and N.~Srebro.
\newblock Minibatch vs local sgd for heterogeneous distributed learning.
\newblock {\em arXiv preprint arXiv:2006.04735}, 2020.

\bibitem{woodworth2020local}
B.~Woodworth, K.~K. Patel, S.~U. Stich, Z.~Dai, B.~Bullins, H.~B. McMahan,
  O.~Shamir, and N.~Srebro.
\newblock Is local sgd better than minibatch sgd?
\newblock {\em arXiv preprint arXiv:2002.07839}, 2020.

\bibitem{xiao2004fast}
L.~Xiao and S.~Boyd.
\newblock Fast linear iterations for distributed averaging.
\newblock {\em Systems \& Control Letters}, 53(1):65--78, 2004.

\bibitem{xu2019accelerated}
J.~Xu, Y.~Tian, Y.~Sun, and G.~Scutari.
\newblock Accelerated primal-dual algorithms for distributed smooth convex
  optimization over networks.
\newblock {\em arXiv preprint arXiv:1910.10666}, 2019.

\bibitem{yang2020catalyst}
J.~Yang, S.~Zhang, N.~Kiyavash, and N.~He.
\newblock A catalyst framework for minimax optimization.
\newblock {\em Advances in Neural Information Processing Systems}, 2020.

\bibitem{ye2020multi}
H.~Ye, L.~Luo, Z.~Zhou, and T.~Zhang.
\newblock Multi-consensus decentralized accelerated gradient descent.
\newblock {\em arXiv preprint arXiv:2005.00797}, 2020.

\bibitem{ye2020decentralized}
H.~Ye, Z.~Zhou, L.~Luo, and T.~Zhang.
\newblock Decentralized accelerated proximal gradient descent.
\newblock {\em Advances in Neural Information Processing Systems}, 33, 2020.

\bibitem{yu2019linear}
H.~Yu, R.~Jin, and S.~Yang.
\newblock On the linear speedup analysis of communication efficient momentum
  sgd for distributed non-convex optimization.
\newblock {\em arXiv preprint arXiv:1905.03817}, 2019.

\bibitem{yuan2016convergence}
K.~Yuan, Q.~Ling, and W.~Yin.
\newblock On the convergence of decentralized gradient descent.
\newblock {\em SIAM Journal on Optimization}, 26(3):1835--1854, 2016.

\bibitem{zhou2018direct}
K.~Zhou.
\newblock Direct acceleration of saga using sampled negative momentum.
\newblock {\em arXiv preprint arXiv:1806.11048}, 2018.

\bibitem{zhou2018simple}
K.~Zhou, F.~Shang, and J.~Cheng.
\newblock A simple stochastic variance reduced algorithm with fast convergence
  rates.
\newblock {\em arXiv preprint arXiv:1806.11027}, 2018.

\end{thebibliography}

\end{document}